\documentclass{amsart}
\usepackage[utf8]{inputenc}
\usepackage{amsmath}
\usepackage{amsfonts}
\usepackage{amssymb}
\usepackage{amsthm}
\usepackage{mathtools}	
\usepackage{tikz}       
\usepackage{tikz-cd}	
\usepackage{hyperref}	
\usepackage{enumitem}
\usepackage{graphicx}

\DeclareMathOperator{\id}{id}
\DeclareMathOperator{\im}{im}

\DeclareMathOperator{\End}{End}

\newcommand{\F}{\ensuremath{\mathbb{F}} }
\newcommand{\N}{\ensuremath{\mathbb{N}} }
\newcommand{\Z}{\ensuremath{\mathbb{Z}} }

\newcommand{\GL}[2]{\ensuremath{\mathrm{GL}_{#1}(#2)}}

\newcommand{\gr}{\mathrm{gr}}
\newcommand{\CFK}{\mathit{CFK}}

\newcommand{\HF}{\mathit{HF}}
\newcommand{\CFL}{\mathit{CFL}}
\newcommand{\HFL}{\mathit{HFL}}

\newtheorem{theorem}{Theorem}[section]

\newtheorem{proposition}[theorem]{Proposition}

\newtheorem{lemma}[theorem]{Lemma}
\newtheorem{corollary}[theorem]{Corollary}

\theoremstyle{definition}
\newtheorem{definition}[theorem]{Definition}

\newtheorem{claim}[theorem]{Claim}
\theoremstyle{remark}
\newtheorem*{remark}{Remark}
\newtheorem{example}[theorem]{Example}

\usepackage{subcaption}

\usepackage{tikz}
\usetikzlibrary{arrows}
\usetikzlibrary{decorations.pathreplacing,calligraphy}

\newcommand\scalemath[2]{\scalebox{#1}{\mbox{\ensuremath{\displaystyle #2}}}}

\colorlet{dark green}{green!50!black}

\title[Link Floer homology as snake complexes and local systems]{Link Floer homology splits into snake complexes and local systems}
\author{David Popovi\'c}
\address{Department of Mathematics\\
  University of California, Los Angeles}
\email{dpopovic@math.ucla.edu}

\date{September 2023}

\begin{document}
\maketitle

\begin{abstract}
    We classify isomorphism and chain homotopy equivalence classes of finitely generated $\Z\oplus\Z$ graded free chain complexes over $\frac{\F[U,V]}{(UV)}$. As a corollary, we establish that every link Floer complex $\CFL(Y, L)$ over the ring $\frac{\F[U,V]}{(UV)}$ splits uniquely as a direct sum of snake complexes and local systems. This generalizes and extends the results of Petkova and Dai, Hom, Stoffregen, and Truong. We give the first example of an essentially infinite knot Floer-like complex, \emph{i.e.}, a complex satisfying all formal properties of link Floer complexes of knots in $S^3$ and whose chain homotopy equivalence class does not admit a representative of the form $C \otimes_{\F} \F[U,V]$. Finally, we also describe the first example of a knot Floer-like complex that does not admit a simultaneously vertically and horizontally simplified basis.
\end{abstract}

\section{Introduction}
Link Floer homology is an invariant of links in 3-manifolds. It was introduced by Ozsv\'ath and Szab\'o \cite{ozsvath2008holomorphicLinks} as a generalization of knot Floer homology \cite{ozsvath2004holomorphicKnots, rasmussen2003floer}, which is itself a refinement of an earlier theory of $3$-manifold invariants Heegaard Floer homology \cite{ozsvath2004holomorphic, ozsvath2004holomorphicSequel}. Since then, different variations of the link Floer homology package have been found to contain much information about the geometric properties of the link. For example, link Floer homology detects the genus of a knot \cite{OS2004holomorphicDisksAndGenusBounds}, its fiberedness \cite{ghiggini2008knot, ni2007knot, juhasz2008floer}, the Thurston norm of a link complement \cite{OS2008linkThurstonNorm} and has been used to obtain bounds on the unknotting number \cite{OS-HFKandUnknottingNumber} and slice genus \cite{ozsvath2003tau}. Link Floer homology was also used to construct several knot concordance invariants \cite{ozsvath2003tau, OS-HFKandRationalSurgeries, hom2014epsilon, hom2014nu+, OS-HFKandIntegerSurgeries, hendricks2017involutive, OSS2017upsilon, dai2021more}.

\vspace{1em}

Since the early days of knot and link Floer homology, much research has been dedicated to describing the structure of $\CFL(L)$ for certain families of knots and links. Complete descriptions have been obtained for $L$-space knots \cite{OS-HFKandIntegerSurgeries}, Floer homologically thin knots \cite{petkova2013cables}, and more recently for almost $L$-space knots \cite{binns2023cfk}, ribbon knots with fusion number one \cite{hom2020ribbon}, and other smaller families of knots. Working over the ring $\frac{\F_2[U,V]}{(UV)}$, Dai, Hom, Stoffregen, and Truong proved the first general structural result \cite{dai2021more}. They showed that any link Floer complex of a knot $K$ in $S^3$ splits as a direct sum $\CFL(S^3, K) \cong C(a_1, \dots, a_{2n}) \oplus T$ where $C(a_1, \dots, a_{2n})$ is a standard complex and $T$ is a chain complex with torsion homology. We generalize this classification substantially: from knots in $S^3$ to nullhomologous links in an arbitrary closed oriented $3$-manifold $Y$ and from $\F_2$ to an arbitrary field $\F$. Moreover, we extend the result by describing the internal structure of $T$.
\begin{theorem}\label{thm:classification of CFL up to iso}
Let $\F$ be a field and $\mathcal{R}_1 = \frac{\F[U,V]}{(UV)}$. Let $L \subset Y$ be a nullhomologous link in a closed oriented $3$-manifold $Y$ and let $\CFL_{\mathcal{R}_1}(Y,L)$ be its link Floer complex. Then
\[\CFL_{\mathcal{R}_1}(Y,L) \simeq S_1 \oplus \dots \oplus S_m \oplus L_1 \oplus \dots \oplus L_k\]
where $S_1, \dots, S_m$ are snake complexes (Def. \ref{def:snake complex}) and $L_1, \dots, L_k$ are local systems (Def. \ref{def:local system}). Moreover, the direct summands are unique up to permutation.
\end{theorem}
The following special case is worth noting.
\begin{corollary}\label{cor:thm for knots in S3}
Let $\mathcal{R}_1=\frac{\F_2[U,V]}{(UV)}$, $K \subset S^3$ be a knot and let $\CFL_{\mathcal{R}_1}(S^3,K)$ be its link Floer complex. Then
\[\CFL_{\mathcal{R}_1}(S^3,K) \simeq C(a_1, \dots, a_{2n}) \oplus L_1 \oplus \dots \oplus L_k\]
where $C(a_1, \dots, a_{2n})$ is a standard complex (Def. \ref{def:standard complex}) and $L_1, \dots, L_k$ are local systems. Moreover, the direct summands are unique up to permutation.
\end{corollary}

Local systems are a natural generalization of complexes admitting simultaneously vertically and horizontally simplified bases. The notion was inspired by \cite{hanselman2016bordered} and \cite{hanselman2022heegaard} where Hanselman, Rasmussen, and Watson represent $\CFL_{\mathcal{R}_1}(S^3, K)$ as immersed curves on a punctured torus decorated with local systems. Theorem \ref{thm:classification of CFL up to iso} is a generalized version of an analogous result in an algebraic setting. The proof we give is related to the techniques in their papers. Similar generalization has very recently been obtained in an independent work of Hanselman \cite{hanselman2023knot}.

\begin{figure}[t]
    \begin{tikzpicture}[scale=0.75]
        \def\a{0.2}
        \def\b{0.5}

        \draw[] (-1,0) node[]{$\CFL_{\mathcal{R}_1}(K) \simeq$};
        \draw[very thick] (1,1)--(2,1)--(2,-1)--(1,-1)--(1,0)--(3,0)--(3,-1) {};
        \draw[] (4,0) node[]{$\oplus$};
        \draw[very thick] (5,1)--(7,1)--(7,-1)--(6,-1)--(6,0)--(5,0)--(5,1) {};
        \draw[] (8,0) node[]{$\oplus$};
        \draw[very thick] (9,1)--(11,1)--(11,-1)--(9,-1)--(9,1) {};
        \draw[very thick] (9-\a,1+\a)--(11+\a,1+\a)--(11+\a,-1-\a)--(9-\a,-1-\a)--(9-\a,1+\a) {};
        \draw[very thick] (9+\a,1-\a)--(11-\a,1-\a)--(11-\a,-1+\a)--(9+\a,-1+\a)--(9+\a,1-\a) {};
        \draw[very thick] (11-\a,-1+\a)--(9,-1) {};
        \draw[very thick] (11+\a,-1-\a)--(9+\a,-1+\a) {};
        \draw[] (12,0) node[]{$\oplus$};
        \draw[very thick] (13,0.5)--(14,0.5)--(14,-0.5)--(13,-0.5)--(13,0.5) {};

        \filldraw[] (1,1) circle (2pt) node[anchor=north west]{};
        \filldraw[] (2,1) circle (2pt) node[anchor=north west]{};
        \filldraw[] (2,-1) circle (2pt) node[anchor=north west]{};
        \filldraw[] (1,-1) circle (2pt) node[anchor=north west]{};
        \filldraw[] (1,0) circle (2pt) node[anchor=north west]{};
        \filldraw[] (3,0) circle (2pt) node[anchor=north west]{};
        \filldraw[] (3,-1) circle (2pt) node[anchor=north west]{};
        \filldraw[] (5,1) circle (2pt) node[anchor=north west]{};
        \filldraw[] (7,1) circle (2pt) node[anchor=north west]{};
        \filldraw[] (7,-1) circle (2pt) node[anchor=north west]{};
        \filldraw[] (6,-1) circle (2pt) node[anchor=north west]{};
        \filldraw[] (6,0) circle (2pt) node[anchor=north west]{};
        \filldraw[] (5,0) circle (2pt) node[anchor=north west]{};
        \filldraw[] (5,1) circle (2pt) node[anchor=north west]{};
        \filldraw[] (9,1) circle (2pt) node[anchor=north west]{};
        \filldraw[] (11,1) circle (2pt) node[anchor=north west]{};
        \filldraw[] (11,-1) circle (2pt) node[anchor=north west]{};
        \filldraw[] (9,-1) circle (2pt) node[anchor=north west]{};
        \filldraw[] (9+\a,1-\a) circle (2pt) node[anchor=north west]{};
        \filldraw[] (11-\a,1-\a) circle (2pt) node[anchor=north west]{};
        \filldraw[] (11-\a,-1+\a) circle (2pt) node[anchor=north west]{};
        \filldraw[] (9+\a,-1+\a) circle (2pt) node[anchor=north west]{};
        \filldraw[] (9-\a,1+\a) circle (2pt) node[anchor=north west]{};
        \filldraw[] (11+\a,1+\a) circle (2pt) node[anchor=north west]{};
        \filldraw[] (11+\a,-1-\a) circle (2pt) node[anchor=north west]{};
        \filldraw[] (9-\a,-1-\a) circle (2pt) node[anchor=north west]{};
        \filldraw[] (13,0.5) circle (2pt) node[anchor=north west]{};
        \filldraw[] (14,0.5) circle (2pt) node[anchor=north west]{};
        \filldraw[] (14,-0.5) circle (2pt) node[anchor=north west]{};
        \filldraw[] (13,-0.5) circle (2pt) node[anchor=north west]{};

        \draw [very thick, decorate, decoration = {calligraphic brace, mirror}] (1-\a,-1-\b)--(3+\a,-1-\b);
        \draw [very thick, decorate, decoration = {calligraphic brace, mirror}] (5-\a,-1-\b)--(14+\a,-1-\b);

        \draw[] (2,-2) node[]{standard complex};
        \draw[] (9.5,-2) node[]{local systems};
        
    \end{tikzpicture}
    \caption{Schematic depiction of Corollary \ref{cor:thm for knots in S3}. The link Floer complex of a knot $K \subset S^3$ over $\F_2$ splits into a standard complex and local systems.}
    \label{fig:cfk splitting}
\end{figure}
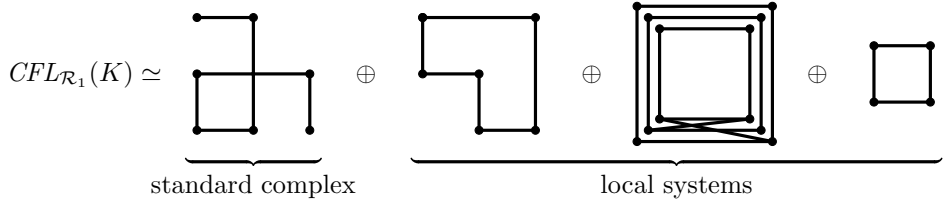

\vspace{1em}

Theorem \ref{thm:classification of CFL up to iso} also generalizes a result of Petkova \cite[Lemma 7]{petkova2013cables}, which states that $\CFL_{\mathcal{R}_1}(S^3, K)$ of any Floer homologically thin knot splits as a direct sum of a standard complex and a finite number of square complexes. This is a large class of knots, containing as a proper subset all quasi-alternating knots, so her result inspires some hope that we might be able to impose even stronger regularity conditions on our complexes $L_i$, at least for knots in $S^3$. However, we argue that a significant strengthening the conclusions of Theorem \ref{thm:classification of CFL up to iso} is unlikely to be possible. To explain why we think so, note that there are two main structural properties the chain complexes $L_i$ would ideally possess.
\begin{enumerate}
    \item $L_i = C \otimes_{\F} \mathcal{R}_1$ where $C$ is finitely generated as an $\F$-module, and
    \item $L_i$ admits a simultaneously vertically and horizontally simplified basis.
\end{enumerate}
We define knot Floer-like complexes (Def. \ref{def:knot Floer-like}) as algebraic analogues to knot Floer complexes of knots in $S^3$ and provide first examples to show that both of the above wishes are false algebraically, not just up to isomorphism, but even up to chain homotopy equivalence.
\begin{theorem}\label{thm:example D}
The knot Floer-like complex $D$ (Fig. \ref{fig:David's infinite example}) has the following property. It is not chain homotopy equivalent to any knot Floer-like complex of the form $C \otimes_{\F} \F[U,V]$ where $C$ is finitely generated as an $\F$-module.
\end{theorem}
\begin{theorem}\label{thm:example P}
The knot Floer-like complex $P$ (Fig. \ref{fig:David's basis example}) has the following property. It is not chain homotopy equivalent to any knot Floer-like complex that admits a simultaneously horizontally and vertically simplified basis.
\end{theorem}
The author is not aware of any purely algebraic restrictions that could be used to rule out the potential existence of knots in $S^3$ with knot Floer complexes $D$ and $P$. 

\subsection*{Organization}
In Section \ref{section:background section}, we recall algebraic properties of link Floer homology and introduce our conventions for the pictures we draw throughout the paper. In Section \ref{section:local systems}, we define the notion of a local system and classify them. Section \ref{section:classifying algebraic complexes} contains the algebraic results of the paper: it introduces two-story complexes as our main technical tool for keeping track of two bases simultaneously and describes the algorithm used to simplify finitely generated free chain complexes over $\mathcal{R}_1$ until they split as a direct sum. In Section \ref{section:topological applications}, we use this classification to prove Theorem \ref{thm:classification of CFL up to iso} and Corollary \ref{cor:thm for knots in S3} and mention other topological applications of our result. In Section \ref{section:immersed curves proof}, we indicate the relationship between our proof and its inspiration \cite{hanselman2016bordered}. Section \ref{section:examples} contains novel examples of knot Floer-like complexes.

\subsection*{Acknowledgements} I would like to thank Sucharit Sarkar for his advice, many helpful discussions and a careful reading of the preprint. I would also like to thank Linh Truong for pointing out a mistake in an earlier version of the preprint. This work was partially supported by NSF Grant DMS-1905717.

\section{Background on link Floer complexes}\label{section:background section}
In this section, we establish our notation, conventions, and recall prerequisites on link Floer homology. We work over an arbitrary field $\F$ throughout most of the paper until we specialize to the case $\F=\F_2$ in Sections \ref{section:immersed curves proof} and \ref{section:examples}.

\subsection{Complexes}
We begin by reviewing formal aspects of link Floer homology. We use the version of link Floer homology with a single $U$ and a single $V$ variable as described in \cite{zemke2019connected}. The other main reference for the case of knots is \cite{dai2021more} and the original paper containing the proofs of the properties is \cite{ozsvath2008holomorphicLinks}. 

\vspace{1em}

Let $Y$ be an oriented closed $3$-manifold and let $L \subset Y$ be an $l$-component link all of whose components represent the trivial homology class in $H_1(Y; \Z)$. The link Floer complexes $\CFL_{\F[U,V]}(Y, L)$ have the following algebraic properties.
\begin{enumerate}
    \item \textbf{Gradings}: $\CFL_{\F[U,V]}(Y, L)$ is a finitely and freely generated bigraded chain complex over $\F[U,V]$ with a differential $\partial$ and bigrading $\gr=(\gr_U, \gr_V)$ satisfying $\gr(U)=(-2,0)$, $\gr(V)=(0,-2)$, and $\gr(\partial)=(-1,-1)$.
    \item \textbf{Symmetry}: There is a chain homotopy equivalence
    $$\CFL_{\F[U,V]}(Y, L)\simeq \overline{\CFL_{\F[U,V]}}(Y, L)$$
    where the overline denotes the complex in which the roles of $U$ and $V$ are exchanged and $\gr_U$ and $\gr_V$ are switched. 
    \item \textbf{Homology}: There is an isomorphism
    $H_*(\CFL_{\F[U,V]}(Y, L)/U) \cong \HFL^-(Y,L)$ as bigraded chain complexes over $\F[V]$. As a special case when $Y=S^3$ we have
    $$\frac{H_*(\CFL_{\F[U,V]}(S^3, L)/U)}{V\text{-torsion}} \cong \F[V]^{2^{l-1}}$$
    where $V$-torsion denotes the torsion subcomplex. The isomorphism stems from the fact that we are considering $l$-component links in $S^3$, which correspond to nullhomologous knots in $(S^2\times S^1)^{\# l-1}$ via the knotification construction of \cite{ozsvath2004holomorphicKnots}. Setting $U=0$, $V=1$ and taking homology recovers $\widehat{\HF}((S^2\times S^1)^{\# l-1})$ which is isomorphic to $\F[V]^{2^{l-1}}$.
    There is a similar isomorphism obtained by exchanging the roles of $U$ and $V$.
\end{enumerate}
With these formal properties of $\CFL_{\F[U,V]}(Y, L)$ in mind, we establish an algebraic framework for dealing with such complexes.

\vspace{1em}

Let the ring $\F[U,V]$ be equipped with a $\Z\oplus\Z$ grading $\gr=(\gr_U, \gr_V)$, where $\gr(U) = (-2, 0)$ and $\gr(V)=(0, -2)$. Any quotient of $\F[U,V]$ by an ideal generated by homogeneous elements inherits the grading $\gr$ in a natural way. The cases we consider most often are when the quotient is $\mathcal{R}_1 = \frac{\F[U,V]}{(UV)}$ and when it is $\F[U,V]$ itself. Any module $C$ over these graded rings is automatically considered $\Z\oplus\Z$ graded and any endomorphism of $C$ that is denoted by $\partial$ is required to satisfy $\partial^2=0$ and have degree $(-1, -1)$.

\vspace{1em}

We have introduced the necessary terminology for dealing with the `Gradings' part of the list. Encoding the `Symmetry' and `Homology' properties of link Floer complexes into a purely algebraic language is more straightforward.
\begin{definition}
    Let $\mathcal{R}$ be either $\F[U,V]$ or $\mathcal{R}_1$. A chain complex $(C, \partial)$ over $\mathcal{R}$ is \emph{symmetric} if $C \simeq \overline{C}$ where the overline denotes the complex in which the roles of $U$ and $V$ are interchanged and $\gr_U$ and $\gr_V$ are switched.
\end{definition}
For any ideal $I \vartriangleleft\F[U,V]$ generated by homogeneous elements, we denote by $C/I  = C \otimes_{\F[U,V]} \F[U,V]/I$ and treat is as a $\Z\oplus\Z$ graded $\F[U,V]/I$-module. The cases we encounter most often are $C/U = C \otimes_{\F[U,V]} \F[V]$ and $C/V = C \otimes_{\F[U,V]} \F[U]$.
\begin{definition}
    Let $\mathcal{R}$ be either $\F[U,V]$ or $\mathcal{R}_1$. A chain complex $(C, \partial)$ over $\mathcal{R}$ has
    \begin{itemize}[label=$-$]
        \item \emph{torsion homology} if $\frac{H_*(C/U)}{V-\text{torsion}} \cong 0$ and $\frac{H_*(C/V)}{U-\text{torsion}} \cong 0$,
        \item the \emph{homology of a knot in $S^3$} if $\frac{H_*(C/U)}{V-\text{torsion}} \cong \F[V]$ and $\frac{H_*(C/V)}{U-\text{torsion}} \cong \F[U]$, where the element $x$ generating $\F[V]$ satisfies $\gr_U(x)=0$ and the element $y$ generating $\F[U]$ satisfies $\gr_V(y)=0$, and
        \item the \emph{homology of a link in $S^3$} if for some $l\in\N$ we have that $\frac{H_*(C/U)}{V-\text{torsion}} \cong \F[V]^{2^{l-1}}$
        and $\frac{H_*(C/V)}{U-\text{torsion}} \cong \F[U]^{2^{l-1}}$.
    \end{itemize}
\end{definition}
We now define knot Floer-like and link Floer-like complexes as algebraic counterparts to knot and link Floer complexes.
\begin{definition}\label{def:knot Floer-like}
    A \emph{knot Floer-like complex} $C$ is a symmetric finitely generated free chain complex over $\F[U,V]$ with the homology of a knot in $S^3$. Similarly, a \emph{link Floer-like complex} $C$ is a symmetric finitely generated free chain complex over $\F[U,V]$ with the homology of a link.
\end{definition}
Finding bases with respect to which the chain complexes under consideration take a particularly nice form will be of interest to us in most of the paper.
\begin{definition}
    Let $\mathcal{R}$ be either $\F[U,V]$ or $\mathcal{R}_1$ and let $(C, \partial)$ be a finitely generated free chain complex over $\mathcal{R}$. A basis $\{x_1, \dots, x_n\}$ of $C$ is \emph{vertically simplified} if for each $i \in \{1, \dots, n\}$ we have exactly one of the possibilities in the complex $C/U$:
    \begin{enumerate}
        \item $\partial x_i = V^a x_j$ for some $j \in \{1, \dots, n\}$ and $a \in \N_0$, or
        \item $\partial x_i = 0$ and there exist at most one pair $j \in \{1, \dots, n\}$ and $a \in \N_0$ such that $\partial x_j = V^ax_i$.
    \end{enumerate}
    In the former case we say that there is a \emph{vertical arrow of length} $a$ from $x_i$ to $x_j$.
\end{definition}
The conditions in the above definition ensure that there is at most one vertical arrow starting at each element of a vertically simplified basis and at most one vertical arrow ending at each element of a vertically simplified basis. Moreover, there cannot exist vertical arrows both starting and ending at the same basis element since $\partial^2=0$. There is also a similar notion of a horizontally simplified basis to which analogous comments apply.
\begin{definition}
Let $\mathcal{R}$ be either $\F[U,V]$ or $\mathcal{R}_1$ and let $(C, \partial)$ be a finitely generated free chain complex over $\F[U,V]$. A basis $\{x_1, \dots, x_n\}$ of $C$ is \emph{horizontally simplified} if for each $i \in \{1, \dots, n\}$ we have exactly one of the possibilities in the complex $C/V$:
\begin{enumerate}
        \item $\partial x_i = \lambda U^a x_j$ for some $j \in \{1, \dots, n\}$, nonzero $\lambda \in \F$ and $a \in \N_0$, or
        \item $\partial x_i = 0$ and there exist at most one triple $j \in \{1, \dots, n\}$, nonzero $\lambda \in\F$ and $a \in \N_0$ such that $\partial x_j = \lambda U^ax_i$.
    \end{enumerate}
    In the former case we say that there is a \emph{horizontal arrow of length} $a$ from $x_i$ to $x_j$.
\end{definition}
A basis that is simultaneously vertically and horizontally simplified is called \emph{simplified}.

\vspace{1em}

Note that an arrow of length $0$ is both a horizontal and a vertical arrow. This makes such arrows special and often unpleasant to handle. The existing literature usually works with chain complexes up to chain homotopy equivalence or up to local equivalence and finds a way to dismiss them in these settings; see for example \cite[Proposition 11.57]{lipshitz2018bordered} or \cite[Lemma 3.8]{dai2021more}. We also frequently consider chain complexes up to isomorphism and therefore need a slightly stronger result whose proof we include for completeness.

Note that the number of arrows of certain length in a chain complex $(C, \partial)$ is not an isomorphism invariant since it depends on the choice of a distinguished basis $B$. However, we prove that having an arrow of length $0$ is a basis-invariant property.
\begin{lemma}\label{lemma:having an arrow of length 0 is a basis invariant}
   Let $\mathcal{R}$ be either $\F[U,V]$ or $\mathcal{R}_1$ and let $(C, \partial)$ be a finitely generated free chain complex over $\mathcal{R}$. Let $B_1$ and $B_2$ be two bases of $C$. Then $(C, \partial, B_1)$ and $(C, \partial, B_2)$ either both have or do not have an arrow of length $0$.
\end{lemma}
\begin{proof}
    Suppose $(C, \partial, B_2)$ has no arrows of length $0$. Let $f: (C, \partial, B_1) \to (C, \partial, B_2)$ be an isomorphism. In particular, $f$ is an invertible chain map and so $f\partial = \partial f$ and $\partial = f^{-1}\partial f$. Consider the mod $(U,V)$ reduction of $C$, which deletes all arrows of positive length. Since there are no arrows of length $0$ in $(C, \partial, B_2)$, it follows that the mod $(U,V)$ reduction of $\partial$ with respect to $B_2$ is $0$. We have $\partial = f^{-1}\partial f$ so the same is true for $B_1$. By lifting the result back to complexes over $\mathcal{R}_1$, we establish that $(C,\partial, B_1)$ also has no arrows of length $0$, as required.
\end{proof}
Having shown that the presence of arrows of length $0$ is a basis invariant property, we can isolate them in the following manner.
\begin{definition}
Let $\mathcal{R}$ be either $\F[U,V]$ or $\mathcal{R}_1$. A zero complex is a free chain complex over $\mathcal{R}$ of rank two with generators $x, y$ such that $\partial x=y$ and $\partial y=0$.
\end{definition}
\begin{lemma}\label{lemma:split away a single zero complex}
Let $\mathcal{R}$ be either $\F[U,V]$ or $\mathcal{R}_1$ and let $(C, \partial)$ be a finitely generated free chain complex over $\mathcal{R}$. If $C$ has an arrow of length $0$, then $C \cong D \oplus Z$ where $D$ is some finitely generated free chain complex over $\mathcal{R}$ and $Z$ is a zero complex. It follows that $C \simeq D$.
\end{lemma}
\begin{proof}
Let $\{x_1, \dots, x_n\}$ be a basis of $C$ and without loss of generality let there be an arrow of length $0$ from $x_1$ to $x_2$. For $i \in \{3, \dots, n\}$, let $p_i \in \mathcal{R}$ be the coefficient of $x_2$ in $\partial x_i$. With respect to the new basis $\{x_1, \partial x_1, x_3-p_3x_1, \dots, x_n-p_nx_1\}$, $C$ splits into the direct sum of $D = \mathcal{R}\langle x_3-p_3x_1, \dots, x_n-p_nx_1 \rangle$ and the zero complex $Z= \mathcal{R}\langle x_1, \partial x_1 \rangle$ as required. Since $Z \simeq 0$, we have $C \simeq D$.
\end{proof}
\begin{corollary}\label{cor:we split away zero complexes}
Let $\mathcal{R}$ be either $\F[U,V]$ or $\mathcal{R}_1$ and let $(C, \partial)$ be a finitely generated free chain complex over $\mathcal{R}$. Then $C \cong D \oplus Z_1 \oplus \dots \oplus Z_k$ where $D$ is some finitely generated free chain complex over $\mathcal{R}$ with no arrows of length $0$ and $Z_i$ are zero complexes. It follows that $C \simeq D$.
\end{corollary}
\begin{proof}
Keep splitting away the $2$-step complexes with a single arrow of length $0$ from $C$ using Lemma \ref{lemma:split away a single zero complex} for as long as arrows of length $0$ exist. This is a finite process since $C$ is finitely generated.
\end{proof}
We now return to the discussion of vertically and horizontally simplified bases. Corollary \ref{cor:we split away zero complexes} shows that one might without much loss of generality assume that $C$ has no arrows of length $0$. From now on, this is the assumption we will liberally adopt. In such a setting, again generalized to an arbitrary field $\F$, Proposition 11.57 of \cite{lipshitz2018bordered} shows that every finitely generated free chain complex over $\F[U,V]$ or $\mathcal{R}_1$ admits a vertically simplified basis and a horizontally simplified basis. Neither of them is unique in general.

However, simplified bases are markedly rarer. Even if one imposes significant additional restrictions on the complexes, there is no guarantee that a simplified basis can be found. In fact, there exist knot Floer-like chain complexes that do not admit simplified bases. See Example $P$ in Section \ref{section:examples} for a more in-depth discussion of this phenomenon.

\vspace{1em}

An especially simple class of complexes admitting simplified bases are snake complexes, a mild generalization of standard complexes introduced in \cite{dai2021more}.
\begin{definition}\label{def:hor or vert snake complex}
Let $m \in 2\N_0+1$ and let $b_1, \dots, b_{m}$ be a sequence of nonzero integers. A \emph{horizontal snake complex} $S_h(b_1, \dots, b_{m})$ is a free chain complex over $\mathcal{R}_1$ with a distinguished basis $B=\{x_0, \dots, x_m\}$ and a differential $\partial$ defined as follows. For each odd $i$, there is a horizontal arrow of length $|b_i|$ connecting $x_i$ and $x_{i-1}$. For each even $i$, there is a vertical arrow of length $|b_i|$ connecting $x_i$ and $x_{i-1}$. The direction of the arrow is determined by the sign of $b_i$, as follows. If $b_i > 0$, then the arrow goes from $x_i$ to $x_{i-1}$, and if $b_i < 0$, then the arrow goes from $x_{i-1}$ to $x_i$. The $\Z\oplus\Z$ grading on $S_h(b_1, \dots, b_m)$ is allowed to be arbitrary.

Similarly, let $m \in 2\N_0$ and let $b_0, \dots, b_m$ be a sequence of nonzero integers. A \emph{vertical snake complex} $S_v(b_0, \dots, b_m)$ is a free chain complex over $\mathcal{R}_1$ with a distinguished basis $B=\{x_{-1}, \dots, x_m\}$ and a differential $\partial$ defined exactly as above.
\end{definition}
\begin{definition}\label{def:standard complex}
Let $n \in 2\N_0$ and let $a_1, \dots, a_n$ be a sequence of nonzero integers. The \emph{standard complex} $C(a_1, \dots, a_n)$ is a free chain complex over $\mathcal{R}_1$ with a distinguished basis $B=\{x_{0}, \dots, x_n\}$ and a differential $\partial$ defined exactly as above. That is, for each odd $i$, there is a horizontal arrow of length $|a_i|$ connecting $x_i$ and $x_{i-1}$. For each even $i$, there is a vertical arrow of length $|a_i|$ connecting $x_i$ and $x_{i-1}$. The direction of the arrow is determined by the sign of $a_i$, as follows. If $a_i > 0$, then the arrow goes from $x_i$ to $x_{i-1}$, and if $a_i < 0$, then the arrow goes from $x_{i-1}$ to $x_i$. The $\Z\oplus\Z$ grading on $C(a_1, \dots, a_n)$ is uniquely determined by the condition $\gr_U(x_0)=0$.
\end{definition}
\begin{definition}\label{def:snake complex}
    A \emph{snake complex} refers to either a standard complex, a horizontal snake complex, or a vertical snake complex.
\end{definition}
For internalizing the notions of different flavours of snake complexes, Figure \ref{fig:example of a standard complex} is likely to be a lot more illuminating than the preceding definitions. Horizontal snake complexes start and end with horizontal arrows, vertical snake complexes start and end with vertical arrows and standard complexes start with a horizontal and end with a vertical arrow.

\vspace{1em}

Observe that there are exactly two distinct sequences describing each horizontal and vertical snake complex, depending on which end one starts with. The higher one, with respect to the order in Definition \ref{def:DHST order}, is called the \emph{shape}. On the other hand, standard complexes are always described starting from the horizontal end and thus correspond to a unique sequence.
\begin{figure}[t]
    \begin{subfigure}[b]{0.40\textwidth}
        \centering
        \begin{tikzpicture}[scale=1.0]
            \draw[step=1.0,gray,thin] (-0.5,1.5) grid (2.5,4.5);
            \draw[black, very thick] (0,2) -- (1,2) -- (1,4) -- (2,4);
            \filldraw[black] (0,2) circle (2pt) node[anchor=north west]{$x_0$};
            \filldraw[black] (1,2) circle (2pt) node[anchor=north west]{$x_1$};
            \filldraw[black] (1,4) circle (2pt) node[anchor=north west]{$x_2$};
            \filldraw[black] (2,4) circle (2pt) node[anchor=north west]{$x_3$};
        \end{tikzpicture}
        \caption{}
        \label{fig:example of a horizontal snake cx}
    \end{subfigure}
    \begin{subfigure}[b]{0.40\textwidth}
        \centering
        \begin{tikzpicture}[scale=1.0]
            \draw[step=1.0,gray,thin] (-0.5,1.5) grid (2.5,4.5);
            \draw[black, very thick] (0,2) -- (0,3) -- (2,3) -- (2,4);
            \filldraw[black] (0,2) circle (2pt) node[anchor=north west]{$x_0$};
            \filldraw[black] (0,3) circle (2pt) node[anchor=north west]{$x_1$};
            \filldraw[black] (2,3) circle (2pt) node[anchor=north west]{$x_2$};
            \filldraw[black] (2,4) circle (2pt) node[anchor=north west]{$x_3$};
        \end{tikzpicture}
        \caption{}
        \label{fig:example of a vertical snake cx}
    \end{subfigure}
    \begin{subfigure}[b]{0.40\textwidth}
        \centering
        \begin{tikzpicture}[scale=1.0]
            \draw[step=1.0,gray,thin] (-0.5,1.5) grid (4.5,4.5);
            \draw[black, very thick] (0,4) -- (2,4) -- (2,2) -- (1,2) -- (1,3) -- (4,3) -- (4,2);
            \filldraw[black] (0,4) circle (2pt) node[anchor=north west]{$x_0$};
            \filldraw[black] (2,4) circle (2pt) node[anchor=north west]{$x_1$};
            \filldraw[black] (2,2) circle (2pt) node[anchor=north west]{$x_2$};
            \filldraw[black] (1,2) circle (2pt) node[anchor=north west]{$x_3$};
            \filldraw[black] (1,3) circle (2pt) node[anchor=north west]{$x_4$};
            \filldraw[black] (4,3) circle (2pt) node[anchor=north west]{$x_5$};
            \filldraw[black] (4,2) circle (2pt) node[anchor=north west]{$x_6$};
        \end{tikzpicture}
        \caption{}
        \label{fig:example of a standard cx}
    \end{subfigure}
    \caption{A horizontal snake complex $S_h(1, 2, 1)$ is shown in ($\textsc{a}$), a vertical snake complex $S_v(1,2,1)$ in ($\textsc{b}$), and the standard complex $C(2, -2, -1, 1, 3, -1)$ in ($\textsc{c}$).}
    \label{fig:example of a standard complex}
\end{figure}
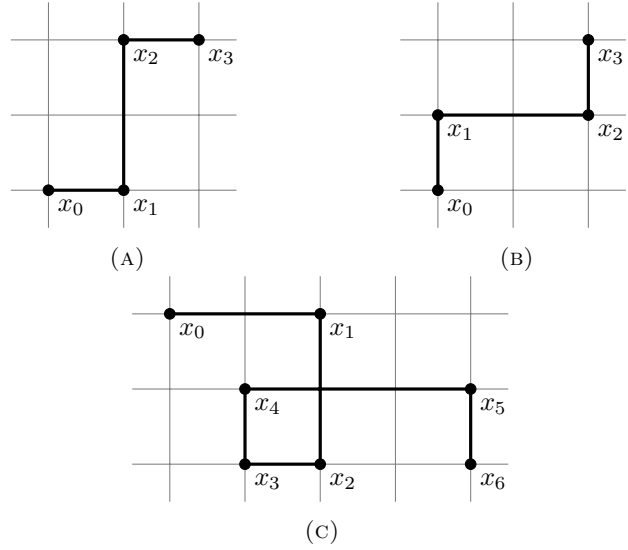

\vspace{1em}

Standard complexes are interesting primarily because they classify local equivalence types of finitely generated free chain complexes over $\mathcal{R}_1$ with the homology of a knot in $S^3$.
\begin{theorem}[\cite{dai2021more}, Corollary 6.2]\label{classification}
Let $C$ be a finitely generated free chain complex over $\mathcal{R}_1$ with the homology of a knot in $S^3$. Then there exists a unique standard complex $C(a_1, \dots, a_{2n})$ such that $C$ splits as a direct sum $C \cong C(a_1, \dots, a_{2n}) \oplus T$ for some chain complex $T$ over $\mathcal{R}_1$.
\end{theorem}
Note that the exact phrasing of the theorem in \cite{dai2021more} is slightly different. In particular, the authors only claim $C \simeq C(a_1, \dots, a_{2n}) \oplus T$, but they prove the above stronger version.
\begin{remark}
We emphasize that the splitting in Theorem \ref{classification} only works over $\mathcal{R}_1$ and, indeed, is the main reason for working over this quotient ring as opposed to $\F[U,V]$.
\end{remark}
The main result of this paper refines the above classification by characterizing the structure of the complex $T$. We also remove the requirement that $C$ have the homology of a knot in $S^3$ and we work over an arbitrary field as opposed to $\F_2$. A topological ramification of this fact is that our results can also be used in the more general case of links and link Floer homology and can potentially be used to detect torsion in link Floer homology.

\subsection{Pictures}\label{subsection:pictures}
In its classical setting, knot Floer homology has been studied as a chain complex $\CFK^-(K)$ over $\F[U]$ \cite{ozsvath2004holomorphicKnots}. To depict chain complexes in a geometric manner, their generators were depicted on a two-dimensional lattice grid in which the $x$-axis represents the $U^{-1}$ power and the $y$-axis represents the Alexander grading $A$. In recent years, it has gradually been realized that the full invariant $\CFL(S^3, K)$ over $\F[U,V]$ or a related ring is even more useful and a new way of drawing the complexes is needed. We introduce our conventions for drawing the complexes in this subsection. However, the pictures we obtain are equivalent to the ones in the existing literature. To be precise, the chain complex generators might have different decorations by monomials, but the shapes in our pictures are the same as the shapes obtained by drawing them classically or as in \cite{dai2021more}.

\vspace{1em}

Let us describe our pictures. The link Floer generators over $\F$ are drawn on a $2$-dimensional lattice where the $x$-coordinate denotes the $U^{-1}$ power and the $y$-coordinate denotes the $V^{-1}$ power. The arrows connecting the generators are decorated with elements of $\F$ and indicate the differential $\partial$, as explained below. Since $\gr(\partial)=(-1, -1)$, the Alexander grading $A=\frac{1}{2}(\gr_U-\gr_V)$ is preserved by $\partial$ and $\CFL_{\F[U,V]}(K)$ splits as an $\F[UV]$-module (but \emph{not} as an $\F[U,V]$-module) as a direct sum of summands with a constant Alexander grading. Since the pictures corresponding to different summands are all translates of each other, we always only draw the part in the Alexander grading $A=0$ and are intentionally ambiguous about the origin of the plane. Note that even after this restriction, the same letter will appear in infinitely many places. For example, in Figure \ref{fig:infinite CFK(figure 8 knot)}, the generator $d$ is drawn three times in positions $(0,-1)$, $(-1,-2)$, and $(-2, -3)$, and the three locations denote the elements $Vd$, $UV^2d$, and $U^2V^3d$. The horizontal arrow between $d$ and $e$ is decorated with a $2\in\F_3$ which indicates that $\partial Vd = 2UVe$ or equivalently that $\partial d = 2Ue$. An absence of a decoration means the same as a decoration by $1\in\F$.

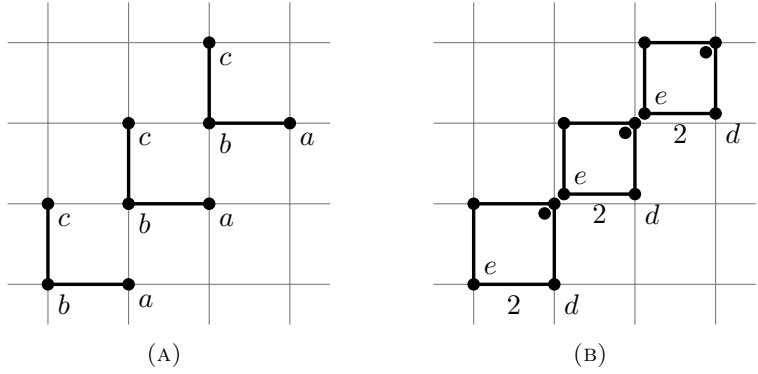
\begin{figure}[t]
    \begin{subfigure}[b]{0.35\textwidth}
        \centering
        \resizebox{\linewidth}{!}{
        \begin{tikzpicture}[scale=1.0]
            \draw[step=1.0,gray,thin] (-0.5,-0.5) grid (3.5,3.5);
            \draw[black, very thick] (1,0) -- (0,0) -- (0,1);
            \draw[black, very thick] (2,1) -- (1,1) -- (1,2);
            \draw[black, very thick] (3,2) -- (2,2) -- (2,3);
            
            \filldraw[black] (0,0) circle (2pt) node[anchor=north west]{$b$};
            \filldraw[black] (1,0) circle (2pt) node[anchor=north west]{$a$};
            \filldraw[black] (0,1) circle (2pt) node[anchor=north west]{$c$};

            \filldraw[black] (1,1) circle (2pt) node[anchor=north west]{$b$};
            \filldraw[black] (2,1) circle (2pt) node[anchor=north west]{$a$};
            \filldraw[black] (1,2) circle (2pt) node[anchor=north west]{$c$};

            \filldraw[black] (2,2) circle (2pt) node[anchor=north west]{$b$};
            \filldraw[black] (3,2) circle (2pt) node[anchor=north west]{$a$};
            \filldraw[black] (2,3) circle (2pt) node[anchor=north west]{$c$};
        \end{tikzpicture}}
        \caption{}
        \label{fig:infinite CFK(trefoil)}
    \end{subfigure}
    \hspace{1cm}
    \begin{subfigure}[b]{0.35\textwidth}
        \centering
        \resizebox{\linewidth}{!}{
        \begin{tikzpicture}[scale=1.0]
            \def\a{0.12}
            \def\b{0.17}
            \draw[step=1.0,gray,thin] (-0.5,-0.5) grid (3.5,3.5);
            \draw[black, very thick] (1,0) -- node[below]{$2$} (0,0) -- (0,1) -- (1,1) -- (1,0);
            \draw[black, very thick] (2,1+\a) -- node[below]{$2$} (1+\a,1+\a) -- (1+\a,2) -- (2,2) -- (2,1+\a);
            \draw[black, very thick] (3,2+\a) -- node[below]{$2$} (2+\a,2+\a) -- (2+\a,3) -- (3,3) -- (3,2+\a);
            
            \filldraw[black] (0,0) circle (2pt) node[anchor=south west]{$e$};
            \filldraw[black] (1,0) circle (2pt) node[anchor=north west]{$d$};
            \filldraw[black] (0,1) circle (2pt) node[anchor=north west]{};
            \filldraw[black] (1,1) circle (2pt) node[anchor=north west]{};
            \filldraw[black] (1-\a, 1-\a) circle (2pt) node[anchor=north east]{};

            \filldraw[black] (1+\a,1+\a) circle (2pt) node[anchor=south west]{$e$};
            \filldraw[black] (2,1+\a) circle (2pt) node[anchor=north west]{$d$};
            \filldraw[black] (1+\a,2) circle (2pt) node[anchor=north west]{};
            \filldraw[black] (2,2) circle (2pt) node[anchor=north west]{};
            \filldraw[black] (2-\a,2-\a) circle (2pt) node[anchor=north east]{};

            \filldraw[black] (2+\a,2+\a) circle (2pt) node[anchor=south west]{$e$};
            \filldraw[black] (3,2+\a) circle (2pt) node[anchor=north west]{$d$};
            \filldraw[black] (2+\a,3) circle (2pt) node[anchor=north west]{};
            \filldraw[black] (3,3) circle (2pt) node[anchor=north west]{};
            \filldraw[black] (3-\a,3-\a) circle (2pt) node[anchor=north east]{};
        \end{tikzpicture}}
        \caption{}
        \label{fig:infinite CFK(figure 8 knot)}
    \end{subfigure}
    \caption{The knot Floer complexes of the positive trefoil over $\F_2$ ($\textsc{a}$) and the figure-eight knot over $\F_3$ ($\textsc{b}$) with some generators labelled.} 
    \label{fig:infinite CFK}
\end{figure}

The knot Floer complexes of the positive trefoil over $\F_2$ and the figure-eight knot over $\F_3$ drawn in Figure \ref{fig:infinite CFK} consist of a finite piece and $\Z$ many of its translations. As such, it is sufficient to only draw the finite piece, which by itself uniquely specifies the chain complex. The corresponding finite pictures of the positive trefoil and the figure-eight knot are shown in Figure \ref{fig:finite CFK}.
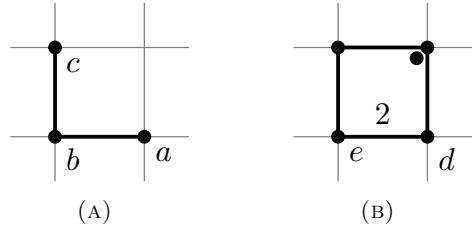
\begin{figure}[t]
    \begin{subfigure}[b]{0.20\textwidth}
        \centering
        \resizebox{\linewidth}{!}{
        \begin{tikzpicture}[scale=1.0]
            \draw[step=1.0,gray,thin] (-0.5,-0.5) grid (1.5,1.5);
            \draw[black, very thick] (1,0) -- (0,0) -- (0,1);
            \filldraw[black] (0,0) circle (2pt) node[anchor=north west]{$b$};
            \filldraw[black] (1,0) circle (2pt) node[anchor=north west]{$a$};
            \filldraw[black] (0,1) circle (2pt) node[anchor=north west]{$c$};
        \end{tikzpicture}}
        \caption{}
        \label{fig:CFK(trefoil)}
    \end{subfigure}
    \hspace{1cm}
    \begin{subfigure}[b]{0.20\textwidth}
        \centering
        \resizebox{\linewidth}{!}{
        \begin{tikzpicture}[scale=1.0]
            \draw[step=1.0,gray,thin] (-0.5,-0.5) grid (1.5,1.5);
            \def\a{0.12}
            \draw[black, very thick] (1,0) -- node[above]{$2$} (0,0) -- (0,1) -- (1,1) -- (1,0);
            \filldraw[black] (0,0) circle (2pt) node[anchor=north west]{$e$};
            \filldraw[black] (1,0) circle (2pt) node[anchor=north west]{$d$};
            \filldraw[black] (0,1) circle (2pt) node[anchor=north west]{};
            \filldraw[black] (1,1) circle (2pt) node[anchor=north west]{};
            \filldraw[black] (1-\a,1-\a) circle (2pt) node[anchor=north east]{};
        \end{tikzpicture}}
        \caption{}
        \label{fig:CFK(figure 8 knot)}
    \end{subfigure}
    \caption{The finite pieces of the knot Floer complexes of the positive trefoil over $\F_2$ ($\textsc{a}$) and the figure-eight knot over $\F_3$ ($\textsc{b}$).}
    \label{fig:finite CFK}
\end{figure}
In general, however, it can happen that the complex is essentially infinite, \emph{i.e.}, it does not split into finite pieces.
\begin{definition}
    Let $\mathcal{R}$ be either $\F[U,V]$ or $\mathcal{R}_1$ and let $(C, \partial)$ be a finitely generated free chain complex over $\mathcal{R}$. Then $C$ is \emph{essentially infinite} if there is no finitely generated chain complex $D$ over $\F$ such that
    $C \cong D \otimes_{\F} \mathcal{R}$.
\end{definition}

Chain complexes $T$ and $D$ constructed in Section \ref{section:examples} are essentially infinite. In such cases, we draw pictures large enough to completely specify the differential $\partial$. These may sometimes contain the same label twice, for example in Figure \ref{fig:local system 1}, but there is no reason to be alarmed. We reiterate that the multiple $d$'s appearing in the picture represent different elements $Vd$, $UV^2d$, $U^2V^3d$ and so on and there is \emph{a priori} no reason why these would not be connected in the knot Floer complex. Indeed, one of the results of this paper is to show that they can be and that there are complexes which remain infinite after any change of basis or even chain homotopy equivalence.

While the knot Floer complexes of the trefoil and the figure-eight knot depicted in Figure \ref{fig:infinite CFK} and Figure \ref{fig:finite CFK} only contain horizontal and vertical arrows, diagonal arrows will also be present in general. In that case, setting $UV=0$ corresponds to deleting all diagonal arrows and the resulting picture represents a chain complex over $\mathcal{R}_1$. (Over $\F[U,V]$, it no longer satisfies $\partial^2=0$ so it is only an $\F[U,V]$-module with an endomorphism.) Most of our pictures represent chain complexes over $\mathcal{R}_1$.

\section{Local systems}\label{section:local systems}
\subsection{Simplified decomposition}
Link Floer complexes admitting simplified bases are generally pleasant to work with. Different Heegaard Floer type invariants can be extracted from the chain homotopy type of $\CFL(L)$ much more easily if we have a vertically and horizontally simplified representative at our disposal. However, it is presently not known whether every chain homotopy type of knot or link Floer complexes contains such a representative. We construct Example $P$ in Section \ref{section:examples} to show that there exist knot Floer-like complexes that do not. This is why we weaken the notion of a simplified basis and replace it with the notion of a simplified decomposition, in hopes that every link Floer complex admits at least a simplified decomposition. This turns out to be true, but proving it will take all of Section \ref{section:classifying algebraic complexes}.
\begin{definition}
    Let $\mathcal{R}$ be either $\F[U,V]$ or $\mathcal{R}_1$ and let $(C, \partial)$ be a finitely generated free chain complex over $\mathcal{R}$. A decomposition $C \cong X_1 \oplus \dots \oplus X_n$ of $C$ into free submodules is \emph{vertically simplified} if for each $i \in \{1, \dots, n\}$ we have exactly one of the possibilities in the complex $C/U$:
    \begin{enumerate}
        \item $\partial|_{X_i}: X_i \to V^a X_j$ is an isomorphism for some $j \in \{1, \dots, n\}$ and $a \in \N_0$, or
        \item $\partial X_i = 0$ and there exist at most one pair $j \in \{1, \dots, n\}$ and $a \in \N_0$ such that $\partial|_{X_j}: X_j \to V^aX_i$ is an isomorphism.
    \end{enumerate}
    A decomposition is \emph{horizontally simplified} if for each $i \in \{1, \dots, n\}$ we have exactly one of the possibilities in the complex $C/V$:
    \begin{enumerate}
        \item $\partial|_{X_i}: X_i \to U^a X_j$ is an isomorphism for some $j \in \{1, \dots, n\}$ and $a \in \N_0$, or
        \item $\partial X_i = 0$ and there exist at most one pair $j \in \{1, \dots, n\}$ and $a \in \N_0$ such that $\partial|_{X_j}: X_j \to U^aX_i$ is an isomorphism.
    \end{enumerate}
    A decomposition that is simultaneously vertically and horizontally simplified is called \emph{simplified}. The maps $V^{-a}\partial|_{X_i}: X_i \to X_j$ and $U^{-a}\partial|_{X_i}: X_i \to X_j$ are called \emph{vertical} and \emph{horizontal isomorphisms} respectively. (Since $X_j$ is a free submodule of $C$, there are well-defined maps $V^{-a}: V^aX_j \to X_j$ and $U^{-a}: U^aX_j \to X_j$.)
\end{definition}
We emphasize that the decompositions considered in the above definition are decompositions of $C$ as an $\F[U,V]$-module, \emph{not} as a chain complex over $\F[U,V]$. There is a separate notion of decomposability of $C$ as a chain complex that we will also need.
\begin{definition}
    Let $(C, \partial)$ be a finitely generated free chain complex over $\F[U,V]$ or $\mathcal{R}_1$. We say $C$ is \emph{indecomposable} if $C \cong C_1 \oplus C_2$ implies that $C_1$ or $C_2$ is zero.
\end{definition}
Note that any simplified basis gives rise to a simplified decomposition in which every submodule is generated by a single element. An example of such a complex is depicted in Figure \ref{fig:simplified complex 1}. See also Figure \ref{fig:simplified complex 2} for an example of a complex admitting a simplified decomposition that does not arise from a simplified basis. \emph{A priori}, it is unclear whether this complex has a simplified basis, but we address the general case of this question in Section \ref{section:examples}.
\begin{figure}[t]
    \begin{subfigure}[b]{0.40\textwidth}
        \centering
        \begin{tikzpicture}[scale=1.1]
            \def\a{0.1}
            \draw[step=1.0,gray,thin] (0.5,-0.5) grid (4.5,3.5);
            \draw[black, very thick] (1,1-\a)--(2,1-\a)--(2,0)--(3,0)--(3,1+\a)--(2,1+\a)--(2,2)--(1,2)--(1,1-\a);
            \draw[black, very thick] (4,3)--(3,3)--(3,2)--(4,2)--(4,3);

            \filldraw[black] (4,3) circle (2pt);
            \filldraw[black] (3,3) circle (2pt);
            \filldraw[black] (3,2) circle (2pt);
            \filldraw[black] (4,2) circle (2pt);
            
            \filldraw[black] (1,1-\a) circle (2pt) node[anchor=north west]{};
            \filldraw[black] (2,1+\a) circle (2pt) node[anchor=north west]{};
            \filldraw[black] (2,1-\a) circle (2pt) node[anchor=north west]{};
            \filldraw[black] (2,0) circle (2pt) node[anchor=north west]{};
            \filldraw[black] (3,0) circle (2pt) node[anchor=north west]{};
            \filldraw[black] (3,1+\a) circle (2pt) node[anchor=north west]{};
            \filldraw[black] (2,2) circle (2pt) node[anchor=north west]{};
            \filldraw[black] (1,2) circle (2pt) node[anchor=north west]{};
        \end{tikzpicture}
        \caption{}
        \label{fig:simplified complex 1}
    \end{subfigure}
    \hspace{1em}
    \begin{subfigure}[b]{0.40\textwidth}
        \centering
        \begin{tikzpicture}[scale=1.1]
        \draw[step=1.0,gray,thin] (0.5,-0.5) grid (4.5,3.5);
        \def\a{0.1}
        \draw[black, very thick] (4+\a,3+\a)--(1-\a,3+\a)--(1-\a,0-\a)--(2+\a,0-\a)--(2+\a,1-\a)--(4+\a,1-\a);
        \draw[black, very thick] (4-\a,3-\a)--(1+\a,3-\a)--(1+\a,0+\a)--(2-\a,0+\a)--(2-\a,1+\a)--(4-\a,1+\a) --(4-\a,3-\a);
        
        \draw[black, very thick] (4+\a, 3+\a)--(4-\a, 1+\a);
        \draw[black, very thick] (4-\a, 3-\a)--(4+\a, 1-\a);

        \draw[black, very thick] (2,2)--(3,2)--(3,0);
        \filldraw[black] (2,2) circle (2pt);
        \filldraw[black] (3,2) circle (2pt);
        \filldraw[black] (3,0) circle (2pt);
        
        \filldraw[black] (4+\a,3+\a) circle (2pt);
        \filldraw[black] (1-\a,3+\a) circle (2pt);
        \filldraw[black] (1-\a,0-\a) circle (2pt);
        \filldraw[black] (2+\a,0-\a) circle (2pt);
        \filldraw[black] (2+\a,1-\a) circle (2pt);
        \filldraw[black] (4+\a,1-\a) circle (2pt);

        \filldraw[black] (4-\a,3-\a) circle (2pt);
        \filldraw[black] (1+\a,3-\a) circle (2pt);
        \filldraw[black] (1+\a,0+\a) circle (2pt);
        \filldraw[black] (2-\a,0+\a) circle (2pt);
        \filldraw[black] (2-\a,1+\a) circle (2pt);
        \filldraw[black] (4-\a,1+\a) circle (2pt);
        \end{tikzpicture}
        \caption{}
        \label{fig:simplified complex 2}
    \end{subfigure}
    \caption{Two examples of chain complexes over $\mathcal{R}_1$ admitting simplified decompositions.}
    \label{fig:simplified complexes}
\end{figure}
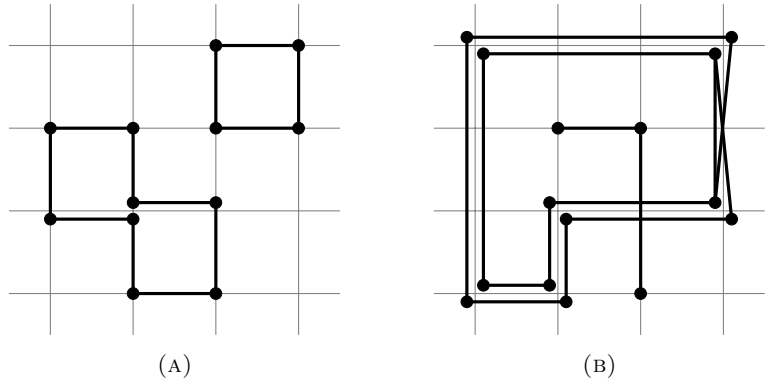

\begin{definition}\label{def:indecomposable local system}
    Let $\mathcal{R}$ be either $\F[U,V]$ or $\mathcal{R}_1$ and let $(L, \partial)$ be a finitely generated free chain complex over $\mathcal{R}$ with no arrows of length $0$. Then $L$ is an \emph{indecomposable local system} if it
    \begin{enumerate}
        \item admits a simplified decomposition,
        \item is indecomposable as a chain complex over $\mathcal{R}$, and
        \item has torsion homology.
    \end{enumerate}
\end{definition}
Let us consider some examples to better understand what kind of objects are captured by this definition. Both complexes in Figure \ref{fig:simplified complexes} decompose into two summands as chain complexes and are hence not indecomposable. For the example in Figure \ref{fig:simplified complex 1}, this is the only issue and each summand considered separately becomes an indecomposable local system. Of the two summands in Figure \ref{fig:simplified complex 2}, only the one of rank $12$ is an indecomposable local system, since the complex of rank $3$ does not satisfy the condition on homology. Figure \ref{fig:local systems} depicts two further examples of indecomposable local systems. Whenever the dots are unlabelled, they represent different generators.
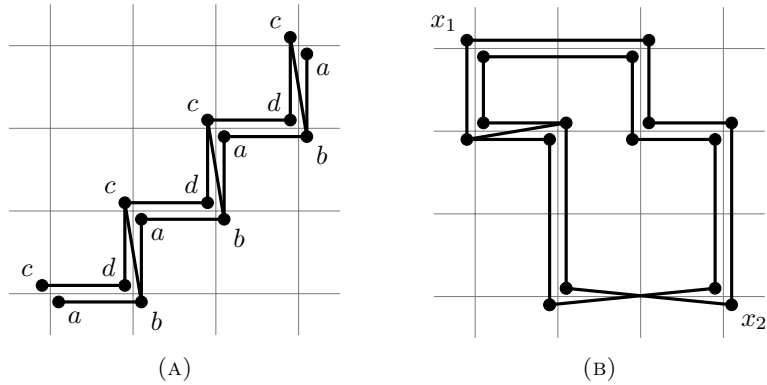
\begin{figure}[t]
    \begin{subfigure}[b]{0.40\textwidth}
        \centering
        \begin{tikzpicture}[scale=1.1]
            \def\a{0.1}
            \draw[step=1.0,gray,thin] (0.5,-0.5) grid (4.5,3.5);
            \draw[black, very thick](4-\a,3+\a)--(4-\a,2+\a)--(3-\a,2+\a)--(3-\a,1+\a)--(2-\a,1+\a)--(2-\a,0+\a)--(1-\a,0+\a);
            \draw[black, very thick](4+\a,3-\a)--(4+\a,2-\a)--(3+\a,2-\a)--(3+\a,1-\a)--(2+\a,1-\a)--(2+\a,0-\a)--(1+\a,0-\a);

            \draw[very thick](4-\a, 3+\a)--(4+\a,2-\a){};
            \draw[very thick](3-\a, 2+\a)--(3+\a,1-\a){};
            \draw[very thick](2-\a, 1+\a)--(2+\a,0-\a){};
            
            \filldraw[black] (4+\a,3-\a) circle (2pt) node[anchor=north west]{$a$};
            \filldraw[black] (4+\a,2-\a) circle (2pt) node[anchor=north west]{$b$};
            \filldraw[black] (3+\a,2-\a) circle (2pt) node[anchor=north west]{$a$};
            \filldraw[black] (3+\a,1-\a) circle (2pt) node[anchor=north west]{$b$};
            \filldraw[black] (2+\a,1-\a) circle (2pt) node[anchor=north west]{$a$};
            \filldraw[black] (2+\a,0-\a) circle (2pt) node[anchor=north west]{$b$};
            \filldraw[black] (1+\a,0-\a) circle (2pt) node[anchor=north west]{$a$};

            \filldraw[black] (4-\a,3+\a) circle (2pt) node[anchor=south east]{$c$};
            \filldraw[black] (4-\a,2+\a) circle (2pt) node[anchor=south east]{$d$};
            \filldraw[black] (3-\a,2+\a) circle (2pt) node[anchor=south east]{$c$};
            \filldraw[black] (3-\a,1+\a) circle (2pt) node[anchor=south east]{$d$};
            \filldraw[black] (2-\a,1+\a) circle (2pt) node[anchor=south east]{$c$};
            \filldraw[black] (2-\a,0+\a) circle (2pt) node[anchor=south east]{$d$};
            \filldraw[black] (1-\a,0+\a) circle (2pt) node[anchor=south east]{$c$};
        \end{tikzpicture}
        \caption{}
        \label{fig:local system 1}
    \end{subfigure}
    \hspace{1em}
    \begin{subfigure}[b]{0.40\textwidth}
        \centering
        \begin{tikzpicture}[scale=1.1]
        \draw[step=1.0,gray,thin] (0.5,-0.5) grid (4.5,3.5);
        \def\a{0.1}
        \draw[black, very thick] (2+\a,0+\a)--(4+\a,0-\a)--(4+\a,2+\a)--(3+\a,2+\a)--(3+\a,3+\a)--(1-\a,3+\a)--(1-\a,2-\a)--(2-\a,2-\a)--(2-\a,0-\a);
        \draw[black, very thick] (2-\a,0-\a)--(4-\a,0+\a)--(4-\a,2-\a)--(3-\a,2-\a)--(3-\a,3-\a)--(1+\a,3-\a)--(1+\a,2+\a)--(2+\a,2+\a)--(2+\a,0+\a);

        \draw[black, very thick](1-\a,2-\a)--(2+\a,2+\a);
        
        \filldraw[black] (2-\a,0-\a) circle (2pt);
        \filldraw[black] (4+\a,0-\a) circle (2pt) node[anchor=north west]{$x_2$};
        \filldraw[black] (4+\a,2+\a) circle (2pt);
        \filldraw[black] (3+\a,2+\a) circle (2pt);
        \filldraw[black] (3+\a,3+\a) circle (2pt);
        \filldraw[black] (1-\a,3+\a) circle (2pt) node[anchor=south east]{$x_1$};
        \filldraw[black] (1-\a,2-\a) circle (2pt);
        \filldraw[black] (2-\a,2-\a) circle (2pt);
        \filldraw[black] (2-\a,0-\a) circle (2pt);

        \filldraw[black] (2+\a,0+\a) circle (2pt);
        \filldraw[black] (4-\a,0+\a) circle (2pt);
        \filldraw[black] (4-\a,2-\a) circle (2pt);
        \filldraw[black] (3-\a,2-\a) circle (2pt);
        \filldraw[black] (3-\a,3-\a) circle (2pt);
        \filldraw[black] (1+\a,3-\a) circle (2pt);
        \filldraw[black] (1+\a,2+\a) circle (2pt);
        \filldraw[black] (2+\a,2+\a) circle (2pt);
        \filldraw[black] (2+\a,0+\a) circle (2pt);
        \end{tikzpicture}
        \caption{}
        \label{fig:local system 2}
    \end{subfigure}
    \caption{Two examples of indecomposable local systems over $\mathcal{R}_1$. Some of their generators are labelled.}
    \label{fig:local systems}
\end{figure}

\vspace{1em}

Let us now try to extend this to the definition to all local systems, not just indecomposable ones. Removing the word `indecomposable' from Definition \ref{def:indecomposable local system} does not work; we do not want the complex in Figure \ref{fig:simplified complex 1} to be a local system since the two summands have nothing to do with each other. On the other hand, consider the complex in Figure \ref{fig:local system}.
\begin{figure}[t]
    \begin{tikzpicture}[scale=1.0]
            \draw[step=1.0,gray,thin] (0.5,0.5) grid (3.5,3.5);
            \def\a{0.2}
            \draw[very thick] (1+\a, 1+\a) -- (1+\a, 3-\a) -- (3-\a, 3-\a) -- (3-\a,1+\a) -- (1+\a,1+\a);
            \draw[very thick] (1-\a, 1-\a) -- (1-\a, 3+\a) -- (3+\a, 3+\a) -- (3+\a,1-\a) -- (1-\a,1-\a);
            \draw[very thick](1,1)--(3,1)--(3,3)--(1,3)--(1,1);
            \draw[very thick](3,1)--(1+\a,1+\a);
            
            \filldraw[black] (1+\a,1+\a) circle (2pt);
            \filldraw[black] (1+\a,3-\a) circle (2pt);
            \filldraw[black] (3-\a,1+\a) circle (2pt);
            \filldraw[black] (3-\a,3-\a) circle (2pt);
            \filldraw[black] (1-\a,1-\a) circle (2pt);
            \filldraw[black] (1-\a,3+\a) circle (2pt);
            \filldraw[black] (3+\a,1-\a) circle (2pt);
            \filldraw[black] (3+\a,3+\a) circle (2pt);
            \filldraw[black] (1,1) circle (2pt);
            \filldraw[black] (3,1) circle (2pt);
            \filldraw[black] (1,3) circle (2pt);
            \filldraw[black] (3,3) circle (2pt);
        \end{tikzpicture}
    \caption{A local system that is not indecomposable.}
    \label{fig:local system}
\end{figure}
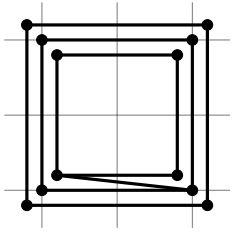
Even though it decomposes into two summands as a chain complex, we would like them to be a part of a single local system. More generally, if some indecomposable local systems have the same `shape' and position in the plane, they will form a local system. To define this rigorously, we must first introduce the notion of a shape.

\subsection{Shape}
We introduce an unusual order on $\Z$ that was first considered in \cite[Section 4.2]{dai2021more}.
\begin{definition}\label{def:DHST order}
    Let $\Z^!=(\Z, \leq^!)$ be the totally ordered set with
    $$ -1 \leq^! -2 \leq^! -3 \leq^! \dots \leq^! 0 \leq^! \dots \leq^! 3 \leq^! 2 \leq^! 1.$$
\end{definition}
We can extend the unusual order $\leq^!$ from $\Z$ to a total order on the set $\Z^\N$ of integer sequences by ordering them lexicographically with respect to $\leq^!$. This lets us formally define what a shape of an indecomposable local system is.
\begin{definition}
Let $L$ be an indecomposable local system. The shape of $L$ is
$$s(L) = {\sup}^!\{ s \ | \ \text{there exists an injective chain map } S_h(s) \to L \}.$$
\end{definition}
We also have the following equivalent definition, which is more geometric. Consider the set $S'$ of periodic nonzero integer sequences $(s_i)_{i\in\N}$. For each sequence $s\in S'$ there is a minimal \emph{even} period $2p\in 2\N$ such that $s$ is periodic with period $2p$. Let us restrict our attention to those $s\in S'$ with a minimal even period $2p$ that satisfy $s_1 + \dots + s_{2p-1} = s_2+\dots+s_{2p}$. Such sequences describe the shapes of local systems over $\mathcal{R}_1$ as illustrated by the following two examples. The sequence $1, 1, \dots$ describes the shape of the local system in Figure \ref{fig:local system 1} and the sequence $2, 2, -1, 1, -2, -1, 1, -2, \dots$ describes the shape of the local system in Figure \ref{fig:local system 2}. However, observe that this description is not unique; cyclic permutations of the sequences can also describe the same shape and so does `going in the opposite direction'. To eliminate this ambiguity in a formal manner we define an equivalence relation $\sim$ by setting $s \sim s'$ if $s,s'$ are both periodic with a minimal even period $2p$ and there exists a $2k < 2p$ such that $s'_i=s_{i+2k}$ or $s'_i=-s_{2p+2k-i}$ for all $i < 2p$. A shape of an equivalence class is the highest representative of that class.



\vspace{1em}

Let $L$ be an indecomposable local system and let $L \cong X_1 \oplus \dots \oplus X_n$ be its simplified decomposition. For an arbitrary $i \in \{1, \dots, n\}$ consider $X_i$. Because the decomposition is horizontally simplified we have one of the following possibilities in $L/V$.
\begin{enumerate}
    \item $\partial|_{X_i}: X_i \to U^a X_j$ is an isomorphism for some $j$ and $a \in \N$, or
    \item $\partial X_i = 0$ and there are at most one $j$ and $a \in \N$ such that $\partial|_{X_j}: X_j \to U^aX_i$ is an isomorphism.
\end{enumerate}
The homology of $L/V$ is torsion. Provided the possibility (1) does not occur, the possibility (2) must occur with exactly one $j$ and $a\in \N$ such that $\partial|_{X_j}: X_j \to U^aX_i$ is an isomorphism. In other words, each $X_i$ is a domain or a codomain of some horizontal isomorphism. It cannot be both because $\partial^2=0$ in $L$. A similar analysis allows us to conclude that each $X_i$ is a domain or a codomain of exactly one vertical isomorphism. Starting at $X_1$ and moving in a horizontal direction, we can travel along these isomorphisms and construct a periodic sequence. 

\begin{definition}\label{def:local system}
A \emph{local system} is a direct sum of indecomposable local systems of the same shape and in the same position in the plane.
\end{definition}

\subsection{Classifying local systems}\label{subsection:classification of local systems}
In this subsection, we classify local systems up to isomorphism by proving that they admit a simple combinatorial description: they are completely characterized by the data of an invertible matrix together with the shape of $L$.

\begin{proposition}\label{prop:classification of local systems}
Let $\mathcal{C}$ be the set of isomorphism classes of local systems over $\mathcal{R}_1$ and let $\mathcal{D} = \{ (s, w, [A]) \ | \ s\in S, \ w\in\N, \ [A] \text{ conjugacy class of }A\in\GL w\F\}$. There is a bijection $f:\mathcal{C}\to\mathcal{D}$.
\end{proposition}
\begin{proof}
Let $L$ be a local system and let $X_1\oplus \dots \oplus X_n$ be a simplified decomposition of $L$. We will now describe the triple $f(L)=(s, w, [A])$. Since $L$ consists of indecomposable local systems of the same shape, this gives a well-defined $s$. Moreover, we set $w = \mathrm{rk}_{\mathcal{R}_1}X_i$ while observing that this is independent of $i$.

Finally, we define the conjugacy class $[A]$ of the matrix $A \in \GL w \F$. We can relabel the free submodules appearing in the decomposition cyclically in the following manner. Fix $X_1$ and let $X_2$ be chosen so that there is a horizontal isomorphism between $X_1$ and $X_2$. Then we choose $X_3$ in such a way that there is a vertical isomorphism between $X_2$ and $X_3$. We proceed in a similar manner until we return back to $X_1$.

Choose a basis for $X_1$. There exists a unique basis for $X_2$ such that the horizontal isomorphism between $X_1$ and $X_2$ is represented by the identity matrix with respect to these bases. Choose this basis for $X_2$. There is now a unique basis for $X_3$ such that the vertical isomorphism between $X_2$ and $X_3$ is represented by the identity. Choose this basis for $X_3$ and proceed in a similar manner until the bases for all summands $X_1, \dots, X_n$ have been chosen. The vertical isomorphism between $X_1$ and $X_n$ is the only part of $\partial$ that is not represented by an identity matrix. Let it instead be represented by $A$, yielding a conjugacy class $[A]$ we sought to define. A different choice of basis for $X_1$ produces a conjugate matrix so $f$ is well-defined.

\vspace{1em}

We now construct the inverse map. For a triple $(s, w, [A])$, let $2p$ be the smallest even period of the sequence $s$. Let $X_1, \dots, X_{2p}$ be free $\mathcal{R}_1$-modules of rank $w$ whose spatial arrangement in the plane is determined by $s$. Moreover, they are equipped with distinguished bases with respect to which all differentials except the vertical isomorphism between $X_1$ and $X_{2p}$ are identities. The vertical isormorphism between $X_1$ and $X_{2p}$ is given by any matrix in $[A]$. This constructs a finitely generated free chain complex over $\mathcal{R}_1$ and it is not difficult to see that the maps are inverses to each other.
\end{proof}
Let us consider an example to see how this works in practice. The local systems depicted in Figure \ref{fig:local system 2} and Figure \ref{fig:local system in classification form} both correspond to the triple $(s, w, [A])=\left((1,-2,2,2,-1,1,-2,-1, \dots), 2, \begin{pmatrix}
    0&1\\
    1&1
\end{pmatrix}\right)$. Indeed, they are isomorphic via the change of basis described in the proof of Proposition \ref{prop:classification of local systems}.

\begin{figure}[t]
    \begin{tikzpicture}[scale=1.1]
        \draw[step=1.0,gray,thin] (0.5,-0.5) grid (4.5,3.5);
        \def\a{0.1}
        \draw[black, very thick] (2-\a,0-\a)--(4+\a,0-\a)--(4+\a,2+\a)--(3+\a,2+\a)--(3+\a,3+\a)--(1-\a,3+\a)--(1-\a,2-\a)--(2-\a,2-\a)--(2+\a,0+\a);
        \draw[black, very thick] (2+\a,0+\a)--(4-\a,0+\a)--(4-\a,2-\a)--(3-\a,2-\a)--(3-\a,3-\a)--(1+\a,3-\a)--(1+\a,2+\a)--(2+\a,2+\a)--(2+\a,0+\a);

        \draw[black, very thick](2-\a,0-\a)--(2+\a,2+\a);
        
        \filldraw[black] (2-\a,0-\a) circle (2pt) node[anchor=north east]{$X_1$};
        \filldraw[black] (4+\a,0-\a) circle (2pt) node[anchor=north west]{$X_2$};
        \filldraw[black] (4+\a,2+\a) circle (2pt) node[anchor=south west]{$X_3$};
        \filldraw[black] (3+\a,2+\a) circle (2pt);
        \filldraw[black] (3+\a,3+\a) circle (2pt);
        \filldraw[black] (1-\a,3+\a) circle (2pt);
        \filldraw[black] (1-\a,2-\a) circle (2pt);
        \filldraw[black] (2-\a,2-\a) circle (2pt) node[anchor=north east]{$X_{2p}$};
        \filldraw[black] (2-\a,0-\a) circle (2pt);

        \filldraw[black] (2+\a,0+\a) circle (2pt);
        \filldraw[black] (4-\a,0+\a) circle (2pt);
        \filldraw[black] (4-\a,2-\a) circle (2pt);
        \filldraw[black] (3-\a,2-\a) circle (2pt);
        \filldraw[black] (3-\a,3-\a) circle (2pt);
        \filldraw[black] (1+\a,3-\a) circle (2pt);
        \filldraw[black] (1+\a,2+\a) circle (2pt);
        \filldraw[black] (2+\a,2+\a) circle (2pt);
        \filldraw[black] (2+\a,0+\a) circle (2pt);

        \draw (0,1) node[anchor = east]{$\begin{pmatrix}
        0&1\\
        1&1
        \end{pmatrix}$};
        \draw[<->, black, very thick] (0,1) to (2-\a,1);
    \end{tikzpicture}
    \caption{Local system from Figure \ref{fig:local system 2} with a different choice of basis. The basis was chosen so that all isomorphisms except one are identities.}
    \label{fig:local system in classification form}
\end{figure}

\section{Classifying algebraic complexes\label{section:classifying algebraic complexes}}
Our techniques can be used for the classification of a larger class of algebraically defined chain complexes, not just of those that are realized as link Floer complexes $\CFL_{\mathcal{R}_1}(L)$. We classify both isomorphism classes and chain homotopy equivalence classes of finitely generated free chain complexes over $\mathcal{R}_1$. This section proves the following two theorems.

\begin{theorem}\label{thm:classification of algebraic complexes up to iso}
    Let $C$ be a finitely generated free chain complex over $\mathcal{R}_1$. Then there exist unique snake complexes $S_1, \dots, S_m$, local systems $L_1, \dots, L_n$, and zero complexes $Z_1, \dots, Z_k$ such that 
    $$C \cong S_1 \oplus \dots \oplus S_m \oplus L_1 \oplus \dots \oplus L_n \oplus Z_1 \oplus \dots \oplus Z_k.$$
\end{theorem}
\begin{theorem}\label{thm:classification of algebraic complexes up to homotopy}
    Let $C$ be a finitely generated free chain complex over $\mathcal{R}_1$. Then there exist unique snake complexes $S_1, \dots, S_m$ and local systems $L_1, \dots, L_n$ such that 
    $$C \simeq S_1 \oplus \dots \oplus S_m \oplus L_1 \oplus \dots \oplus L_n.$$
\end{theorem}
\begin{remark}
Note that Theorem \ref{thm:classification of algebraic complexes up to homotopy} is not a trivial corollary of Theorem \ref{thm:classification of algebraic complexes up to iso} since we are also requiring the uniqueness of the summands up to chain homotopy equivalence.
\end{remark}

When $C$ has the homology of a knot in $S^3$, either of the theorems can be used to recover the results of \cite{dai2021more} with a different argument. Theorem \ref{thm:classification of CFL up to iso}, Corollary \ref{cor:thm for knots in S3} and other topological applications are deduced as immediate corollaries in Section \ref{section:topological applications}.

\subsection{Results on bases}\label{subsection:bases}
In this section we prove some preliminary results about general finitely generated free chain complexes over $\F[U,V]$ and $\mathcal{R}_1 = \frac{\F[U,V]}{(UV)}$.
\begin{lemma}\label{lemma:bigrading (i,j)}
    Let $\mathcal{R}$ be either $\F[U,V]$ or $\mathcal{R}_1$ and let $(C, \partial)$ be a finitely generated free chain complex over $\mathcal{R}$. For any $i, j \in \Z$, any two bases of $C$ have the same number of elements in bigrading $(i,j)$.
\end{lemma}
   \begin{proof}
    For any $i, j \in \Z$, let $V_{ij}$ denote the $\F$-vector space in bigrading $(i,j)$. If $\mathcal{R} = \F[U,V]$, then $\dim_\F V_{ij}$ is the sum of the numbers of generators in bigradings $(i+2k_1, j+2k_2)$ as $k_1$ and $k_2$ range over $\N_0$. If $\mathcal{R} = \mathcal{R}_1$, then $\dim_\F V_{ij}$ is the sum of the numbers of generators in bigradings $(i+2k_1, j+2k_2)$ as $k_1$ and $k_2$ range over $\N_0$ and at least one of them is $0$.

    In either case, we partially order $\Z^2$ such that $(i_1, j_1) \leq (i_2, j_2)$ if $i_1 \leq i_2$ and $j_1 \leq j_2$. Let $(i,j)$ be maximal with respect to this partial order such that $\dim_{\F} V_{ij} \neq 0$. Then there are $\dim_{\F} V_{ij}$ generators of $C$ in bigrading $(i,j)$. Remove them and by inductive hypothesis we can determine the bigradings of the other basis elements.
\end{proof}
Let $(C, \partial)$ be a chain complex over $\mathcal{R}_1$. In what follows, we will find a nice basis for $C$ in which it splits into standard complexes and local systems. Let $\{x_1', \dots, x_n'\}$ be a vertically simplified basis for $C$ and let $\{y_1', \dots, y_n'\}$ be a horizontally simplified basis for $C$. In light of Lemma \ref{lemma:bigrading (i,j)}, we can assume that the generators are graded so that $\gr(x_i') = \gr(y_i')$ for every $i \in \{1, \dots, n\}$. We now consider how these bases are related to each other. For every $i \in \{1, \dots, n\}$ there exist unique monomials $p_{i1}, \dots, p_{in}, q_{i1}, \dots, q_{in}  \in \mathcal{R}_1$ such that
$$x_i' = \sum_{j=1}^{n} p_{ij}y_j' \quad \text{and} \quad y_i' = \sum_{j=1}^{n} q_{ij}x_j'.$$
Moreover, the above equations preserve the bigrading. 

\vspace{1em}
    
Monomials $p_{ij}, q_{ij}$ give rise to matrices $P', Q' \in \GL n{\mathcal{R}_1}$ satisfying $P'Q'=I$. We emphasize that in general $P'$ and $Q'$ can have multiples of $U$ and $V$ as their elements. However, we can choose bases in a way to avoid this.

\begin{lemma}\label{lemma:matrix operations}
Let $(C, \partial)$ be a finitely generated free chain complex over $\mathcal{R}_1$. There exist a vertically simplified basis $\{x_1, \dots, x_{n}\}$ and a horizontally simplified basis $\{y_1, \dots, y_n\}$ for $C$ such that the corresponding transition matrices $P, Q\in \GL n \F$.
\end{lemma}
Before we prove the lemma, we review some linear algebra over $\mathcal{R}_1$ in order to establish the notation. Elementary matrices over $\mathcal{R}_1$ that we will use are
    \[ 
    T_{ij} =
    \scalemath{0.8}{
    \begin{pmatrix}
    1&&&&&&\\
    &\ddots&&&&&\\
    &&0&\cdots&1&&\\
    &&\vdots&&\vdots&&\\
    &&1&\cdots&0&&\\
    &&&&&\ddots&\\
    &&&&&&1
    \end{pmatrix}},
    \quad
    E_{ij}^\lambda = 
    \scalemath{0.8}{
    \begin{pmatrix}
    1&&&&&&\\
    &\ddots&&&&&\\
    &&1&\cdots&\lambda&&\\
    &&\vdots&&\vdots&&\\
    &&0&\cdots&1&&\\
    &&&&&\ddots&\\
    &&&&&&1
    \end{pmatrix}},
    \]
    \[
    E_{ij}^{\lambda U^k} = 
    \scalemath{0.8}{
    \begin{pmatrix}
    1&&&&&&\\
    &\ddots&&&&&\\
    &&1&\cdots&\lambda U^k&&\\
    &&\vdots&&\vdots&&\\
    &&0&\cdots&1&&\\
    &&&&&\ddots&\\
    &&&&&&1
    \end{pmatrix}},
    \quad
    E_{ij}^{\lambda V^k} = 
    \scalemath{0.8}{
    \begin{pmatrix}
    1&&&&&&\\
    &\ddots&&&&&\\
    &&1&\cdots&\lambda V^k&&\\
    &&\vdots&&\vdots&&\\
    &&0&\cdots&1&&\\
    &&&&&\ddots&\\
    &&&&&&1
    \end{pmatrix}},
    \]
    and
    \[
    D_i^\lambda =
    \scalemath{0.8}{
    \begin{pmatrix}
    1&&&&\\
    &\ddots&&&\\
    &&\lambda&&\\
    &&&\ddots&\\
    &&&&1
    \end{pmatrix}}
    \]
    for some distinct $i,j$ and arbitrary $k \in \N$ and nonzero $\lambda\in\F$. Performing row operations on a matrix $P$ is equivalent to multiplying it with elementary matrices from the left and performing column operations on $P$ is equivalent to multiplying it with elementary matrices from the right. Explicitly, right multiplication of a given matrix by $T_{ij}$ swaps its $i^{\text{th}}$ and $j^{\text{th}}$ columns, right multiplication by $D_i^\lambda$ multiplies its $i^\text{th}$ column by $\lambda$ and right multiplication by $E_{ij}^p$ for $p\in\{\lambda, \lambda U^k, \lambda V^k\}$ adds $p$ times its $i^\text{th}$ column to its $j^\text{th}$ column. Left multiplication of a given matrix by elementary matrices has an analogous effect on its rows.

    Note that the obvious generalization $E_{ij}^p$ is an elementary matrix for any $p \in \mathcal{R}_1$. Since our matrices correspond to basis changes of $\Z\oplus\Z$ graded chain complexes, the element $p$ should be homogeneous. It follows that it is a unit multiplied by $1$, a power of $U$, or a power of $V$, \emph{i.e.}, exactly of one of the previously given forms.

    Finally, let us remark that the inverses of the given elementary matrices are also elementary matrices. We have $T_{ij}^2=I$ so $T_{ij}$ is its own inverse and the inverses of the other elementary matrices are $(D_i^\lambda)^{-1}=D_i^{\lambda^{-1}}$ and $(E_{ij}^p)^{-1}=E_{ij}^{-p}$. 

    \begin{proof}[Proof of Lemma \ref{lemma:matrix operations}]
    Let $\{x_1', \dots, x_{n}'\}$ be any vertically simplified, $\{y_1', \dots, y_{n}'\}$ any horizontally simplified basis for $C$, and let $P'$ and $Q'$ be the transition matrices between the two bases. Reducing $P'$ and $Q'$ mod $(U,V)$ gives block diagonal matrices $\widetilde{P}, \widetilde{Q} \in \GL n \F$ satisfying $\widetilde{P}\widetilde{Q}=I$. It follows that $\widetilde{P}$ is invertible and linear algebra tells us that it can be reduced to $I$ using elementary row and column operations. This gives a factorization $R_m \cdots R_1\widetilde{P} C_1 \cdots C_n = I$ for some $m, n \in \N_0$ where each $R_i$ and $C_i$ is an elementary matrix $T_{ij}$ or $E_{ij}$ for some distinct $i,j$. Applying the same operations to the nonreduced matrix $P'$ gives $R_m \cdots R_1 P' C_1 \cdots C_n = J$ where $J \equiv I$ mod $(U, V)$. For future use in this proof we denote $R = (R_m \cdots R_1)^{-1}$ and $C = (C_1 \cdots C_n)^{-1}$ and note that they are both elements of $\GL n \F$.

     We will now perform further elementary row and column operations over $\mathcal{R}_1$ to reduce $J$ to $I$. These will all be of the forms $E_{ij}^{\lambda V^k}$ or $E_{ij}^{\lambda U^k}$ for various distinct $i,j$, a nonzero $\lambda \in \F$ and $k\in\N$ such that all elements of the form $\lambda U^k$ will be eliminated through column operations and all elements of the form $\lambda V^k$ will be eliminated through row operations. Start in the first row of $J$ and move from left to right. If there is a $\lambda U^k$ in position $(1, j)$, we add $\lambda U^k$ times the $1^{\text{st}}$ column to the $j^{\text{th}}$ column. In other words, we multiply the existing matrix from the right with $E_{1j}^{\lambda U^k}$. This makes the element in position $(1, j)$ equal to $0$ and we can move onward. Once we are done with the $1^{\text{st}}$ row, we continue with the second row and so on until we have removed all powers of $U$ from the matrix. This algorithm works because of the following two observations.
     \begin{enumerate}
         \item When we are in the $i^{\text{th}}$ row, all elements in positions $(l,i)$ for $l < i$ are $0$ mod $V$. In other words, they are, possibly zero, multiples of $V$. This is because all powers of $U$ have already been removed in previous steps. Therefore, adding a multiple of the $i^{\text{th}}$ column to any other column cannot undo our progress.
         \item The diagonal elements of the matrix are $1$ at any stage of the process. Reducing the procedure mod $(U,V)$ shows that they are at least equal to $1$ mod $(U,V)$. We infer they must all be $1$ since the elements of the matrix must be homogeneous for grading reasons.
     \end{enumerate}
     We have now removed all multiples of $U$ from $J$. With an analogous algorithm using row operations, all multiples of $V$ can be removed. Thus, there is a factorization $R_V J C_U= I$ where $R_V$ is a product of elementary matrices of the form $E_{ij}^{\lambda V^k}$ and $C_U$ is a product of elementary matrices of the form $E_{ij}^{\lambda U^k}$ for various distinct $i,j$, a nonzero $\lambda \in \F$ and $k\in\N$.

    \vspace{1em}

     Let $\{x_1, \dots, x_n\}$ be a new basis for $C$ related to $\{x_1', \dots, x_n'\}$ by the transition matrix $M_U = C^{-1}C_UC$. Note that the new basis is still vertically simplified since $C_U \equiv I$ mod $U$ implies $M_U \equiv I$ mod $U$ and the vertical arrows are invariant under such basis changes. Similarly, let $\{y_1, \dots, y_n\}$ be a new basis for $C$ related to $\{y_1', \dots, y_n'\}$ by the transition matrix  $M_V = RR_V^{-1}R^{-1}$. The new basis is still horizontally simplified since $R_V \equiv I$ mod $V$ and thus also $R_V^{-1} \equiv I$ mod $V$ and $M_V \equiv I$ mod $V$.

     The original bases $\{x_1', \dots, x_n'\}$ and $\{y_1, \dots, y_n'\}$ are related through the transition matrix $P'$. The new bases $\{x_1, \dots, x_n\}$ and $\{y_1, \dots, y_n\}$ are related through the transition matrix $P = M_V^{-1}P'M_U = RR_VR^{-1} P' C^{-1} C_U C = RR_VJC_UC=RC$. Since $Q=P^{-1}$ and $P = RC \in \GL n \F$, we are done.
    \end{proof}
    From now on, we may assume that our bases $\{x_1, \dots, x_n\}$ and $\{y_1, \dots, y_n\}$ are such that the matrices $P, Q \in \GL n \F$. In addition to that, the proof of Lemma \ref{lemma:matrix operations} provides us with a factorization $P = P_1 \cdots P_k$ where each $P_i \in \GL n \F$ is an elementary matrix. In fact, the following lemma shows that we can strengthen this result slightly.
    \begin{lemma}\label{lemma:can remove annoying arrows}
    Let $P \in \GL n \F$ be a matrix. There is a factorization $P = P_1 \cdots P_k$ where each $P_i$ is of the form $T_{ij}$, $E_{ij}^1$ or $D_i^\lambda$ for some distinct $i,j$ and a nonzero $\lambda \in \F$.
    \end{lemma}
    \begin{proof}
    By linear algebra or the preceding discussion, every invertible matrix can be factorized into elementary matrices. The only thing that needs to be done is remove $E_{ij}^\lambda$ from the factorization for $\lambda \neq 1$. This can always be done since $E_{ij}^\lambda = D_i^\lambda E_{ij}^1 D_i^{\lambda^{-1}}$.
    \end{proof}
    From now on, we will simplify the notation and use $E_{ij}$ for $E_{ij}^1$.
    \begin{remark}
    It might be interesting to observe how the discussion from this section simplifies in case $\F=\F_2$. In this scenario, $\lambda = 1$ is the unique unit in $\F_2$ and so the elementary matrices $D_i^\lambda$ become redundant. This substantially simplifies most of this section.
    \end{remark}
    
    We now use the factorization from Lemma \ref{lemma:can remove annoying arrows} to construct a two-story complex associated to $C$.

    \subsection{Two-story complexes} As mentioned in the introduction and elaborated on in Section \ref{section:examples}, simultaneously horizontally and vertically simplified bases do not exist for general knot Floer-like complexes. Two-story complexes are a technical tool we invented to simultaneously keep track of two bases; one horizontally simplified and one vertically simplified. They allow us to keep changing the bases in a controlled manner and bring them as close as possible, culminating in a basis for a direct sum of standard complexes and local systems.


\vspace{1em}
    
Let's be precise. A \emph{two-story complex} $\Xi(C)$ associated to $(C, \partial)$ consists of a \emph{bottom floor}, \emph{top floor} and \emph{elevator arrows} connecting the two floors. Its bottom floor is the complex $C/U$ with a distinguished vertically simplified basis $\{x_1, \dots, x_n\}$ and its top floor is the complex $C/V$ with a distinguished horizontally simplified basis $\{y_1, \dots, y_n\}$. The elevator arrows connecting the floors encode the change of basis between the bases $\{y_1, \dots, y_n\}$ and $\{x_1, \dots, x_n\}$.

By the remark following Lemma \ref{lemma:matrix operations}, we can assume that the transition matrix $P$ contains no powers of $U$ or $V$ as coefficients, \emph{i.e.}, $P \in \GL n \F$. This means that all elevator arrows map the bigrading $(i,j)$ in the bottom floor to the bigrading $(i,j)$ in the top floor. A collection of basis elements and elevator arrows in some bigrading is called an \emph{elevator shaft}. 

By Lemma \ref{lemma:can remove annoying arrows}, $P$ factorizes as $P = P_1 \cdots P_k$ where the matrices $P_i \in \GL n \F$ are elementary matrices of the forms $T_{ij}$, $E_{ij}$ and $D_i^\lambda$. For every elementary matrix $T_{ij}$ appearing in the product, add a \emph{crossing} between the elevator arrows corresponding to $x_i$ and $x_j$, for every elementary matrix $E_{ij}$ appearing in the product, add a \emph{crossover arrow} from $x_i$ to $x_j$ and for every elementary matrix $D_i^\lambda$ appearing in the product, add a \emph{black dot} with decoration $\lambda$ to the elevator arrow corresponding to $x_i$. Do this in order so that the graphical element corresponding to $P_1$ is at the bottom and the graphical element corresponding to $P_k$ is at the top. See Figure \ref{fig:encoding the matrix as arrows} for an example of this construction and Figure \ref{fig:two-story complex} for a schematic depiction of a two-story complex. 

\begin{figure}[t]
    \begin{tikzpicture}[scale=1.0]
        \def\a{0.1}
        \draw[-implies,double equal sign distance] (2,5.5) to (1,5.5);
        \draw[-implies,double equal sign distance] (1,2.5) to (3,2.5);
        \draw[-implies,double equal sign distance] (3,0.5) to (2,0.5);

        \filldraw[] (2,4.5) circle (2pt) node[anchor = east]{$2$};
        
        \draw[very thick] (1,3) to[out=90, in=-90] (2,4);
        \draw[very thick] (2,3) to[out=90, in=-90] (1,4);
        \draw[very thick] (2,1) to[out=90, in=-90] (3,2);
        \draw[very thick] (3,1) to[out=90, in=-90] (2,2);
        
        \draw[very thick] (1,0)--(1,3){};
        \draw[very thick] (1,4)--(1,6){};
        \draw[very thick] (2,0)--(2,1){};
        \draw[very thick] (2,2)--(2,3){};
        \draw[very thick] (2,4)--(2,6){};
        \draw[very thick] (3,0)--(3,1){};
        \draw[very thick] (3,2)--(3,6){};

        \filldraw[black] (1,0) circle (2pt) node[anchor=north]{$x_1$};
        \filldraw[black] (2,0) circle (2pt) node[anchor=north]{$x_2$};
        \filldraw[black] (3,0) circle (2pt) node[anchor=north]{$x_3$};
        \filldraw[black] (1,6) circle (2pt) node[anchor=south]{$y_1$};
        \filldraw[black] (2,6) circle (2pt) node[anchor=south]{$y_2$};
        \filldraw[black] (3,6) circle (2pt) node[anchor=south]{$y_3$};
        
        \draw[<->, black, very thick] (-1,2) to (0,2);

        \draw[] (-3,2) node[]{$P=E_{32}T_{23}E_{13}T_{12}D_2^2E_{21}$};
    \end{tikzpicture}
    \caption{Representing the change of basis matrix $P$ in a graphical manner as a sequence of crossings and crossover arrows.}
    \label{fig:encoding the matrix as arrows}
\end{figure}
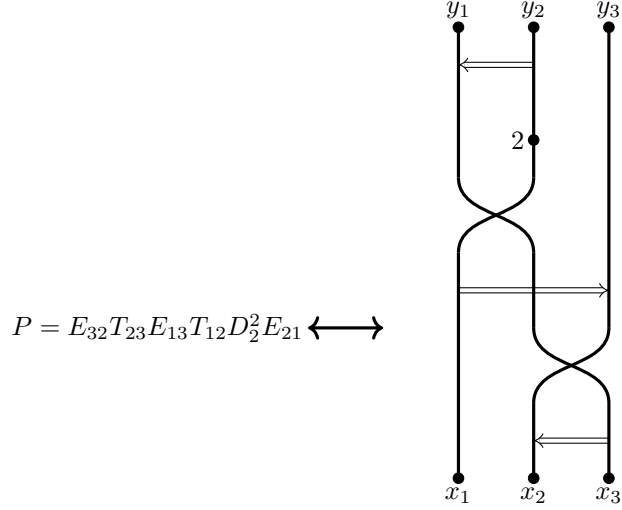

\begin{figure}
    \centering
    \begin{tikzpicture}[scale=1.2]
        \def\a{0.05}
        \def\b{0.6}
        \def\c{0.4}

        
        \draw[thin](0-\a,0+\a)--(0-\a,1+\a){};
        \draw[thin](0+\a,0-\a)--(0+\a,1-\a){};
        \draw[thin](1-\a,0+\a)--(1-\a,1+\a){};
        \draw[thin](1+\a,0-\a)--(1+\a,1-\a){};
        \draw[thin](2-\a,0+\a)--(2-\a,1+\a){};
        \draw[thin](2+\a,0-\a)--(2+\a,1-\a){};
        \draw[thin](3-\a,0+\a)--(3-\a,1+\a){};
        \draw[thin](3+\a,0-\a)--(3+\a,1-\a){};

        \draw[thin](0-\a,0+\a)--(0+\a,1-\a){};

        \draw[very thick](0-\a,1+\a)--(1-\a,1+\a);
        \draw[very thick](0+\a,1-\a)--(1+\a,1-\a);
        \draw[very thick](2-\a,1+\a)--(3-\a,1+\a);
        \draw[very thick](2+\a,1-\a)--(3+\a,1-\a);

        \draw[very thick](0-\a,0+\a)--(0-\a+\c,0+\a+\b){};
        \draw[very thick](2-\a,0+\a)--(2-\a+\c,0+\a+\b){};
        \draw[very thick](0+\a,0-\a)--(0+\a+\c,0-\a+\b){};
        \draw[very thick](2+\a,0-\a)--(2+\a+\c,0-\a+\b){};

        \draw[thin](0-\a+\c,0+\a+\b)--(0-\a+\c,1+\a+\b){};
        \draw[thin](0+\a+\c,0-\a+\b)--(0+\a+\c,1-\a+\b){};
        \draw[thin](2-\a+\c,0+\a+\b)--(2-\a+\c,1+\a+\b){};
        \draw[thin](2+\a+\c,0-\a+\b)--(2+\a+\c,1-\a+\b){};

        \draw[very thick](0-\a+\c,1+\a+\b)--(2-\a+\c,1+\a+\b){};
        \draw[very thick](0+\a+\c,1-\a+\b)--(2+\a+\c,1-\a+\b){};

        \draw[very thick](1-\a,0+\a)--(1-\a-2*\c,0+\a-2*\b){};
        \draw[very thick](1+\a,0-\a)--(1+\a-2*\c,0-\a-2*\b){};
        \draw[very thick](3-\a,0+\a)--(3-\a-2*\c,0+\a-2*\b){};
        \draw[very thick](3+\a,0-\a)--(3+\a-2*\c,0-\a-2*\b){};

        \draw[very thick](1-\a-2*\c,1+\a-2*\b)--(3-\a-2*\c,1+\a-2*\b){};
        \draw[very thick](1+\a-2*\c,1-\a-2*\b)--(3+\a-2*\c,1-\a-2*\b){};

        \draw[thin](1-\a-2*\c,0+\a-2*\b)--(1+\a-2*\c,1-\a-2*\b){}; 
        \draw[thin](1+\a-2*\c,0-\a-2*\b)--(1-\a-2*\c,1+\a-2*\b){}; 
        \draw[thin](3-\a-2*\c,0+\a-2*\b)--(3-\a-2*\c,1+\a-2*\b){};
        \draw[thin](3+\a-2*\c,0-\a-2*\b)--(3+\a-2*\c,1-\a-2*\b){};

        \filldraw[black] (3+\a-2*\c,0-\a-2*\b) circle (1pt) node[anchor=north west]{$y_2=x_2$};
        \filldraw[black] (3+\a-2*\c,1-\a-2*\b) circle (1pt) node[anchor=south west]{$x_2$};
        \filldraw[black] (0+\a+\c,1-\a+\b) circle (1pt) node[anchor=north west]{$x_1$};
        \filldraw[black] (0+\a+\c,0-\a+\b) circle (1pt);
        \node[] at (-0.6,0-\a+\b) {$y_1=x_1$};
    \end{tikzpicture}
    \caption{An example of a two-story complex. The thin arrows are the elevator arrows connecting the top and bottom floors and the thick arrows represent the differential. There are two nontrivial elevator shafts. This is a two-story complex associated to the chain complex in Figure \ref{fig:local system 2}. The generators $x_1$ and $x_2$ are the same as in Figure \ref{fig:local system 2}.}
    \label{fig:two-story complex}
\end{figure}
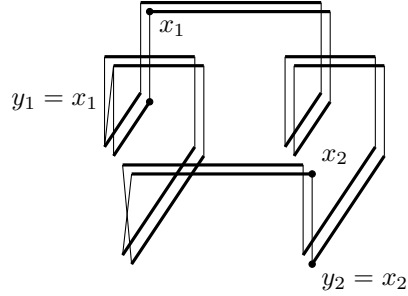

\vspace{1em}

The matrix $P$ can be deduced from its graphical representation by writing down the elementary matrices corresponding to crossings and crossover arrows and multiplying them in the correct order.

While this process might appeal to computers, humans will probably find the following alternative method more practical. The coefficient $p_{ij}$ of the matrix $P$ represents the weighted sum of paths between the generators $x_i$ and $y_j$ where we are allowed to travel along crossover arrows. For example, look at the graphical representation of the matrix 
$$
P=E_{32}T_{23}E_{13}T_{12}D_2^2E_{21}=\begin{pmatrix}
    2&2&1\\
    0&0&1\\
    1&0&1
\end{pmatrix}
$$
over $\F_3$ in Figure \ref{fig:encoding the matrix as arrows}. There is a unique way of travelling upwards from $x_3$ to $y_1$. This corresponds to the matrix coefficient $p_{31}=1$. There is also a unique way of travelling upwards from $x_1$ to $y_1$ by taking the crossover arrow to the left. This way passes through the black dot with decoration $2$ so $p_{11}=2$. A reader can verify that the same is true for the other matrix coefficients.

\vspace{1em}

It is worth emphasizing that the factorization of $P$ into elementary matrices is not unique. We list some sources of non-uniqueness and their counterparts in the world of graphical representations of $P$.
    \begin{enumerate}
        \item $D_i^\lambda D_i^\mu = D_i^{\lambda \mu}$ for any $i \in \{1, \dots, n\}$ and nonzero $\lambda, \mu \in \F$. This corresponds to the fact that two adjacent black dots on the same elevator arrow can be replaced by a single one with a product of their decorations.
        \item $E_{ij}E_{kl} = 
        \begin{cases}
        E_{kl}E_{ij} &| \quad j\neq k \text{ and } i \neq l\\
        E_{il}E_{kl}E_{ij} &| \quad j = k \text{ and } i \neq l\\
        D_j^{-1}E_{kj}D_j^{-1}E_{kl}E_{ij}  &| \quad j \neq k \text{ and } i = l\\
        \end{cases}
        $
        $\quad \ \ \ $ for any $i, j, k, l \in \{1, \dots, n\}$ such that $i,j$ and $k,l$ are distinct. This means that we can freely slide two adjacent crossover arrows past each other if their endpoints are distinct. If the endpoint of one arrow is the same as the starting point of another, we can slide them past each other, but we gain another crossover arrow, which is a composition of the two, and in some cases also two black dots. The only exception is if the crossover arrows are inverses to each other, in which case we have
        
        \noindent $E_{ij}E_{ji} =
        \begin{cases}
        D_2^{\frac{1}{2}}E_{ji}E_{ij}D_1^2 &| \quad \mathrm{char}\ \F \neq 2 \\
        T_{ij}E_{ij} &| \quad \mathrm{char}\ \F =2
        \end{cases}$
        $\quad$ for any distinct $i, j \in \{1, \dots, n\}$. In the first case, we can slide the crossover arrows past each other at the expense of adding two black dots and in the second case we can replace one of them with a crossing.        
        \item $T_{ij}D_i^\lambda = D_j^\lambda T_{ij}$ for all distinct $i, j \in \{1, \dots, n\}$ and nonzero $\lambda \in\F$. This means that a black dot can be slid past a crossing.
        \item $E_{ij} D_i^\lambda = D_j^\lambda D_i^\lambda E_{ij} D_j^{\lambda^{-1}}$ for any distinct $i, j \in \{1, \dots, n\}$ and nonzero $\lambda \in\F$. This corresponds to the fact that a crossover arrow can be moved past a black dot, but we gain two extra black points in the process.
        \item $E_{ij}T_{kl}=
        \begin{cases}
        T_{kl}E_{ij} &| \quad j\neq k \text{ and } i \neq l\\
        T_{kl}E_{il} &| \quad j = k \text{ and } i \neq l\\
        T_{kl}E_{kj} &| \quad j \neq k \text{ and } i = l\\
        E_{ji}D_j^{-1}E_{ij} &| \quad j = k \text{ and } i = l
        \end{cases}
        $
        $\quad \ \ \ $ for any $i, j, k, l \in \{1, \dots, n\}$ such that $i,j$ and $k,l$ are distinct. This means that sliding a crossover arrow past a crossing is usually allowed. The special case $j=k$ and $i=l$ can be interpreted by saying that the crossings are not really needed, since they can always be replaced by some crossover arrows and a black dot.
    \end{enumerate}
    The moves corresponding to matrix identities are shown in Figure \ref{fig:local moves}. In general, moving different graphical elements past each other is allowed when their sets of indices are disjoint. When they are not, moving the elements past each other usually incurs some `cost' that can be described by adding some new graphical elements to the description. The above list is a formalization of this idea.

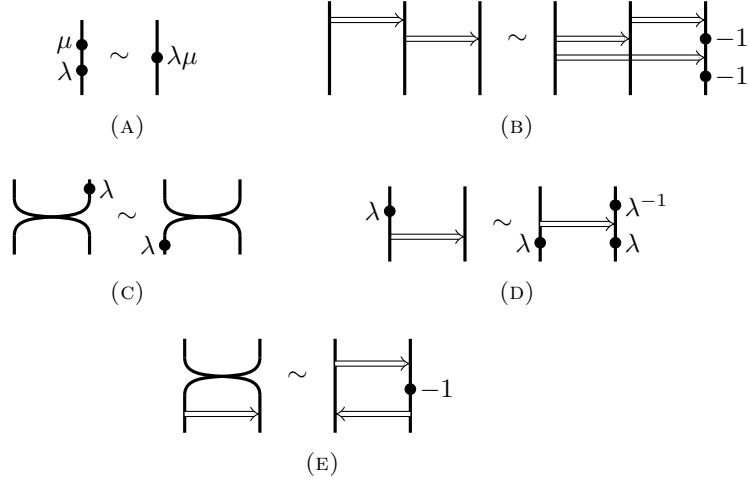
\begin{figure}[t]
    \begin{subfigure}[b]{0.40\textwidth}
        \centering
        \begin{tikzpicture}[scale=1.0]
        \def\a{0.33}
        \draw[very thick](1,1)--(1,0){};
        \draw[very thick](0,1)--(0,0){};
        \filldraw[] (0, \a) circle (2pt) node[anchor = east]{$\lambda$};
        \filldraw[] (0, 1-\a) circle (2pt) node[anchor = east]{$\mu$};
        \filldraw[] (1, 0.5) circle (2pt) node[anchor = west]{$\lambda\mu$};
        \draw[](0.5,0.5) node{$\sim$};
    \end{tikzpicture}
    \caption{}
    \end{subfigure}
    \begin{subfigure}[b]{0.40\textwidth}
        \centering
        \begin{tikzpicture}[scale=1.0]
        \def\a{0.25}
        \draw[-implies,double equal sign distance] (-2,1-\a) to (-1,1-\a);
        \draw[-implies,double equal sign distance] (-1,2*\a) to (0,2*\a);

        \draw[-implies,double equal sign distance] (2,1-\a) to (3,1-\a);
        \draw[-implies,double equal sign distance] (1,2*\a) to (2,2*\a);
        \draw[-implies,double equal sign distance] (1,\a) to (3,\a);

        \filldraw[] (3, 2*\a) circle (2pt) node[anchor = west]{$-1$};
        \filldraw[] (3, 0) circle (2pt) node[anchor = west]{$-1$};

        \draw[very thick](-2,1)--(-2,0-\a){};
        \draw[very thick](1,1)--(1,0-\a){};
        \draw[very thick](2,1)--(2,0-\a){};
        \draw[very thick](3,1)--(3,0-\a){};
        \draw[very thick](-1,1)--(-1,0-\a){};
        \draw[very thick](0,1)--(0,0-\a){};
        \draw[](0.5,0.5) node{$\sim$};
    \end{tikzpicture}
    \caption{}
    \end{subfigure}

    \vspace{1em}

    \begin{subfigure}[b]{0.40\textwidth}
        \centering
        \begin{tikzpicture}[scale=1.0]
        \def\a{0.25}
        \draw[very thick] (-1,1-\a) to[out=-90, in=90] (0,\a);
        \draw[very thick] (0, 1-\a) to[out=-90, in=90] (-1,\a);
        \draw[very thick](-1,1)--(-1,1-\a){};
        \draw[very thick](-1,\a)--(-1,0){};
        \draw[very thick](0,1)--(0,1-\a){};
        \draw[very thick](0,\a)--(0,0){};

        \filldraw[] (0,1-\a/2) circle (2pt) node[anchor = west]{$\lambda$};

        \draw[](0.5,0.5) node{$\sim$};
        
        \draw[very thick] (1,1-\a) to[out=-90, in=90] (2,\a);
        \draw[very thick] (2, 1-\a) to[out=-90, in=90] (1,\a);
        \draw[very thick](1,1)--(1,1-\a){};
        \draw[very thick](1,\a)--(1,0){};
        \draw[very thick](2,1)--(2,1-\a){};
        \draw[very thick](2,\a)--(2,0){};
        \filldraw[] (1,\a/2) circle (2pt) node[anchor = east]{$\lambda$};
    \end{tikzpicture}
    \caption{}
    \end{subfigure}
    \begin{subfigure}[b]{0.40\textwidth}
        \centering
        \begin{tikzpicture}[scale=1.0]
        \def\a{0.33}
        \def\b{0.25}
        \filldraw[] (-1,1-\a) circle (2pt) node[anchor=east]{$\lambda$};
        \draw[-implies,double equal sign distance] (-1,\a) to (0,\a);
        \draw[very thick](1,1)--(1,0){};
        \draw[very thick](2,1)--(2,0){};
        \draw[very thick](-1,1)--(-1,0){};
        \draw[very thick](0,1)--(0,0){};
        \draw[](0.5,0.5) node{$\sim$};

        \draw[-implies,double equal sign distance] (1,0.5) to (2,0.5);
        \filldraw[] (2,1-\b) circle (2pt) node[anchor=west]{$\lambda^{-1}$};
        \filldraw[] (2,\b) circle (2pt) node[anchor=west]{$\lambda$};
        \filldraw[] (1,\b) circle (2pt) node[anchor=east]{$\lambda$};
    \end{tikzpicture}
    \caption{}
    \end{subfigure}

    \vspace{1em}

    \begin{subfigure}[b]{0.40\textwidth}
        \centering
        \begin{tikzpicture}[scale=1.0]
        \def\a{0.25}
        \def\b{0.33}
        \draw[very thick] (-1,1-\a) to[out=-90, in=90] (0,\a);
        \draw[very thick] (0, 1-\a) to[out=-90, in=90] (-1,\a);
        \draw[very thick](-1,1)--(-1,1-\a){};
        \draw[very thick](-1,\a)--(-1,-\a){};
        \draw[very thick](0,1)--(0,1-\a){};
        \draw[very thick](0,\a)--(0,-\a){};

        \draw[-implies,double equal sign distance] (-1,0) to (0,0);

        \draw[](0.5,0.5) node{$\sim$};
        \draw[very thick](1,1)--(1,-\a){};
        \draw[very thick](2,1)--(2,-\a){};
        \filldraw[] (2,\b) circle (2pt) node[anchor = west]{$-1$};
        \draw[-implies,double equal sign distance] (2,0) to (1,0);
        \draw[-implies,double equal sign distance] (1,1-\b) to (2,1-\b);
    \end{tikzpicture}
    \caption{}
    \label{fig:local move e}
    \end{subfigure}
    
    \caption{Different expressions of $P$ as a product of elementary matrices correspond to different graphical representations of $P$, some of which can be related by the depicted moves.}
    \label{fig:local moves}
\end{figure}

\vspace{1em}

Let us summarize our discussion so far. A two-story complex is a complex that is vertically simplified on the bottom floor, horizontally simplified on the top floor, and its elevator arrows encode the change of basis matrix between the two bases in terms of crossings, crossover arrows and decorated black dots. In the rest of the proof, we will find suitable basis changes to remove all crossover arrows that eventually diverge if we slide them along a two-story complex.

\subsection{Removing a single crossover arrow}
In the previous subsection we defined two-story complexes as a technical tool that keeps track of two different bases with the ultimate goal of bringing them `closer' together. A natural measure of closeness one could use is the number of crossover arrows in the two-story complex. While this is not exactly what we choose to optimize for, knowing how to remove crossover arrows will be an integral part of our algorithm. This subsection explains how a sequence of certain simple basis changes on both floors can be used to remove a single crossover arrow between nonparallel complexes while keeping the bases horizontally and vertically simplified. The proofs resemble the arrow sliding algorithm of \cite{hanselman2016bordered} and we elaborate on this remark in Section \ref{section:immersed curves proof}.

\vspace{1em}

Recall that $\{x_1, \dots, x_n\}$ is a vertically simplified basis on the bottom floor and $\{y_1, \dots, y_n\}$ is a horizontally simplified basis on the top floor. The results from this section will usually apply symmetrically to both the top and the bottom floor. To compress notation, we will write $z \in \{x, y\}$ and let $z_i$ denote either $x_i$ or $y_i$.

\begin{lemma}\label{lemma:remove black dots}
Let $z\in\{x,y\}$. Sliding a black dot with a decoration $\lambda$ over $z_i$ corresponds to a change of basis $z_i \mapsto \lambda^{-1}z_i$.
\end{lemma}
\begin{proof}
From the perspective of elevator arrows, this is akin to just removing a black dot. This corresponds to a matrix multiplication by $D_i^{\lambda^{-1}}$ and thus it corresponds to a change of basis $z_i \mapsto \lambda^{-1}z_i$.
\end{proof}
Note that such basis changes preserve the fact that the basis is vertically or horizontally simplified, since the vertical or horizontal isomorphism changes only by multiplication by $D_i^\lambda$. As we shall now see, this is not necessarily the case for the crossover arrows. 

Let $z \in \{x, y\}$. When we say that we slide a crossover arrow through $(z_i, z_j)$, this means that the arrow is pointing from the basis element $z_i$ to the basis element $z_j$.
\begin{lemma}\label{lemma:sliding basis changes}
Let $z \in \{x, y\}$. Sliding a crossover arrow $A$ through $(z_i, z_j)$ corresponds to a change of basis $z_i \mapsto z_i\pm z_j$, where the sign depends on whether $A$ is moving upwards or downwards.
\end{lemma}
\begin{proof}
Letting $E$, $T$, and $B$ denote the elevator shaft, the top floor and the bottom floor respectively, there are $4$ different slides that can happen: $B \rightsquigarrow E$, $E \rightsquigarrow T$, $T \rightsquigarrow E$ and $E \rightsquigarrow B$. From the perspective of the elevator shaft, each of these slides is akin to just removing or adding $A$, sometimes near the bottom and sometimes near the top floor. This corresponds to a matrix multiplication by $E_{ij}$ or $E_{ij}^{-1}$, sometimes from the left and sometimes from the right, which in turn corresponds to the basis changes $z_i \mapsto z_i + z_j$ and $z_i \mapsto z_i - z_j$, sometimes with $z=y$ and sometimes with $z=x$. To be precise, the slides $B \rightsquigarrow E$ and $E \rightsquigarrow T$ correspond to the basis changes $z_i \mapsto z_i + z_j$ and the slides $T \rightsquigarrow E$ and $E \rightsquigarrow B$ correspond to the basis changes $z_i \mapsto z_i -z_j$.
\end{proof}
This means that we can move a crossover arrow to the top or bottom floor with the unfortunate caveat that in general the corresponding basis ceases to be simplified as a result. However, this can sometimes be remedied. In some cases, there is another set of basis changes that can remove the crossover arrows from the top and bottom floors. In order to explain when, we recall the unusual order on $\Z$:
$$ -1 \leq^! -2 \leq^! -3 \leq^! \dots \leq^! 0 \leq^! \dots \leq^! 3 \leq^! 2 \leq^! 1.$$
We can extend it from $\Z$ to a total order on the set $\Z^\N$ of integer sequences by ordering them lexicographically with respect to $\leq^!$. There is also an associated strict total order $<^!$ on $\Z^\N$ where $s_1  <^! s_2$ if $s_1 \leq^! s_2$ and $s_1 \neq s_2$.

\vspace{1em}

Moreover, the construction inspires a family of other relations on $\Z^\N$. For any $m \in \N$ and sequences $s_1, s_2 \in \Z^\N$ define $s_1 \leq^!_m s_2$ by lexicographically comparing only the first $m$ terms of $s_1$ and $s_2$. Note that $\leq^!_m$ is not a total order since it is not antisymmetric and this additional flexibility will be useful in Lemma \ref{lemma:the hard lemma}.

\vspace{1em}

We now define the sequences we will study using the above orders. Starting at each generator $z_i$ for $z\in\{x, y\}$, we travel along the complex $z_i$ is on and keep track of the lengths and directions of the horizontal and vertical arrows we traverse. Let us be precise. The direction of movement is chosen so that the first traversed arrow is in the floor we started in. We then record the following.
\begin{enumerate}
    \item Every time we move along a horizontal arrow of length $l$, we record
    \begin{enumerate}
        \item $l$ if we move to the right and
        \item $-l$ if we move to the left.
    \end{enumerate}
    \item Every time we move along a vertical arrow of length $l$, we record
    \begin{enumerate}
        \item $l$ if we move up and
        \item $-l$ if we move down.
    \end{enumerate}
    \item Every time we move along an elevator shaft, we do not record anything.  When we travel along an elevator shaft, we do not use the crossover arrows, but instead remain on the main elevator arrow.
\end{enumerate}
This constructs a sequence of integers $a(z_i)$ associated to each generator $z_i$. If the path eventually returns to $z_i$, the sequence is periodic and if it does not, then we still treat it as an infinite sequence with infinitely many zeros appended at the end. There is a similar sequence $b(z_i)$ obtained by following the same set of recording rules, but moving in the opposite direction, \emph{i.e.}, so that the first traversed arrow is an elevator arrow.

\begin{example}
    The sequences associated to the generators $x_1$ and $y_1$ in the two-story complex depicted in Figure \ref{fig:two-story complex} are $a(x_1) = (2, -1, 1, -2, -2, 2, -1, 1, \dots)$, $b(x_1) = (-1, 1, -2, 2, 2, -1, 1, -2, \dots)$, $a(y_1) = b(x_1)$ and $b(y_1) = a(x_1)$.
\end{example}

\begin{lemma}\label{lemma:removing arrows}
    Let $\Xi(C)$ be a two-story complex with a single crossover arrow $A$ connecting the elevator arrows next to the basis elements $z_i$ and $z_j$ where $z \in \{x, y\}$. Then $A$ can be slid through $(z_i, z_j)$ and removed if $a(z_i) <^! a(z_j)$.
\end{lemma}
\begin{proof}
    The black dots in $\Xi(C)$ do not play a role in the proof, since they can be removed at the beginning according to Lemma \ref{lemma:remove black dots} with the added observation that this does not change $a(z)$ for any basis element $z$. The crossings may be present, but $A$ slides past them as explained in our list of local moves.
    
    Let now $m$ be the least integer such that the $m^\text{th}$ terms of the sequences $a(z_i)$ and $a(z_j)$ differ. This means that the complexes between which $A$ lies remain parallel for $m-1$ turns. We can slide the crossover arrow until that point, performing basis changes of the form $z_i' \mapsto z_i' \pm z_j'$ every time we slide through $(z_i', z_j')$ as in Lemma \ref{lemma:sliding basis changes}. Every turn consists of two basis changes (since each horizontal and vertical arrow has two endpoints) and note that after every even number of basis changes, the bases of the top and bottom floor remain horizontally and vertically simplified respectively. Consider now the point after the $(2m-1)^\text{th}$ basis change, which is the last one we can do since the complexes diverge afterwards. Assume first that this happens near the top floor of the complex. Since $a(z_i)_m <^! a(z_j)_m$, we have one of the following two cases.
    \begin{enumerate}
        \item If $a(z_j)_m > 0 > a(z_i)_m$, then the last basis change removed the crossover arrow $A$ while keeping the basis of the top floor horizontally simplified. As such, there is nothing that needs to be done. See Figure \ref{fig:removing an arrow, easy case}.
        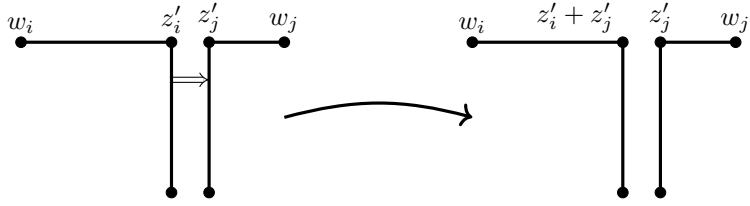
\begin{figure}[t]
        \begin{tikzpicture}[scale=1.0]
            \def\a{6}
            \def\b{0.5}
            \draw[-implies,double equal sign distance] (1+\b,1.5) to (2,1.5);
        
            \draw[very thick](1+\b,0)--(1+\b,2)--(-1+\b,2){};
            \draw[very thick](2,0)--(2,2)--(3,2){};
        
            \filldraw[black] (1+\b,2) circle (2pt) node[anchor=south]{$z_i'$};
            \filldraw[black] (2,2) circle (2pt) node[anchor=south]{$z_j'$};
            \filldraw[black] (1+\b,0) circle (2pt) node[anchor=south]{};
            \filldraw[black] (2,0) circle (2pt) node[anchor=south]{};
            \filldraw[black] (-1+\b,2) circle (2pt) node[anchor=south]{$w_i$};
            \filldraw[black] (3,2) circle (2pt) node[anchor=south]{$w_j$};
        
            \draw[very thick](1+\a+\b,0)--(1+\a+\b,2)--(-1+\a+\b,2){};
            \draw[very thick](2+\a,0)--(2+\a,2)--(3+\a,2){};

            \draw[->, very thick] (3,1) to[out=15, in=165] (5.5,1) {};
        
            \filldraw[black] (1+\a+\b,2) circle (2pt) node[anchor=south east]{$z_i'+z_j'$};
            \filldraw[black] (2+\a,2) circle (2pt) node[anchor=south]{$z_j'$};
            \filldraw[black] (1+\a+\b,0) circle (2pt) node[anchor=south]{};
            \filldraw[black] (2+\a,0) circle (2pt) node[anchor=south]{};
            \filldraw[black] (-1+\a+\b,2) circle (2pt) node[anchor=south]{$w_i$};
            \filldraw[black] (3+\a,2) circle (2pt) node[anchor=south]{$w_j$};
        \end{tikzpicture}
        \caption{The picture illustrates case (1) in the proof of Lemma \ref{lemma:removing arrows} when the last basis change is near the top. On the left, a situation after $m-1$ turns, just before the complexes diverge, is depicted. Since $a(z_j)_m > 0 > a(z_i)_m$, the basis change $z_i' \mapsto z_i'+z_j'$ removes the crossover arrow while keeping the top basis horizontally simplified.}
        \label{fig:removing an arrow, easy case}
        \end{figure}
        
        \item Otherwise we have $a(z_i)_m > a(z_i)_m$ where both terms have the same sign. Irrespective of which case we encounter, we perform the basis change $w_i \mapsto w_i + U^{a(z_i)_m-a(z_j)_m}w_j$ where $(w_i, w_j)$ would be the next pair of generators the crossover arrow $A$ would be slid through if the complexes remained parallel. This change of basis was chosen so that the crossover arrow can be removed and the top basis remains horizontally simplified. See Figure \ref{fig:removing an arrow, hard case} for a visual depiction of these basis changes.

    \begin{figure}[t]
    \begin{tikzpicture}[scale=1.0]
        \def\a{6}
        \def\b{0.5}
        \def\c{0.2}
        \draw[-implies,double equal sign distance] (1+\b,1.5) to (2,1.5);
        
        \draw[very thick](1+\b,0)--(1+\b,2+\c)--(3.5+\b,2+\c){};
        \draw[very thick](2,0)--(2,2)--(3,2){};
        
        \filldraw[black] (1+\b,2+\c) circle (2pt) node[anchor=south]{$z_i'$};
        \filldraw[black] (2,2) circle (2pt) node[anchor=north west]{$z_j'$};
        \filldraw[black] (1+\b,0) circle (2pt) node[anchor=south]{};
        \filldraw[black] (2,0) circle (2pt) node[anchor=south]{};
        \filldraw[black] (3.5+\b,2+\c) circle (2pt) node[anchor=south]{$w_i$};
        \filldraw[black] (3,2) circle (2pt) node[anchor=north west]{$w_j$};
        
        \draw[very thick](1+\b+\a,0)--(1+\b+\a,2+\c)--(3.5+\b+\a,2+\c){};
        \draw[very thick](2+\a,0)--(2+\a,2)--(3+\a,2){};
        
        \filldraw[black] (1+\b+\a,2+\c) circle (2pt) node[anchor=south east]{$z_i'+z_j'$};
        \filldraw[black] (2+\a,2) circle (2pt) node[anchor=north west]{$z_j'$};
        \filldraw[black] (1+\b+\a,0) circle (2pt) node[anchor=south]{};
        \filldraw[black] (2+\a,0) circle (2pt) node[anchor=south]{};
        \filldraw[black] (3.5+\b+\a,2+\c) circle (2pt) node[anchor=south]{$w_i + U^{a(z_i)_m-a(z_j)_m}w_j$};
        \filldraw[black] (3+\a,2) circle (2pt) node[anchor=north west]{$w_j$};

        \draw[->, very thick] (4,1) to[out=15, in=165] (6.5,1) {};
    \end{tikzpicture}
    \caption{The picture illustrates case (2) in the proof of Lemma \ref{lemma:removing arrows} when the last basis change is near the top. On the left, a situation after $m-1$ turns, just before the complexes diverge, is depicted. Since $a(z_i)_m > a(z_i)_m$, the basis change $z_i' \mapsto z_i'+z_j'$ and $w_i \mapsto w_i + U^{a(z_i)_m-a(z_j)_m}w_j$ removes the crossover arrow while keeping the top basis horizontally simplified. Performing the corresponding basis change on the bottom floor in bigrading $\gr(w_i)$ guarantees that the elevator arrows in this bigrading do not change.}
    \label{fig:removing an arrow, hard case}
    \end{figure}
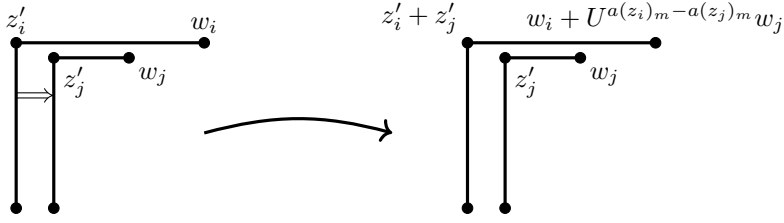
        
    We should now also make the corresponding basis change on the bottom floor. Observe that the basis of the bottom floor will remain vertically simplified since it is still the same mod $U$. This basis change guarantees that the elevator arrows in bigrading $\gr(w_i)$ stay exactly the same, something that would not necessarily hold if we perform a basis change only on the top floor.
    \end{enumerate}
    In both cases, the net effect of the process is that the crossover arrow $A$ was removed, which is what we wanted to prove. The cases (1) and (2) are very similar if the complexes diverge near the bottom floor; in all basis changes we need to replace $+$ with $-$ in accordance with Lemma \ref{lemma:sliding basis changes}.
\end{proof}
The setting of the lemma we just proved was somewhat simplistic, since there was only one crossover arrow in an entire two-story complex $\Xi(C)$ and it was pointing in a convenient direction. We now prove that these assumptions were not necessary.

\vspace{1em}

Let $A$ be a crossover arrow connecting the elevator arrows next to the basis elements $z_i$ and $z_j$. We say that $A$ connects nonparallel complexes if $a(z_i) \neq a(z_j)$.

\begin{lemma}\label{lemma:one crossover arrow can always be removed}
Let $\Xi(C)$ be a two-story complex and let $A$ be a crossover arrow in $\Xi(C)$ connecting nonparallel complexes. Then $A$ can be removed. 
\end{lemma}
\begin{proof}
Consider the bigrading of $A$ and begin by relabelling and rearranging the basis elements $x_1, \dots, x_n$ and $y_1, \dots, y_n$ in this bigrading such that $a(x_1) \geq^! \dots \geq^! a(x_n)$, $a(y_1) \leq^! \dots \leq^! a(y_n)$ and they are arranged in an increasing order of indices from left to right. This process in general introduces new crossings to $\Xi(C)$.

Let $L$ and $R$ be the elevator arrows adjacent to $A$. We will show by induction that we can in fact remove all crossover arrows between $L$ and $R$. The base case is the scenario in which $A$ is the only crossover arrow between $L$ and $R$. Slide $A$ to the top of the elevator shaft, passing any other crossover arrows as depicted in the picture.
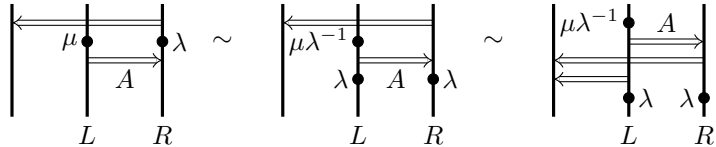
\begin{center}
\begin{tikzpicture}[scale=1.0]
    \def\a{0.25}
    \def\b{0.33}
    \def\c{0.3}

    \draw[-implies,double equal sign distance] (1-2*\c,\a) -- (2-2*\c,\a) node[midway, below]{$A$};
    \draw[-implies,double equal sign distance] (2-2*\c,1-\a) -- (0-2*\c,1-\a);

    \draw[-implies,double equal sign distance] (4,\a) -- (5,\a) node[midway, below]{$A$};
    \draw[-implies,double equal sign distance] (5,1-\a) -- (3,1-\a);

    \draw[-implies,double equal sign distance] (7+2*\c,0.5) -- (8+2*\c,0.5) node[midway, above]{$A$};
    \draw[-implies,double equal sign distance] (8+2*\c,\a) -- (6+2*\c,\a);
    \draw[-implies,double equal sign distance] (7+2*\c,0) -- (6+2*\c,0);
        
    \draw[very thick] (0-2*\c,0-2*\a) -- (0-2*\c,1) {};
    \draw[very thick] (1-2*\c,0-2*\a) -- (1-2*\c,1) {};
    \draw[very thick] (2-2*\c,0-2*\a) -- (2-2*\c,1) {};
    \draw[very thick] (3,0-2*\a) -- (3,1) {};
    \draw[very thick] (4,0-2*\a) -- (4,1) {};
    \draw[very thick] (5,0-2*\a) -- (5,1) {};
    \draw[very thick] (6+2*\c,0-2*\a) -- (6+2*\c,1) {};
    \draw[very thick] (7+2*\c,0-2*\a) -- (7+2*\c,1) {};
    \draw[very thick] (8+2*\c,0-2*\a) -- (8+2*\c,1) {};

    \node[] at (1-2*\c, 0-2*\a) [anchor = north]{$L$};
    \node[] at (2-2*\c, 0-2*\a) [anchor = north]{$R$};

    \node[] at (4, 0-2*\a) [anchor = north]{$L$};
    \node[] at (5, 0-2*\a) [anchor = north]{$R$};

    \node[] at (7+2*\c, 0-2*\a) [anchor = north]{$L$};
    \node[] at (8+2*\c, 0-2*\a) [anchor = north]{$R$};

    \draw[](2.5-\c,0.5) node{$\sim$};
    \draw[](5.5+\c,0.5) node{$\sim$};
    
    \filldraw[] (2-2*\c,0.5) circle (2pt) node[anchor = west]{$\lambda$};
    \filldraw[] (1-2*\c,0.5) circle (2pt) node[anchor = east]{$\mu$};
    
    \filldraw[] (4,0.5) circle (2pt) node[anchor = east]{$\mu\lambda^{-1}$};
    \filldraw[] (5,0) circle (2pt) node[anchor = west]{$\lambda$};
    \filldraw[] (4,0) circle (2pt) node[anchor = east]{$\lambda$};
    
    \filldraw[] (7+2*\c,1-\a) circle (2pt) node[anchor = east]{$\mu\lambda^{-1}$};
    \filldraw[] (7+2*\c,-\a) circle (2pt) node[anchor = west]{$\lambda$};
    \filldraw[] (8+2*\c,-\a) circle (2pt) node[anchor = east]{$\lambda$};   
\end{tikzpicture}
\end{center}
Let us describe this in words. We would like to slide $A$ past the long crossover arrow that shares an endpoint with $A$. To do this, we would like to use the move (2) from our list of local moves. However, we cannot do this immediately because of the potential presence of black dots between the two arrows. This is what the leftmost picture illustrates. We begin with removing the black dot on the elevator arrow adjacent to both crossover arrows by sliding it past $A$. This is depicted in the middle picture and now allows us to slide the other black dot and $A$ past the long crossover arrow. This produces a new crossover arrow as in the rightmost picture, but it is adjacent to a different set of elevator arrows as $A$, so it will not cause problems. Therefore, we have successfully slid $A$ past another crossover arrow without introducing new crossover arrows between $L$ and $R$. The same strategy can be applied to all other cases in which the crossover arrow we are trying to pass shares an endpoint with $A$. The cases where the arrows do not share endpoints are easier since we can just slide them past each other.

\vspace{1em}

This lets us slide $A$ to the top of the elevator shaft. If $A$ points to the right, then we can remove it by Lemma \ref{lemma:removing arrows}. Otherwise we slide $A$ to the bottom of the elevator shaft in an analogous manner and remove it by Lemma \ref{lemma:removing arrows} if it points to the left. If neither of these work, it means that $A$ points to the left at the top of the elevator shaft and to the right at the bottom of the elevator shaft, which means there is an odd number of crossings between $L$ and $R$. We slide $A$ to the nearest such crossing and apply the local move from Figure \ref{fig:local move e}. 
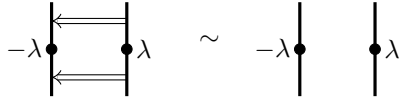
\begin{claim}\label{claim:technical}
This local move \emph{changes} $L$ and $R$ and as such it changes the sequences $a(z_i)$ too. However, it does not change their order with respect to $\leq^!$, \emph{i.e.}, if $a(z_i) \leq^! a(z_j)$, then $a'(z_i) \leq^! a'(z_j)$ where $a'$ denotes the sequence after the change.
\end{claim}
\begin{proof}
If the sequences $a(z_i)$ and $a(z_j)$ diverge before reaching the same elevator shaft again, then the local move cannot have had an effect on their relative order. Otherwise, they both reach the same elevator shaft before diverging. At this point, at most one of them is a part of the new $L$ or $R$, since the original $A$ lied between nonparallel strands. Since $a(z_i) <^! a(z_j)$ and the strands $L$ and $R$ are adjacent, it follows that we must have $a'(z_i) <^! a'(z_j)$ too.
\end{proof}
Therefore, the local move from Figure \ref{fig:local move e} replaces $A$ with two crossover arrows, but they are oriented correctly, \emph{i.e.} we can slide the higher one to the top of the elevator shaft and remove it and we can slide the lower one to the bottom of the elevator shaft and remove it.

\vspace{1em}

The general case is now relatively straightforward. Consider the topmost crossover arrow between $L$ and $R$. As before slide it to the top of the elevator shaft and remove it by Lemma \ref{lemma:removing arrows} if it points to the left. If this is possible, we are now done by the inductive hypothesis. Otherwise, we slide the crossover arrow in the other direction until it is adjacent to another crossover arrow between $L$ and $R$. Two cases arise depending on whether the crossover arrows point in the same or in different directions. If they point in the same direction, we perform the following replacement
\begin{center}
\begin{tikzpicture}[scale=1.0]
    \def\a{0.25}
    \def\b{0.33}
    \def\c{0.3}

    \draw[-implies,double equal sign distance] (2-2*\c,1-4*\a) -- (1-2*\c,1-4*\a);
    \draw[-implies,double equal sign distance] (2-2*\c,1-\a) -- (1-2*\c,1-\a);
    \draw[-implies,double equal sign distance] (5-\c,1-5*\a/2) -- (4-\c,1-5*\a/2);       
    
    \draw[very thick] (1-2*\c,0-\a) -- (1-2*\c,1) {};
    \draw[very thick] (2-2*\c,0-\a) -- (2-2*\c,1) {};
    \draw[very thick] (4-\c,0-\a) -- (4-\c,1) {};
    \draw[very thick] (5-\c,0-\a) -- (5-\c,1) {};

    \draw[](3-3*\c/2,1-5*\a/2) node{$\sim$};
    
    \filldraw[] (2-2*\c,1-5*\a/2) circle (2pt) node[anchor = west]{$\lambda$};
    \filldraw[] (1-2*\c,1-5*\a/2) circle (2pt) node[anchor = east]{$\mu$};
    \filldraw[] (4-\c,0) circle (2pt) node[anchor = east]{$\frac{\lambda\mu}{\lambda+\mu}$};
    \filldraw[] (5-\c,0) circle (2pt) node[anchor = west]{$\lambda$};
    \filldraw[] (4-\c,1-\a) circle (2pt) node[anchor = east]{$\frac{\lambda+\mu}{\lambda}$}; 
\end{tikzpicture}
\end{center}
if $\lambda+\mu \neq 0$. This decreases the number of crossover arrows and means that we are done by the inductive hypothesis. If $\lambda+\mu=0$, then we can just remove them since
\begin{center}
\begin{tikzpicture}[scale=1.0]
    \def\a{0.25}
    \def\b{0.33}
    \def\c{0.3}

    \draw[-implies,double equal sign distance] (2-2*\c,1-4*\a) -- (1-2*\c,1-4*\a);
    \draw[-implies,double equal sign distance] (2-2*\c,1-\a) -- (1-2*\c,1-\a);    
    
    \draw[very thick] (1-2*\c,0-\a) -- (1-2*\c,1) {};
    \draw[very thick] (2-2*\c,0-\a) -- (2-2*\c,1) {};
    \draw[very thick] (4-\c,0-\a) -- (4-\c,1) {};
    \draw[very thick] (5-\c,0-\a) -- (5-\c,1) {};

    \draw[](2.5,0.5) node{$\sim$};
    
    \filldraw[] (2-2*\c,1-5/2*\a) circle (2pt) node[anchor = west]{$\lambda$};
    \filldraw[] (1-2*\c,1-5/2*\a) circle (2pt) node[anchor = east]{$-\lambda$};
    \filldraw[] (4-\c,1-5/2*\a) circle (2pt) node[anchor = east]{$-\lambda$};
    \filldraw[] (5-\c,1-5/2*\a) circle (2pt) node[anchor = west]{$\lambda$};
\end{tikzpicture}
\end{center}
and we are once again done by the inductive hypothesis. Finally, if the crossover arrows point in the opposite directions, then we perform the replacement
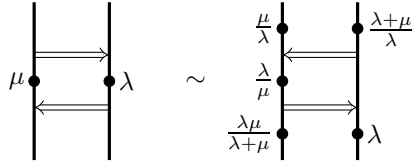
\begin{center}
\begin{tikzpicture}[scale=1.0]
    \def\a{0.35}
    \def\b{0.33}
    \def\c{0.3}

    \draw[-implies,double equal sign distance] (2-2*\c,1-4*\a) -- (1-2*\c,1-4*\a);
    \draw[-implies,double equal sign distance] (1-2*\c,1-2*\a) -- (2-2*\c,1-2*\a);
    \draw[-implies,double equal sign distance] (5-\c,1-2*\a) -- (4-\c,1-2*\a);       
    \draw[-implies,double equal sign distance] (4-\c,1-4*\a) -- (5-\c,1-4*\a);       
    
    \draw[very thick] (1-2*\c,1-6*\a) -- (1-2*\c,1) {};
    \draw[very thick] (2-2*\c,1-6*\a) -- (2-2*\c,1) {};
    \draw[very thick] (4-\c,1-6*\a) -- (4-\c,1) {};
    \draw[very thick] (5-\c,1-6*\a) -- (5-\c,1) {};

    \draw[](3-3*\c/2,1-3*\a) node{$\sim$};
    
    \filldraw[] (2-2*\c,1-3*\a) circle (2pt) node[anchor = west]{$\lambda$};
    \filldraw[] (1-2*\c,1-3*\a) circle (2pt) node[anchor = east]{$\mu$};
    \filldraw[] (4-\c,1-\a) circle (2pt) node[anchor = east]{$\frac{\mu}{\lambda}$};
    \filldraw[] (5-\c,1-\a) circle (2pt) node[anchor = west]{$\frac{\lambda+\mu}{\lambda}$};
    \filldraw[] (4-\c,1-3*\a) circle (2pt) node[anchor = east]{$\frac{\lambda}{\mu}$};
    \filldraw[] (4-\c,1-5*\a) circle (2pt) node[anchor = east]{$\frac{\lambda\mu}{\lambda+\mu}$}; 
    \filldraw[] (5-\c,1-5*\a) circle (2pt) node[anchor = west]{$\lambda$}; 
\end{tikzpicture} 
\end{center}
if $\lambda+\mu\neq 0$. While this does not decrease the number of crossover arrows between these elevator arrows, it swaps the orientation of the topmost arrow. This means that it can now be slid to the top of the elevator shaft and removed by Lemma \ref{lemma:removing arrows}. The last case is when $\lambda+\mu=0$ and the crossover arrows are pointing in different directions. We have the equivalence
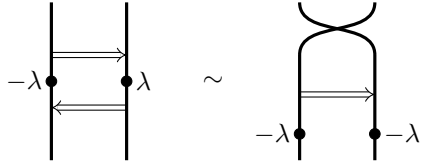
\begin{center}
\begin{tikzpicture}[scale=1.0]
    \def\a{0.35}
    \def\b{0.33}
    \def\c{0.3}

    \draw[-implies,double equal sign distance] (2-2*\c,1-4*\a) -- (1-2*\c,1-4*\a);
    \draw[-implies,double equal sign distance] (1-2*\c,1-2*\a) -- (2-2*\c,1-2*\a);       
    \draw[-implies,double equal sign distance] (4-\c,1-7/2*\a) -- (5-\c,1-7/2*\a);

    \draw[very thick] (4-\c, 1-2*\a) to[out = 90, in = -90] (5-\c, 1) {};
    \draw[very thick] (5-\c, 1-2*\a) to[out = 90, in = -90] (4-\c, 1) {};
    
    \draw[very thick] (1-2*\c,1-6*\a) -- (1-2*\c,1) {};
    \draw[very thick] (2-2*\c,1-6*\a) -- (2-2*\c,1) {};
    \draw[very thick] (4-\c,1-6*\a) -- (4-\c,1-2*\a) {};
    \draw[very thick] (5-\c,1-6*\a) -- (5-\c,1-2*\a) {};

    \draw[](3-3*\c/2,1-3*\a) node{$\sim$};
    
    \filldraw[] (4-\c,1-5*\a) circle (2pt) node[anchor = east]{$-\lambda$}; 
    \filldraw[] (5-\c,1-5*\a) circle (2pt) node[anchor = west]{$-\lambda$}; 
    \filldraw[] (1-2*\c,1-3*\a) circle (2pt) node[anchor = east]{$-\lambda$}; 
    \filldraw[] (2-2*\c,1-3*\a) circle (2pt) node[anchor = west]{$\lambda$}; 
\end{tikzpicture} 
\end{center}
so the number of crossover arrows decreases and the lemma follows by Claim \ref{claim:technical} and the inductive hypothesis. Note that in all of the cases $\lambda$ and $\mu$ are allowed to equal $1$, which is the same as deleting a black dot.
\end{proof}
In this subsection we showed how a single crossover arrow between nonparallel complexes can be slid around the two-story complex and removed. However, during this procedure other crossover arrows might be encountered and passing some of them might introduce new crossover arrows to the two-story complex. This might seem counterproductive. However, we show in the next subsection that the crossover arrows between nonparallel complexes can actually be removed in a controlled manner so that none remains after finitely many steps.

\subsection{Arrow sliding algorithm}\label{subsection:arrow sliding algorithm} In this section, we generalize the arrow sliding algorithm of \cite{hanselman2016bordered} and prove that all crossover arrows connecting nonparallel complexes can be systematically removed.

\vspace{1em}

Let us restate the ambient setting for clarity. There is a two-story complex $\Xi(C)$ in which a factorization of the transition matrix into elementary matrices has been chosen in each bigrading. We call this data a \emph{parametrized} two-story complex. Each parametrized two-story complex defines the sequences $a(z_i)$ and $b(z_i)$ for all $i$ and $z\in\{x,y\}$. Some additional combinatorial data can be associated to $\Xi(C)$ when it is \emph{straightly parametrized}.
\begin{definition}
A matrix factorization $P=STS'$ of $P \in \GL n \F$ is \emph{straight} if $T$ is a permutation matrix, $S=S_1\dots S_k$ and $S'=S_1'\dots S_{k'}'$ where none of the elementary matrices $S_i$ and $S_i'$ are transpositions.

A parametrized two-story complex $\Xi(C)$ is \emph{straightly parametrized} if all transition matrices have straight factorizations.
\end{definition}
There exists a straight parametrization for every two-story complex. In fact, we have the following stronger result.
\begin{lemma}\label{lemma:LTU factorization}
Let $P \in \GL n \F$ be an invertible matrix. Then $P=LTU$ for some lower-triangular matrix $L$, upper-triangular matrix $U$, and permutation matrix $T$.
\end{lemma}
\begin{proof}
See for instance \cite[Exercise 2.M.11(c)]{artin2011algebra} for the existence of $LTU$ factorization.
\end{proof}
We can further decompose $L$, $U$, and $T$ as a product of elementary matrices. By Gaussian elimination, the matrix $U$ is a product of matrices $D_i^\lambda$ and $E_{ij}$ with $i<j$, the matrix $L$ is a product of matrices $D_i^\lambda$ and $E_{ij}$ with $i>j$, and the matrix $T$ is a product of transposition matrices $T_{ij}$. This shows that each two story complex has a straight parametrization.

\vspace{1em}

A factorization $P=LTU$ is reflected in the following form of the graphical representation of $P$ with crossings, crossover arrows, and black dots. There are some crossover arrows pointing to the left near the bottom, followed by some crossings in the middle, and finally there are some crossover arrows pointing to the right near the top. There are also some black dots above and below the crossings. See Figure \ref{fig:LTU factorization} for a visual depiction of this construction. 

More generally, a straight factorization $P=STS'$ has crossover arrows, possibly in both directions, near the bottom and the top, and crossings in the middle. There are also some black dots above and below the crossings. The crossover arrows near the top will be called the \emph{upper arrows} and the crossover arrows near the bottom will be called the \emph{lower arrows}. There is also a well-defined notion of endpoints associated to crossover arrows in an elevator shaft with a straight factorization; an upper arrow goes from $y_i$ to $y_j$ and a lower arrow goes from $x_k$ to $x_l$. For example, on Figure \ref{fig:LTU factorization}, the unique lower arrow goes from $x_3$ to $x_1$.

\vspace{1em}

Now that the existence of straight parametrizations has been established and some accompanying terminology introduced, we can proceed with the main strategy. We associate a system of weights to any straightly parametrized two-story complex $\Xi(C)$. This is a combinatorial measure of the complexity of crossover arrows and the weights will specify the order in which the crossover arrows will be removed.

\begin{definition}
Let $\Xi(C)$ be a straightly parametrized two-story complex and let $A$ be a crossover arrow in $\Xi(C)$. The weight of $A$ is $w = (\hat{w}, \check{w}) \in (\Z \setminus \{0\} \cup \{\infty\})^2$ where
$$\hat{w} = \pm\inf \{ k \in \N \ | \ a(z_i)_k \neq a(z_j)_k \} \quad \text{and} \quad \check{w} = \pm\inf \{ k \in \N \ | \ b(z_i)_k \neq b(z_j)_k \},$$
where $A$ goes from $z_i$ to $z_j$. In words, if the complexes $A$ is connecting are parallel, the weight of the arrow is $(\infty, \infty)$. Otherwise, there is a smallest index in which the sequences $a(z_i)$ and $a(z_j)$ differ and we denote it as $\hat{w}$. The weight $\check{w}$ is defined similarly to be the smallest index in which the sequences $b(z_i)$ and $b(z_j)$ differ. The sign of the weights is determined as follows. The sign of $\hat{w}$ is $+$ if $a(z_i) <^! a(z_j)$ and $-$ otherwise and the sign on $\check{w}$ is $+$ if $b(z_i) <^! b(z_j)$ and $-$ otherwise.
\end{definition}
A careful reader might notice that the definition of the weight system is motivated by Lemma \ref{lemma:removing arrows}. If the weight of a crossover arrow $A$ is $(\hat{w}, \check{w})$, then we can slide it towards the floors for $|\hat{w}|-1$ turns before its endpoints diverge and away from the floors for $|\check{w}|-1$ turns before its endpoints diverge. At that time, we will be able to remove $A$ using the lemma provided the sign of the corresponding weight is positive.
\begin{definition}
Let $A$ be a crossover arrow in a straightly parametrized two-story complex $\Xi(C)$ with weight $w=(\hat{w}, \check{w})$. The depth of $A$ is $\min(|\hat{w}|, |\check{w}|)$ and the depth of $\Xi(C)$ is the minimum of depths across all of its crossover arrows.
\end{definition}
We emphasize here that the crossover arrows, crossings, black dots, sequences $a(z_i)$ and $b(z_i)$, and the weights and depths of $\Xi(C)$ all depend heavily on its straight parametrization. While this may seem unnecessary considering that the theorem statement is parametrization independent, we do not know if this technology can be avoided.

\vspace{1em}

Lemma \ref{lemma:removing arrows} and Lemma \ref{lemma:one crossover arrow can always be removed} show that the orientation of crossover arrows plays an important role in the approach to their removal. As such, straight factorizations of transition matrices that respect this orientation (such that upper arrows can be slid to the top floor and removed and the lower arrows can be slid to the bottom floor and removed) are much more convenient than others. Formally, we have the following definition.
\begin{definition}
Let $<$ be a total order on the set of basis elements in some elevator shaft with a transition matrix $P$. Then a straight factorization $P=STS'$ is ordered with respect to $<$ if $z_i < z_j$ whenever there is a crossover arrow from $z_i$ to $z_j$.
\end{definition}
\begin{lemma}\label{lemma:preliminary sorting}
Let $\Xi(C)$ be a straightly parametrized two-story complex. Fix an elevator shaft of $\Xi(C)$ with a transition matrix $P$ and let $<$ be a total order on the set of horizontal and vertical basis elements in this shaft. Then there exists a straight factorization $P = STS'$ that is ordered with respect to $<$.
\end{lemma}
\begin{proof}
If $<$ is such that the horizontally simplified basis elements are ordered in an increasing way from left to right and the vertically simplified basis elements are ordered in a decreasing way from left to right, then this is the $LTU$ factorization of Lemma \ref{lemma:LTU factorization}. Any other order $<$ can be related to the above one by two permutations; hence $\tau_1 P \tau_2 = LTU$ for some permutation matrices $\tau_1$ and $\tau_2$. It follows that $P=\tau_1^{-1}LTU\tau_2^{-1}$. The crossover arrows and black dots from $L$ can be slid downwards over the crossings from $\tau_1^{-1}$ and similarly the crossover arrows and black dots from $U$ can be slid upwards over the crossings of $\tau_2^{-1}$. This produces a straight factorization of $P$ that is ordered with respect to $<$ as required.
\end{proof}
We now upgrade the result by proving that the new straight factorization does not change the weights and sequences $a(z_i)$ too much.
\begin{lemma}\label{lemma:no technical problems}
Let $\Xi(C)$ be a straightly parametrized two-story complex of depth $m$. Fix an elevator shaft of $\Xi(C)$ with a transition matrix $P$ and let $<$ be a total order on the set of horizontal and vertical basis elements in this shaft. Then there exists another straight parametrization of the same two-story complex such that all transition matrices except $P$ retain the same factorization and $P=STS'$ is ordered with respect to $<$.

The depth of the new parametrized two-story complex is $\geq m$ and the first $m$ terms of the new sequences $a(z_i)$ agree with the old ones for all horizontal and vertical basis elements $z_i$.
\end{lemma}
\begin{proof}
This is equivalent to the statement of Lemma 3.14 in \cite{hanselman2016bordered} after the translation between the two settings.

Declare two elevator arrows to be equivalent if they remain parallel for $m-1$ steps in both directions, \emph{i.e.} if $a(y_i)_k=a(y_j)_k$ and $b(y_i)_k=b(y_j)_k$ for $k\in\{1, \dots, m-1\}$ where $y_i$ and $y_j$ are the upper endpoints of the two elevator arrows. This is an equivalence relation on the set of elevator arrows in a certain bigrading. Because the depth of the configuration is $m$, there are no crossover arrows between different equivalence classes. 

The content of the first paragraph was proven in Lemma \ref{lemma:preliminary sorting}. Note that the absence of crossover arrows between different equivalence classes remains true with respect to the new parametrization. We now inspect the proof to ensure that the first $m$ terms of the new sequences $a(z_i)$ have not changed. Any changes in $a(z_i)$ are a result of a different permutation matrix $T$ in the chosen shaft. However, any basis element $z_i$ needs at least $1$ move until it reaches the new permutation and after that, since the old and new elevator arrows are in the same equivalence class, they remain parallel for $m-1$ more turns. Therefore, the first $m$ terms in the sequences $a(z_i)$ remain unchanged. It is an immediate consequence that the depth of $\Xi(C)$ is still at least $m$.
\end{proof}

\begin{figure}[t]
    \begin{tikzpicture}[scale=1.0]
        \def\a{0.5}
        \def\b{0.15}
        \draw[-implies,double equal sign distance] (1,3.5) to (2,3.5);
        \draw[-implies,double equal sign distance] (1,2.5) to (3,2.5);
        \draw[-implies,double equal sign distance] (3,0.5) to (1,0.5);
        
        \draw[very thick] (2,2) to[out=-90, in=90] (3,1);
        \draw[very thick] (3,2) to[out=-90, in=90] (2,1);
        
        \draw[very thick](1,4)--(1,0){};
        \draw[very thick](2,4)--(2,2){};
        \draw[very thick](2,1)--(2,0){};
        \draw[very thick](3,4)--(3,2){};
        \draw[very thick](3,1)--(3,0){};

        \filldraw[black] (1,0) circle (2pt) node[anchor=north]{$x_1$};
        \filldraw[black] (2,0) circle (2pt) node[anchor=north]{$x_2$};
        \filldraw[black] (3,0) circle (2pt) node[anchor=north]{$x_3$};
        \filldraw[black] (1,4) circle (2pt) node[anchor=south]{$y_1$};
        \filldraw[black] (2,4) circle (2pt) node[anchor=south]{$y_2$};
        \filldraw[black] (3,4) circle (2pt) node[anchor=south]{$y_3$};
        
        \draw[<->, black, very thick] (-1,2) to (0,2);

        \draw[] (-3,2) node[]{$P=\overbrace{E_{31}}^L\overbrace{T_{23}}^T\overbrace{E_{13}E_{12}}^U$};

        \draw [very thick, decorate, decoration = {calligraphic brace, mirror}] (3+\a ,0+\b) --  (3+\a,1-\b);
        \draw [very thick, decorate, decoration = {calligraphic brace, mirror}] (3+\a ,1+\b) --  (3+\a,2-\b);
        \draw [very thick, decorate, decoration = {calligraphic brace, mirror}] (3+\a ,2+\b) --  (3+\a,4-\b);

        \draw[] (4-\b,0.5) node[]{$L$};
        \draw[] (4-\b,1.5) node[]{$T$};
        \draw[] (4-\b,3) node[]{$U$};
    \end{tikzpicture}
    \caption{A decomposition $P=LTU$ corresponds to the diagram with a block of crossover arrows pointing to the left near the bottom, a block of crossings in the middle and a block of crossover arrows pointing to the right near the top.}
    \label{fig:LTU factorization}
\end{figure}
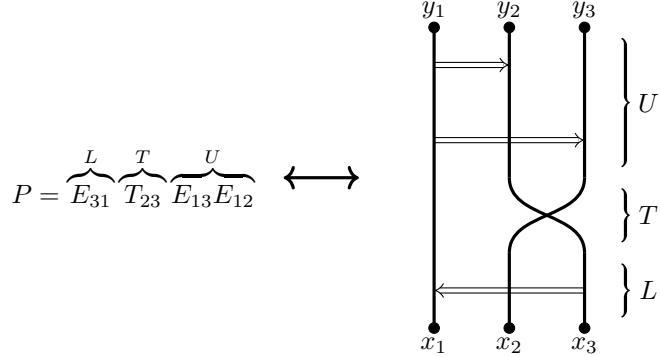
\begin{lemma}\label{lemma:the hard lemma}
Let $m \in \N$. A straightly parametrized two-story complex $\Xi(C)$ of depth $m$ is isomorphic to a straightly parametrized two-story complex of depth $\geq m+1$.
\end{lemma}
\begin{proof}
This is equivalent to the statement of Proposition 3.13 in \cite{hanselman2016bordered} after the translation between the curve configurations, handles and $\{n,e,s,w\}$-order and straightly parametrized two-story complexes, elevator shafts and unusual order $\leq^!$ respectively. Nonetheless, we include the argument here. 

Fix an elevator shaft. All crossover arrows of depth $m$ in this shaft will be removed in four steps using induction.
\begin{enumerate}
    \item $\hat{w}=-m$: Reparametrize the shaft using Lemma \ref{lemma:no technical problems} where $<$ is any extension of $\leq^!_m$ to a total order on the set of $a(z_i)$. After this, there are no arrows with $\hat{w}=-m$. Indeed, since the shaft is ordered with respect to $\leq^!_m$, all arrows of depth $m$ pointing from $z_i$ to $z_j$ satisfy $a(z_i) <^!_m a(z_j)$. Therefore $\hat{w} = +m$.
    \item $\hat{w}=m$: Let us remove the upper arrows with this property, noting that the lower arrows can be removed analogously. We will use the snowplow technique from \cite{hanselman2016bordered}. Pick the uppermost crossover arrow $A$ with $\hat{w}=m$. This means that all arrows above it have $\hat{w} \geq m+1$ so they can be slid towards the floors for $m$ steps and remaining parallel. We may also encounter arrows in other elevator shafts during this process, and we can just push them along in front of $A$. They all have depth at least $m$ so they will remain parallel for $m-1$ turns. When the endpoints of $A$ diverge, we can remove it by Lemma \ref{lemma:removing arrows} and then slide all displaced arrows back to the original position. 
    \item $\check{w}=-m$: This is the technically complicated part of the proof. The idea is that we begin by reparametrizing the lower part of the shaft like we did in Lemma \ref{lemma:no technical problems}. More precisely, we take $<$ to be any extension of $\leq^!_{m+1}$ to a total order on the set of $a(z_i)$ and we reparametrize the region consisting of lower arrows, crossings and black dots only, leaving the upper arrows intact. Lemma \ref{lemma:no technical problems} guarantees that the first $m$ terms of the old and new sequences $a(z_i)$ are the same.
    
    More precisely, the proof of Lemma \ref{lemma:no technical problems} can be refined by defining equivalence classes of elevator arrows in this shaft by declaring two elevator arrows equivalent if they remain parallel for $m$ steps in the direction of the bottom floor and for $m-1$ steps in the direction of the top floor. By Step (1) and Step (2), there are no lower arrows between different equivalence classes, so it is sufficient to work with each equivalence class of elevator strands separately. We now reparametrize as in Lemma \ref{lemma:no technical problems}. Since they are ordered with respect to $\leq^!_{m+1}$, the new lower arrows have $\hat{w}= m+1$ or $|\hat{w}|\geq m+2$. The old upper arrows remain unchanged, but there might be some new upper arrows with $\hat{w}=m$ or $|\hat{w}|\geq m+1$, because the upper strands are ordered with respect to $\leq^!_{m+1}$, and $|\check{w}| \geq m+1$, because all elevator arrows in this equivalence class remain parallel in the downwards direction for $m$ steps. The new upper arrows with $\hat{w}=m$ can be removed as in Step (2).

    As remarked in the previous paragraph, the new lower arrows in this elevator shaft are ordered with respect to $\leq^!_{m+1}$ and satisfy $\hat{w}=m+1$ or $|\hat{w}| \geq m+2$. Since they remain parallel for at least one step in the direction towards the bottom floor, we slide all of them one by one to the neighbouring elevator shaft. In their new position, they satisfy $\check{w}=m$ or $|\check{w}| \geq m+1$. They also satisfy $|\hat{w}|\geq m+1$ because they previously satisfied $|\check{w}|\geq m$. 

    Therefore, all lower arrows were slid out of the elevator shaft without creating any new crossover arrows with $\hat{w}=-m$ or $\check{w}=-m$ anywhere in $\Xi(C)$. We repeat Step (3) with the upper arrows to remove all arrows with $\hat{w}=-m$ or $\check{w}=-m$ from a chosen elevator shaft without creating any such arrows anywhere in the complex. Finally, we repeat the entire process in all elevator shafts to remove the arrows with $\hat{w}=-m$ or $\check{w}=-m$ from the parametrized two-story complex altogether.
    
    \item $\check{w}=m$: We use the snowplow technique again, first on the upper arrows and then on the lower arrows. We use the fact that Steps (1) and (2) have already been completed so all arrows we encounter can be pushed along for $m$ steps without diverging. We use the fact that $\check{w}=+m$ to establish that once the crossover arrow diverges, it can be removed by Lemma \ref{lemma:removing arrows}.
    \end{enumerate}
\end{proof}

\begin{corollary}\label{cor:no arrows between nonparallel complexes}
    A two-story complex $\Xi(C)$ is isomorphic to a two-story complex of depth $\infty$.
\end{corollary}
\begin{proof}
Let $n = \mathrm{rank}(C)$. Then there are $2 \binom{n}{2}$ pairs of basis elements in $\Xi(C)$. This means that if we slide any crossover arrow $2\binom{n}{2}$ turns towards the floors and it remains between parallel strands, it must have returned to its original position at some before then. This shows that any arrow of depth $> 2\binom{n}{2}$ has depth $\infty$. Applying Lemma \ref{lemma:the hard lemma} repeatedly, we can conclude that $\Xi(C)$ is isomorphic to a two-story complex of depth $\infty$.
\end{proof}

\subsection{Existence of decomposition}\label{subsection:existence of decomposition}
In this subsection, we utilize the arrow sliding algorithm to establish the existence of the splitting in Theorem \ref{thm:classification of algebraic complexes up to iso} and Theorem \ref{thm:classification of algebraic complexes up to homotopy}. En route, we also prove several algebraic results about finitely generated free chain complexes over $\mathcal{R}_1$ that might be of independent interest.

\vspace{1em}

Our pictorial representation of a chain complex $(C, \partial)$ is not an isomorphism invariant since it depends on the choice of a distinguished basis $B$. Pictures drawn with respect to different bases in general have different numbers of arrows. However, recall that we proved Lemma \ref{lemma:having an arrow of length 0 is a basis invariant} stating that the property of having an arrow of length $0$ is a basis-invariant property. This makes the statement of the following lemma well-defined.
\begin{lemma}\label{lemma:no arrows or length 0 -> minimal rank}
    Let $(C, \partial)$ be a finitely generated free chain complex over $\mathcal{R}_1$ with no arrows of length $0$. Then $C$ has a minimal rank among all such complexes in its chain homotopy equivalence class.
\end{lemma}
\begin{proof}
    Let $n = \mathrm{rank}(C)$. Since $C$ has no arrows of length zero, $\partial$ mod $(U,V)$ vanishes and it follows that $H_*(C/(U,V)) \cong \F^n$. Any chain complex over $\mathcal{R}_1$ that is chain homotopy equivalent to $C$ has the same homology and so its rank is at least $n$.
\end{proof}

\begin{lemma}\label{lemma:minimal rank trick}
    Let $(C, \partial)$ be a finitely generated free chain complex over $\mathcal{R}_1$ of minimal rank among all such complexes in its chain homotopy equivalence class. Then $(C, \partial)$ is unique up to isomorphism.
\end{lemma}
\begin{proof}
    Let $n=\mathrm{rank}(C)$ and let $(D, \partial_D)$ be another free chain complex over $\mathcal{R}_1$ of rank $n$ such that $C \simeq D$. We will show that $C \cong D$.
    
    Let $f: C\to D$ be a chain homotopy equivalence with a chain homotopy inverse $g:D\to C$ and let $h: C \to D$ be a chain homotopy witnessing $gf \simeq_{h} \id_C$. This means that we have the equality of maps $gf = \id_C + h\partial+\partial_D h$. Let $\{x_1, \dots, x_n\}$ be a basis of $C$. Since $C$ and $D$ both minimize the number of generators in their chain homotopy type, we have $\im \partial \subset (U,V)$ and $\im \partial_D \subset (U,V)$. This is because we can always remove arrows of length $0$ without affecting the chain homotopy type by Corollary \ref{cor:we split away zero complexes}. The map $h$ is an $\mathcal{R}_1$-module homomorphism and so also $(h\partial + \partial_D h)(x_i) \in (U,V)$. It follows that $gf = \id_C + h\partial + \partial_Dh = \id_C \pmod {(U,V)}$. So $gf \pmod {(U,V)}$ is injective and hence $gf: C \to C$ is a change of basis. It follows that $f$ is an isomorphism as required.
\end{proof}

    
\begin{corollary}\label{cor:iso iff homotopy equivalent}
    Let $C$ and $D$ be finitely generated free chain complexes over $\mathcal{R}_1$ with no arrows of length $0$. Then $C \simeq D$ if and only if $C \cong D$.
\end{corollary}
\begin{proof}
    By Lemma \ref{lemma:no arrows or length 0 -> minimal rank} $C$ and $D$ both minimize rank among all complexes in their chain homotopy classes. By Lemma \ref{lemma:minimal rank trick} such complexes are homotopy equivalent if and only if they are isomorphic.
\end{proof}

Equipped with Lemma \ref{lemma:minimal rank trick} and Corollary \ref{cor:iso iff homotopy equivalent}, we are finally ready to prove the existence of the decomposition.

\begin{lemma}\label{lemma:existence of decomposition iso}
Let $C$ be a finitely generated free chain complex over $\mathcal{R}_1$. Then there exist snake complexes $S_1, \dots, S_m$, local systems $L_1, \dots, L_n$, and zero complexes $Z_1, \dots, Z_k$ such that 
$$C \cong S_1 \oplus \dots \oplus S_m \oplus L_1 \oplus \dots \oplus L_n \oplus Z_1 \oplus \dots \oplus Z_k.$$
\end{lemma}
\begin{proof}
Let $(C,\partial)$ be a finitely generated free chain complex over $\mathcal{R}_1$. Corollary \ref{cor:we split away zero complexes} shows that $C \cong D \oplus Z_1 \oplus \dots \oplus Z_k$ for some finitely generated free chain complex $D$ over $\mathcal{R}_1$ with no arrows of length $0$ and zero complexes $Z_i$. Let $\{x_1, \dots, x_n\}$ be a vertically simplified basis for $D$ and let $\{y_1, \dots, y_n\}$ be a horizontally simplified basis for $D$. The two-story complex $\Xi(D)$ is isomorphic to one with no crossover arrows between nonparallel complexes by Corollary \ref{cor:no arrows between nonparallel complexes}. Finally, slide all remaining crossover arrows and black dots to the top floor, \emph{i.e.}, perform the corresponding basis changes in the horizontally simplified basis. 

Observe that sliding the crossover arrows to the top floor will in general not preserve the fact that the top basis $\{y_1', \dots, y_n'\}$ is horizontally simplified. However, we are still able to recover a simplified decomposition of $D$, as follows. The fact that the two story complex associated to $D$ decomposes as a disjoint union of bundles of parallel components means that $D$ splits as a finite direct sum of complexes corresponding to bundles of parallel complexes. Any finite component automatically has width $1$ and yields a snake complex summand of $D$. Any cyclic bundle corresponds to a local system. It follows that $C \cong S_1 \oplus \dots \oplus S_m \oplus L_1 \oplus \dots \oplus L_n \oplus Z_1 \oplus \dots \oplus Z_k$ for some snake complexes $S_i$, local systems $L_i$ and zero complexes $Z_i$.
\end{proof}

\begin{lemma}\label{lemma:existence of decomposition htpy}
Let $C$ be a finitely generated free chain complex over $\mathcal{R}_1$. Then there exist snake complexes $S_1, \dots, S_m$ and local systems $L_1, \dots, L_n$ such that 
$$C \simeq S_1 \oplus \dots \oplus S_m \oplus L_1 \oplus \dots \oplus L_n.$$
\end{lemma}
\begin{proof}
Let $(C,\partial)$ be a finitely generated free chain complex over $\mathcal{R}_1$. By Lemma \ref{lemma:existence of decomposition iso} we have $C \cong S_1 \oplus \dots \oplus S_m \oplus L_1 \oplus \dots \oplus L_n \oplus Z_1 \oplus \dots \oplus Z_k$ for some snake complexes $S_i$, local systems $L_i$ and zero complexes $Z_i$. Since $Z_i \simeq 0$ for all $i$, it follows that $C \simeq S_1 \oplus \dots \oplus S_m \oplus L_1 \oplus \dots \oplus L_n$, which establishes the required decomposition.
\end{proof}

\subsection{Uniqueness of decomposition}
This subsection completes the classification of isomorphism and homotopy equivalence classes of finitely generated free chain complexes over $\mathcal{R}_1$ by establishing the uniqueness of the direct summands in the decomposition. This proves Theorem \ref{thm:classification of algebraic complexes up to iso} and Theorem \ref{thm:classification of algebraic complexes up to homotopy}, which was the main goal of this long section.

\vspace{1em}

In fact, we prove something slightly stronger than uniqueness; we prove that the category of finitely generated free $\Z\oplus\Z$ graded chain complexes over $\mathcal{R}_1$ is a Krull-Schmidt category. All Krull-Schmidt categories have a unique decomposition of objects into indecomposable summands and the conclusion follows.
\begin{definition}
    A ring $R$ is local if for every $x \in R$, either $x$ or $1-x$ is a unit.
\end{definition}
\begin{definition}
An additive category $\mathcal{C}$ is a \emph{Krull–Schmidt category} if every object of $\mathcal{C}$ decomposes into a finite direct sum of objects with local endomorphism rings.
\end{definition}
The previous subsection established the existence of the decomposition in Lemma \ref{lemma:existence of decomposition iso} and so the outstanding claim is that the endomorphism rings of snake complexes, indecomposable local systems and zero complexes are local rings. This is done in turn with the following sequence of lemmas.
\begin{lemma}\label{lemma:std cxs have local endomorphism rings}
    Let $C$ be a standard complex. Then $\End(C)$ is local.
\end{lemma}
\begin{proof}
    Let $f: C \to C$ be an endomorphism of $C$. We have $H_*(C/U)/T \cong \F[V]\langle x_0 \rangle$ where $T$ is the torsion $\F[V]$-submodule of $H_*(C/U)$. The induced map $H_*(f/U)/T$ on homology is either an isomorphism or a zero map because $f$ is grading preserving. In the latter case, the map $\id-f$ induces an isomorphism on homology. By \cite[Lemma 6.6]{dai2021more} either $f$ or $\id-f$ is an isomorphism and thus $\End(C)$ is local.
\end{proof}
We refine this proof slightly in order to apply it to other snake complexes.
\begin{lemma}\label{lemma:snake complexes have local endomorphism rings}
    Let $S$ be a horizontal or vertical snake complex. Then $\End(S)$ is local.
\end{lemma}
\begin{proof}
We prove the statement for horizontal snake complexes. The case with vertical snake complexes is completely analogous.

Let $b_1, \dots, b_m$ be a sequence of nonzero integers describing $S$ and let $f: S \to S$ be a chain map. We have $H_*(S/U)/T \cong \F[V]\langle x_0, x_m \rangle$. Two cases arise depending on whether $f(x_0)$ is supported by $x_0$, \emph{i.e.}, whether $\langle f(x_0), x_0 \rangle \neq 0$. 
    \begin{enumerate}
        \item $f(x_0)$ is supported by $x_0$.
        \item $f(x_0)$ is not supported by $x_0$. In that case $\id -f$ is supported by $x_0$.
    \end{enumerate}
    Let $g \in \{f, \id-f\}$ be the one with $g(x_0)$ supported by $x_0$. Then $g(x_i)$ is supported by $x_i$ for all $i$ by the proof of \cite[Lemma 4.12]{dai2021more}. The proofs of \cite[Lemma 6.4, Lemma 6.5, and Lemma 6.6]{dai2021more} tell us that $g$ is an isomorphism and thus $\End(S)$ is local.
\end{proof}

In particular, Lemma \ref{lemma:std cxs have local endomorphism rings} and Lemma \ref{lemma:snake complexes have local endomorphism rings} show that the snake complexes are indecomposable. To see this, observe that locality of $\End(C)$ implies that $C$ has no nontrivial idempotents, which is equivalent to indecomposability. Other summands are also indecomposable; indecomposable local systems are indecomposable by definition and it is not hard to see that the zero complexes are indecomposable. Thus, the following lemma shows that the endomorphism rings of all our direct summands are local.
\begin{lemma}\label{lemma:local}
Let $C$ be a finitely generated free $\Z\oplus\Z$ graded chain complex over $\mathcal{R}_1$ that is indecomposable. Then $\End(C)$ is local.
\end{lemma}

\begin{proof}
Let $x_1, \dots, x_n$ be a basis for $C$ as a chain complex over $\mathcal{R}_1$ and let $N \in \Z$ be such that $\gr(x_i) \geq (N,N)$ for all $i \in \{1, \dots, n\}$. Denote by $C_{ij}$ the submodule of $C$ generated by the homogeneous elements of degree $(i,j)$ and consider the quotient $D = C/\bigoplus_{\min\{i,j\} < N} C_{ij}$. Any endomorphism $f$ of $C$ is specified by the images $f(x_i)$ for all $i$. Since $f$ is grading preserving, $\gr(f(x_i)) = \gr(x_i)$ and all information about the endomorphism lives in $D$. It follows that $\End(C) \cong \End(D)$.

Since $C$ is indecomposable, $\End(C)$ has no idempotents except $0$ and $\id$. It follows that the same is true for $\End(D)$ and hence $D$ is indecomposable as well. We will now show that $D$ is both a Noetherian and an Artinian $\mathcal{R}_1$-module, \emph{i.e.}, it satisfies the ascending and descending chain conditions on submodules. Note that any chain of $\mathcal{R}_1$-submodules is also a chain of $\F$-submodules by forgetting some structure. Since $D$ itself is finitely generated as an $\F$-vector space, it follows that any increasing or a decreasing chain of submodules of $D$ is finite. Thus $D$ satisfies both chain conditions and by \cite[Proposition 5.4]{krause2015krull}, $\End(D)$ is a local ring. Hence $\End(C) \cong \End(D)$ is local as well.
\end{proof}
We may now apply the Krull-Schmidt theorem to get uniqueness of the decomposition.
\begin{lemma}[{{\cite[Theorem 4.2]{krause2015krull}}}]\label{lemma:krause}
Let $C$ be an object in an additive category and let $X_1 \oplus \dots \oplus X_n \cong Y_1 \oplus \dots \oplus Y_m$ be two decompositions of $C$ into objects with local endomoprism rings. Then $n = m$ and and there exists a permutation $\pi$ such that $X_i \cong Y_{\pi(i)}$ for $i \in \{1, \dots, n\}$.
\end{lemma}
With that, we achieve our goal of classifying finitely generated free $\Z\oplus\Z$ graded chain complexes over $\mathcal{R}_1$ up to isomorphism and up to homotopy equivalence.

\begin{proof}[Proof of Theorem \ref{thm:classification of algebraic complexes up to iso}]
The existence of such a decomposition is shown in Lemma \ref{lemma:existence of decomposition iso}. By Lemma \ref{lemma:local}, the direct summands have local endomorphism rings. The uniqueness of such a decomposition is Lemma \ref{lemma:krause}.
\end{proof}
\begin{proof}[Proof of Theorem \ref{thm:classification of algebraic complexes up to homotopy}]
The existence of such a decomposition is shown in Lemma \ref{lemma:existence of decomposition htpy}. Assume that also $C \simeq S_1' \oplus \dots \oplus S_{m'}' \oplus L_1' \oplus \dots \oplus L_{n'}'$ for some snake complexes $S_i'$ and local systems $L_i'$. By Corollary \ref{cor:iso iff homotopy equivalent} the two decompositions are in fact isomorphic and this establishes uniqueness as above.
\end{proof}

\section{Topological applications}\label{section:topological applications}
In this section, we give some topological applications of our result. We begin with the proof of Theorem \ref{thm:classification of CFL up to iso}, which describes the algebraic structure of link Floer complexes.
\begin{proof}[Proof of Theorem \ref{thm:classification of CFL up to iso}]
Let $L \subset Y$ be a nullhomologous $l$-component link and let $\CFL_{\mathcal{R}_1}(Y, L)$ be its link Floer complex. Since $\CFL_{\mathcal{R}_1}(Y, L)$ is in particular a finitely generated free chain complex over $\mathcal{R}_1$, Theorem \ref{thm:classification of algebraic complexes up to homotopy} applies and $\CFL_{\mathcal{R}_1}(Y, L) \simeq S_1 \oplus \dots \oplus S_m \oplus L_1 \oplus \dots \oplus L_k$ for unique snake complexes $S_i$ and local systems $L_i$. This gives us the required form.
\end{proof}
Note that if $Y = S^3$, then $\CFL_{\mathcal{R}_1}(S^3, L)$ has the homology of a link so $\frac{H_*(C/U)}{V-\text{torsion}}$ has rank $2^{l-1}$ as an $\F[V]$-module and similarly $\frac{H_*(C/V)}{U-\text{torsion}}$ is a free $\F[U]$-module of rank $2^{l-1}$. Each snake complex in the decomposition contributes rank $(1,1)$, $(2,0)$ or $(0,2)$ in the above two complexes, and it follows that there are precisely $2^{l-1}$ snake complexes. If we further restrict to the case of knots, the splitting contains a single standard complex $C(a_1, \dots, a_{2n})$ for some nonzero integers $a_1, \dots, a_{2n}$. However, not every standard complex can actually be realized by a knot. For example, knot Floer complexes are symmetric and so their standard complex summands should be too. Moreover, knot Floer complexes are naturally equipped with the $\F[U,V]$-module structure which places additional restrictions on $a_1, \dots, a_{2n}$. See \cite{popovic2023algebraic} for the classification of standard complexes that arise from $\F[U,V]$-modules. On a similar note, it would be interesting to see what additional restrictions there are on the local systems that can appear in the decomposition. 

\vspace{1em}

In \cite{juhasz2020knot}, a torsion order $\mathrm{Ord}_U(K)$ of knot Floer homology is introduced and it is shown that it can be used to provide bounds on the number of local maxima and the genus of a cobordism between two knots. In turn, this gives lower bounds on the bridge index and the band-unlinking number of a knot, the fusion number of a ribbon knot, and the number of minima appearing in a slice disk of a knot. It also gives a lower bound on the number of bands appearing in a ribbon concordance between two knots. We show that this invariant can be read off the decomposition of $\CFL_{\mathcal{R}_1}(S^3, K)$ into standard complexes and local systems.
\begin{lemma}
Let $K \subset S^3$ be a knot. Then $\mathrm{Ord}_U(K)$ is the length of the longest horizontal arrow in the standard complex or a local system.
\end{lemma}
\begin{proof}
By definition $\mathrm{Ord}_U(K) = \min \{ k \in \N \ | \ U^k \cdot \mathrm{Tor}(\HFL^-(K))=0\}$. Find a basis for which $\CFL_{\mathcal{R}_1}(K)$ splits into a standard complex and local systems. Setting $V=0$ and taking homology gives $\HFL^-(K) \cong \F[U] \oplus \bigoplus_i \F[U]/(U^{d_i})$ where $d_i$ are the lengths of the horizontal arrows. It follows that $\mathrm{Ord}_U(K)$ is the maximal $d_i$. 
\end{proof}

\section{Connection with immersed curves}\label{section:immersed curves proof}
The bordered Heegaard Floer homology package developed by Lipshitz, Ozsv\'ath, and Thurston in \cite{lipshitz2018bordered} can be used to losslessly transfer the information stored in $\CFL_{\mathcal{R}_1}(S^3, K)$ into a bordered type $D$ module $\widehat{\mathit{CFD}}(S^3 \setminus \nu(K))$ where $\nu(K)$ is some tubular neighborhood of $K$. The manifold $S^3 \setminus \nu(K)$ has a torus boundary and thus lends itself to the techniques developed in \cite{hanselman2016bordered} for studying type $D$ structures of such manifolds in a graphical manner. The type $D$ structure is encoded into a train track, which is then simplified using the arrow sliding algorithm. This leads to a collection of immersed curves on a punctured torus, each equipped with a local system. Similarly, our main theorem states that $\CFL_{\mathcal{R}_1}(S^3, K)$ splits as a direct sum of a standard complex and some local systems. There \emph{must} be a connection between the two results. 

\vspace{1em}

Indeed, our algorithm for construction of the suitable basis was heavily inspired by the arrow sliding algorithm in \cite{hanselman2016bordered} and the work of Hanselman, Rasmussen, and Watson more generally. In the next few paragraphs, we explain the informal correspondence between the two.

\vspace{1em}

The main objects of interest in both papers are algebraic. We are primarily interested in knot Floer complexes $\CFL_{\mathcal{R}_1}(S^3, K)$ and \cite{hanselman2016bordered} is studying type $D$ modules $\widehat{\mathit{CFD}}(M)$ associated to manifolds $M$ with torus boundary. It turns out that forgetting the underlying geometry by axiomatizing these notions is very effective and it is also the reason both proofs are combinatorial in nature.

The standard way to draw knot Floer complexes is in the plane as described in Section \ref{subsection:pictures}. Since the differentials in the bordered setting carry more information than the two spatial dimensions can unambiguously encode, type $D$ structures are instead represented by decorated directed graphs. Developing our analogy further, two-story complexes correspond to train tracks as the main technical tool for basis simplification in both cases. While there is no clear way to pass between the two, the basis changes corresponding to the crossover arrow movement in different settings bear close resemblance to each other. Some of our basis changes are of the forms $z_i \mapsto z_i+z_j$ for $z\in\{x,y\}$, $y_i \mapsto y_i+U^ay_j$ and $x_i \mapsto x_i+V^ax_j$ and the basis changes for moves (M1) and (M2) are of the forms $x \mapsto x+y$ and $x \mapsto x + \rho_I \otimes y$.

Finally, the $\{n, e, s, w\}$ order in their setting plays the role of $\leq^!$ in our paper. The orders share the crucial property that a crossover arrow in either setting can be slid in some direction and eventually removed if and only if it points towards the larger complex/curve.

\vspace{1em}

While similar, we do not believe that there is a formal relationship between the two proofs, and there are some salient differences. Firstly, our algorithm preserves all information about the diagonal arrows in the knot Floer complex. The algorithm produces a sequence of genuine basis changes in $\CFL_{\mathcal{R}_1}(S^3, K)$ and by applying the same changes to $\CFL_{\F[U,V]}(S^3, K)$, one can recover the diagonal arrows in the simplified complex. The same cannot be achieved in the bordered setting since this family of invariants discards this information by setting $UV=0$. In fact, this was already noted in \cite[Example 53]{hanselman2022heegaard} where the authors expected not to be able to recover information about the diagonal arrows from their construction in general. Our algorithm also works over an arbitrary field $\F$ whereas \cite{hanselman2016bordered} uses the $\F_2$ coefficients.

Finally, our algorithm applies more generally to nullhomologous knots in arbitrary $3$-manifolds. Nonetheless, a special case of Theorem \ref{thm:classification of CFL up to iso} for knots can also be proven through the use of immersed curves and we outline this now.


\vspace{1em}

The first step is to recover $\CFK_{\mathcal{R}_1}(S^3, K)$ from a set of immersed curves on a punctured torus decorated with local systems that is associated to $S^3 \setminus \nu(K)$. This is done in \cite[Section 4.3]{hanselman2022heegaard}, but the authors implicitly assume that $\CFK_{\mathcal{R}_1}(S^3, K)$ admits a simplified basis; it is currently an open question whether this is necessarily the case.

\vspace{1em}

Let $\gamma$ be the union of train tracks on the punctured torus. Replace the puncture with two points $z$ and $w$ and let the $\beta$ curve go between them. Let $S$ be the set of intersection points between $\gamma$ and $\beta$ and let $C$ be the free $\mathcal{R}_1$-module generated by $S$. Define $\partial: C \to C$ via 
$$\partial x = \sum_{i=0}^\infty \sum_{j=0}^\infty \sum_{y\in S} U^i V^j N_{i,j}(x,y) \ y$$
where $N_{i,j}(x,y)$ is the number of immersed bigons connecting $x$ and $y$ and covering the basepoints $w$ and $z$ $i$ and $j$ times respectively.

\begin{lemma}
$C \simeq \CFK_{\mathcal{R}_1}(K)$.
\end{lemma}
\begin{proof}
 This is essentially a variant of the construction described in \cite{hanselman2022heegaard} in Section 4.3.
\end{proof}

\begin{proof}[Sketch proof of Corollary \ref{cor:thm for knots in S3}]
    Let $\CFL_{\mathcal{R}_1}(S^3, K)$ be a knot Floer complex. We may assume that $\CFL_{\mathcal{R}_1}(S^3, K)$ does not have any arrows of length $0$, otherwise they can be spit off by Corollary \ref{cor:we split away zero complexes}. By Lemma \ref{lemma:no arrows or length 0 -> minimal rank}, $\CFL_{\mathcal{R}_1}(S^3, K)$ minimizes rank in its chain homotopy equivalence class. According to algorithm in \cite{hanselman2016bordered} represent $S^3 \setminus \nu(K)$ as a collection of immersed curves on a torus. By construction, $C$ has the same number of generators as the original complex $\CFL_{\mathcal{R}_1}(K)$ and by Lemma \ref{lemma:minimal rank trick} there is an isomorphism $\CFL_{\mathcal{R}_1}(K) \cong C$.

    It remains to determine the structure of $C$. Since the boundary of any immersed bigon lies entirely on a single component, $C$ splits into a direct sum. A closed curve on a punctured torus with a local system gives a direct summand, which is a local system in the sense of Definition \ref{def:local system}. The unique curve that wraps around the cyllinder as described in \cite[Remark 50]{hanselman2022heegaard} corresponds to the standard complex summand.
\end{proof}
Finally, let us also explain how to move in the opposite direction - how our decomposition of $\CFL_{\mathcal{R}_1}(S^3, K)$ into snake complexes and local systems corresponds to immersed curves on a punctured torus. The process involves two steps.
\begin{enumerate}
    \item Use \cite[Theorem 11.26]{lipshitz2018bordered} to translate the information from the chain complex $\CFL_{\mathcal{R}_1}(S^3, K)$ to the type $D$ structure $\widehat{\mathit{CFD}}(S^3\setminus \nu(K))$. 
    \item Use the results from \cite{hanselman2016bordered} to represent $\widehat{\mathit{CFD}}(S^3\setminus \nu(K))$ as train tracks on a punctured torus and slide them around to obtain a set of immersed curves, each equipped with a pair $(V, \Psi)$ of a vector space and an automorphism.
\end{enumerate}
In general, Step (1) involves a computation of vertically and horizontally simplified bases. The arrow sliding algorithm in Step (2) amounts to further improvements until the train tracks are transformed into closed curves. On the other hand, if we are \emph{given} a simplified decomposition to begin with, both of these steps become immediate and one can transfer directly from $\CFL_{\mathcal{R}_1}(S^3, K)$ to a picture of immersed curves. 

\vspace{1em}

Let us be precise. As is standard in \cite{hanselman2022heegaard}, we draw the immersed curves on an infinite cylinder with $\Z$ many punctures instead. This is a $\Z$-sheeted covering space of a punctured torus, so the composition with the covering projection can be used to recover the curves on a torus. 

The knot Floer complex generators are drawn in the middle of the strip between consecutive punctures. Their relative heights are determined by the Alexander grading $A = \frac{1}{2}(\gr_U-\gr_V)$ so that the generators with the same Alexander grading lie between the same punctures. Horizontal arrows correspond to the segments connecting the generators on the right and vertical arrows correspond to the segments connecting the generators on the left. It is readily verified that the above requirements imply that the arrows of length $l$ correspond to segments passing by $l$ punctures until they switch to the other side.

Standard complexes are already vertically and horizontally simplified, making them easy to draw. See Figure \ref{fig:immersed curve std} for the knot Floer complex and an immersed curve associated to the negative trefoil. The situation is more interesting for local systems, which in general do not admit simplified bases. Let $L$ be a local system and assume that we have a simplified decomposition of $L$ with exactly one nonidentity isomorphism $A$, for example as in Figure \ref{fig:local system immersed curves pic}. If $L$ is vertically simplified, the segments on the left half of the strip can be drawn as earlier. The segments on the right half of the strip, however, use a different basis, which needs to be related to the first one before it can be drawn. This is done exactly through the nonidentity isomoprhism $A$. The immersed train track on a torus can be expressed as an immersed curve associated to the shape of $L$ and with an automorphism $A$.

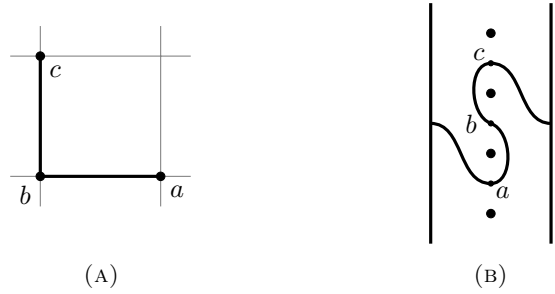
\begin{figure}[t]
    \begin{subfigure}[b]{0.40\textwidth}
    \centering
    \begin{tikzpicture}[scale=0.8]
    \def\a{0.2}
    \draw[step=2,gray,thin] (-2.5,-0.5) grid (0.5,2.5);
    \draw[very thick] (0,0)--(-2,0)--(-2,2);

    \filldraw[] (0,0) circle (2pt) node[anchor=north west]{$a$};
    \filldraw[] (-2,0) circle (2pt) node[anchor=north east]{$b$};
    \filldraw[] (-2,2) circle (2pt) node[anchor=north west]{$c$};

    \draw[] (0,-1) node[]{$ $};
    
    \end{tikzpicture}
    \caption{}
    \end{subfigure}
    \begin{subfigure}[b]{0.40\textwidth}
    \centering
    \begin{tikzpicture}[scale=0.8]
    \def\b{0.33}
    \draw[very thick] (0,0.5) -- (0,4.5);
    \draw[very thick] (2,0.5) -- (2,4.5);
    \filldraw[] (1,1) circle (2pt);
    \filldraw[] (1,2) circle (2pt);
    \filldraw[] (1,3) circle (2pt);
    \filldraw[] (1,4) circle (2pt);

    \filldraw[] (1,1.5) circle (1.2pt){};
    \filldraw[] (1,2.5) circle (1.2pt){};
    \filldraw[] (1,3.5) circle (1.2pt){};

    \draw[] (1.2,1.35) node[]{$a$};
    \draw[] (1-\b,2.5) node[]{$b$};
    \draw[] (0.8,3.65) node[]{$c$};

    \draw[very thick] (1,1.5) to[out=0, in=340] (1, 2.5);
    \draw[very thick] (1,2.5) to[out=160, in=180] (1, 3.5);
    \draw[very thick] (1,1.5) to[out=180, in=0] (0, 2.5);
    \draw[very thick] (1,3.5) to[out=0, in=180] (2, 2.5);
    
    \end{tikzpicture}
    \caption{}
    \end{subfigure}    
    \caption{A knot Floer complex $\CFL_{\mathcal{R}_1}(S^3, T_{2, -3})$ associated to the negative trefoil in $(\textsc{a})$ and the corresponding immersed curve $\gamma$ in $(\textsc{b})$.}
    \label{fig:immersed curve std}
\end{figure}

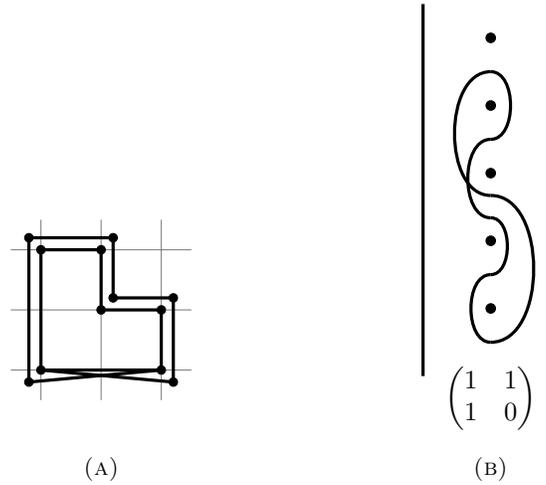
\begin{figure}[t]
    \begin{subfigure}[b]{0.40\textwidth}
    \centering
    \begin{tikzpicture}[scale=0.8]
    \def\a{0.2}

    \draw[step=1,gray,thin] (-0.5,-0.5) grid (2.5,2.5);
    
    \draw[very thick] (0,0)--(2,0)--(2,1)--(1,1)--(1,2)--(0,2)--(0,0){};
    \draw[very thick] (2+\a,0-\a)--(2+\a,1+\a)--(1+\a,1+\a)--(1+\a,2+\a)--(0-\a,2+\a)--(0-\a,0-\a){};
    \draw[very thick] (0,0) -- (2+\a,0-\a){};
    \draw[very thick] (-\a,-\a) -- (2,0){};
    \filldraw[] (0,0) circle (2pt);
    \filldraw[] (2,0) circle (2pt);
    \filldraw[] (2,1) circle (2pt);
    \filldraw[] (1,1) circle (2pt);
    \filldraw[] (1,2) circle (2pt);
    \filldraw[] (0,2) circle (2pt);
    \filldraw[] (0-\a,0-\a) circle (2pt);
    \filldraw[] (2+\a,0-\a) circle (2pt);
    \filldraw[] (2+\a,1+\a) circle (2pt);
    \filldraw[] (1+\a,1+\a) circle (2pt);
    \filldraw[] (1+\a,2+\a) circle (2pt);
    \filldraw[] (0-\a,2+\a) circle (2pt);

    \draw[] (0,-1) node[]{$ $};
    
    \end{tikzpicture}
    \caption{}
    \label{fig:local system immersed curves pic}
    \end{subfigure}
    \begin{subfigure}[b]{0.40\textwidth}
    \centering
    \begin{tikzpicture}[scale=0.9]
    \def\b{0.33}
    \draw[very thick] (0,0) -- (0,5.5);
    \draw[very thick] (2,0) -- (2,5.5);
    \filldraw[] (1,1) circle (2pt);
    \filldraw[] (1,2) circle (2pt);
    \filldraw[] (1,3) circle (2pt);
    \filldraw[] (1,4) circle (2pt);
    \filldraw[] (1,5) circle (2pt);

    \draw[very thick] (1,0.5) to[out=0, in=0] (1, 3-\b);
    \draw[very thick] (1,0.5) to[out=180, in=180] (1, 1.5);
    \draw[very thick] (1,1.5) to[out=0, in=0] (1, 3-2*\b);
    \draw[very thick] (1,3-2*\b) to[out=180, in=180] (1, 3.5);
    \draw[very thick] (1,3-\b) to[out=180, in=180] (1, 4.5);
    \draw[very thick] (1,3.5) to[out=0, in=0] (1, 4.5);

    \draw[] (1,0-\b) node[]{$\begin{pmatrix}
        1&1\\
        1&0
    \end{pmatrix}$};
    
    \end{tikzpicture}
    \caption{}
    \end{subfigure}    
    \caption{A nontrivial local system $L$ in $(\textsc{a})$ and the corresponding immersed curve $\gamma$ in $(\textsc{b})$. Note that the free homotopy class of $\gamma$ depends only on the shape of $L$. The matrix associated to $\gamma$ is the same as the matrix associated to  the nonidentity isomorphism in $L$.}
    \label{fig:immersed curve local}
\end{figure}

\section{Examples of knot Floer-like complexes}\label{section:examples}
Throughout this section, we let $\F=\F_2$. This paper has thus far mainly dealt with chain complexes over $\mathcal{R}_1$ to address the geography problem in knot and link Floer homology. Since all link Floer complexes $\CFL_{\F[U,V]}(L)$ are actually equipped with the $\F[U,V]$-module structure, one might hope that the conclusions of Theorem \ref{thm:classification of CFL up to iso} can be strengthened if we impose some additional algebraic structure on $C$ so that it approximates knot Floer complexes better. In this section, we produce some novel examples of knot Floer-like complexes showing that this hope was overly optimistic.

\vspace{1em}

We begin our construction with the toy example $T$, which is a chain complex over $\F[U,V]$ with the homology of a knot in $S^3$, but it is not symmetric. We later extend it to a symmetric complex $D$.

\subsection{Example \texorpdfstring{$T$}{T}} Let $T$ be the free $\F[U,V]$-module with basis consisting of $11$ generators $a, b, c, d, e, f, g, h, i, j, x$. Equip it with the $\F[U,V]$-module endomorphism $\partial: T\to T$ given on the basis elements by $\partial a=V^4x$, $\partial b = UV^3x$, $\partial c = V^2b + UVa$, $\partial d = U^2a+UVb$, $\partial e = V^2d+UVc$, $\partial f = U^2c+UVd+V^3j$, $\partial g = U^3e+U^2Vf+UV^4h+V^5i$, $\partial h = U^3Vx+ Uj$, $\partial i = U^4x$, $\partial j = 0$, $\partial x = 0$. See Figure \ref{fig:Example T} for a geometric description of $T$.

\begin{figure}[t]
    \centering
    \begin{tikzpicture}[scale=0.9]
    \draw[step=1.0,gray,thin] (0.5,0.5) grid (6.5,7.5);
    \draw[red, very thick] (1, 1) -- (1, 5) -- (3, 5) -- (3, 7) -- (6, 7) -- (6,2) -- (2,2);
    \draw[blue, very thick] (2,4) -- (2,6) -- (4,6) -- (4,3) -- (5,3);
    \draw[black, very thick] (1,5) -- (3,7);
    \draw[black, very thick] (1,1) -- (2,4) -- (4,6) -- (6,7);
    \draw[black, very thick] (2,2) -- (5,3) -- (6,7);
    \filldraw[black] (1,1) circle (2pt) node[anchor=north west]{$x$};
    \filldraw[black] (2,2) circle (2pt) node[anchor=north west]{$x$};
    \filldraw[black] (1,5) circle (2pt);
    \filldraw[black] (3,5) circle (2pt);
    \filldraw[black] (3,7) circle (2pt);
    \filldraw[black] (6,7) circle (2pt);
    \filldraw[black] (6,2) circle (2pt);
    \filldraw[black] (2,4) circle (2pt) node[anchor = north west]{$b$};
    \filldraw[black] (2,6) circle (2pt);
    \filldraw[black] (4,6) circle (2pt);
    \filldraw[black] (4,3) circle (2pt);
    \filldraw[black] (5,3) circle (2pt);
    \end{tikzpicture}
    \caption{Chain complex $T$ that is not chain homotopy equivalent to any finite complex. The infinite summand $A$ over the ring $\frac{\F[U,V]}{(UV)}$ is drawn in red for emphasis. The standard complex is drawn in blue.}
    \label{fig:Example T}
\end{figure}
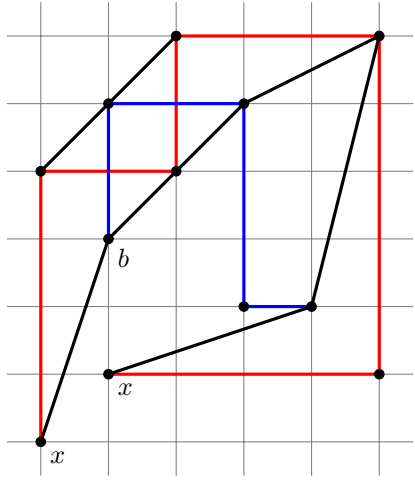

One can easily manually verify that $\partial^2=0$. Relative $\Z\oplus\Z$ gradings of all basis elements are uniquely determined by the requirement that $\partial$ have bidegree $(-1,-1)$. They can be made absolute by fixing $\gr(b)=(4,0)$.

A critical aspect of Figure \ref{fig:Example T} we will be concerned about is that $x$ appears in it twice. In other words, $x$ and $UVx$ are connected so $T$ is an infinite complex. This is not a feature of a specific choice of generators, but holds up to isomorphism.
\begin{lemma}\label{lemma:T is connected}
There is no change of basis disconnecting $T$.
\end{lemma}
\begin{proof}
This follows from the uniqueness part of Theorem \ref{thm:classification of algebraic complexes up to iso}.
\end{proof}

\subsection{Example \texorpdfstring{$D$}{D}}
In the previous subsection, we gave an example of an infinite algebraic complex $T$ that cannot be disconnected under any change of basis. We now extend $T$ to a symmetric complex $D$ with the same property. Let $D$ be a free $\F[U,V]$-module with basis consisting of 21 generators. Instead of writing down the endomorphism $\partial: D \to D$ explicitly, we draw $D$ in Figure \ref{fig:David's infinite example} where the unlabelled dots represent different generators. It is easy to verify that $\partial^2=0$.
\begin{figure}[t]
    \centering
    \begin{tikzpicture}[scale=0.9]

    \draw[step=1.0,gray,thin] (0.5,-2.5) grid (10.5,7.5);

    \coordinate (xup) at (2,2);
    \coordinate (xdown) at (1,1);
    \coordinate (yup) at (5, -1);
    \coordinate (ydown) at (4, -2);
    
    \coordinate (x_0) at (6,3);
    \coordinate (x_1) at (4,3);
    \coordinate (x_2) at (4,6);
    \coordinate (x_3) at (2,6);
    \coordinate (x_4) at (2,4);
    \coordinate (xh)  at (7,7);

    \coordinate (y_1) at (6,1);
    \coordinate (y_2) at (9,1);
    \coordinate (y_3) at (9,-1);
    \coordinate (y_4) at (7,-1);
    \coordinate (yh) at (10, 4);

    \draw[red, very thick] (xdown) -- (1, 5) -- (3, 5) -- (3, 7) -- (xh) -- (7,2) -- (xup);
    \draw[red, very thick] (ydown) -- (8,-2) -- (8, 0) -- (10,0) -- (yh) -- (5,4) -- (yup);

    \draw[blue, very thick] (y_4) -- (y_3) -- (y_2) -- (y_1) -- (x_0) -- (x_1) -- (x_2) -- (x_3) -- (x_4);
    \draw[black, very thick] (1,5) -- (3,7);
    \draw[black, very thick] (8,-2) -- (10,0);
    \draw[black, very thick] (xdown) -- (x_4) -- (x_2) -- (xh);
    \draw[black, very thick] (xup) -- (x_0) -- (xh);

    \draw[black, very thick] (ydown) -- (y_4) -- (y_2) -- (yh);
    \draw[black, very thick] (yup) -- (x_0) -- (yh);

    \draw[black, very thick] (4,3) -- (5,4) -- (xup);
    \draw[black, very thick] (6,1) -- (7,2) -- (yup);

    \filldraw[black] (xup) circle (2pt) node[anchor=north west]{$x$};
    \filldraw[black] (xdown) circle (2pt) node[anchor=north west]{$x$};
    \filldraw[black] (yup) circle (2pt) node[anchor=north west]{$x^\vee$};
    \filldraw[black] (ydown) circle (2pt) node[anchor=north west]{$x^\vee$};
    
    \filldraw[black] (1,5) circle (2pt);
    \filldraw[black] (3,5) circle (2pt);
    \filldraw[black] (3,7) circle (2pt);
    \filldraw[black] (xh) circle (2pt);
    \filldraw[black] (7,2) circle (2pt);

    \filldraw[black] (8,-2) circle (2pt);
    \filldraw[black] (8,0) circle (2pt);
    \filldraw[black] (10,0) circle (2pt);
    \filldraw[black] (yh) circle (2pt);
    \filldraw[black] (5,4) circle (2pt);
    
    \filldraw[black] (x_0) circle (2pt);
    \filldraw[black] (x_1) circle (2pt);
    \filldraw[black] (x_2) circle (2pt);
    \filldraw[black] (x_3) circle (2pt);
    \filldraw[black] (x_4) circle (2pt) node[anchor = north west]{$b$};
    \filldraw[black] (y_1) circle (2pt);
    \filldraw[black] (y_2) circle (2pt);
    \filldraw[black] (y_3) circle (2pt);
    \filldraw[black] (y_4) circle (2pt) node[anchor = south east]{$b^\vee$};
    \end{tikzpicture}
    \caption{The diagram illustrates a knot Floer-like complex $D$; a symmetric chain complex over $\F[U,V]$ with the homology of a knot in $S^3$. The main reason for our interest in $D$ is that it is not chain homotopy equivalent to $C \otimes_{\F} \F[U,V]$ for any finitely generated chain complex $C$ over $\F$. As a chain complex over the ring $\frac{\F[U,V]}{(UV)}$, $D$ splits as a direct sum of the standard complex (drawn in blue) and two local systems (infinite pieces, drawn in red).}
    \label{fig:David's infinite example}
\end{figure}
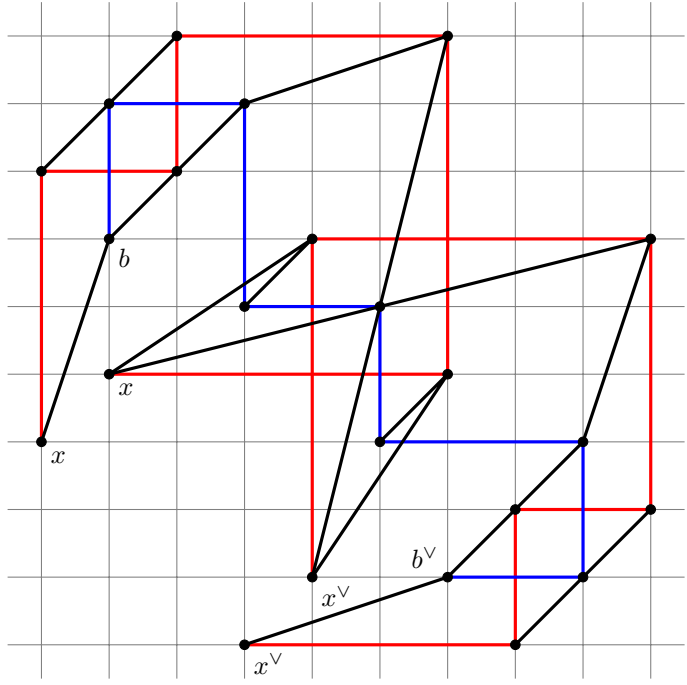
Relative $\Z\oplus\Z$ gradings of all basis elements are uniquely determined by the requirement that $\partial$ have bidegree $(-1,-1)$ and they can be made absolute by fixing $\gr(b)=(10,0)$. Observe that this makes $D$ into a chain complex with the homology of a knot in $S^3$ because the $U$-tower is generated by $b$ with $\gr_V(b)=0$ and the $V$-tower is generated by $b^\vee$ with $\gr_U(b^\vee)=0$. The complex is clearly symmetric. In the same vein as Example $T$, Example $D$ is also infinite.
\begin{lemma}\label{lemma:D is infinite up to iso}
    $D$ is not isomorphic to any complex that can be disconnected.
\end{lemma}
\begin{proof}
As in Example $T$, this follows from uniqueness of the decomposition in Theorem \ref{thm:classification of algebraic complexes up to iso}.   
\end{proof}
We are now ready to prove Theorem \ref{thm:example D}.
\begin{proof}[Proof of Theorem \ref{thm:example D}]
Let $D'$ be a chain complex that can be disconnected and assume for the contradiction that $D \simeq D'$. By Theorem \ref{thm:classification of algebraic complexes up to iso} we have $D' \cong D \oplus Z$ where $Z$ is a direct sum of complexes of the form $x \mapsto y$. So there is a basis of $D$ that disconnects it. However, this is a contradiction with Lemma \ref{lemma:D is infinite up to iso}.
\end{proof}
This is perhaps a very surprising example since none of the knot Floer complexes is known to have an infinite component. We do not know whether the complex $D$ appears as a knot Floer complex $\CFL_{\F[U,V]}(K)$ for some knot $K \subset S^3$. However, it is a $\Z\oplus\Z$ graded symmetric chain complex over $\F[U,V]$ with the homology of a knot in $S^3$, so there are no existing algebraic obstructions that could be used to rule out this possibility.

\subsection{Example \texorpdfstring{$P$}{P}} In this subsection, we describe an example of a knot Floer-like complex that does not admit a simplified basis, not even up to chain homotopy equivalence. 

First note that the proposed example of such a complex from \cite[Figure 3]{hom2015infinite} actually does admit a simplified basis. See Figure \ref{fig:Hom's example} for a change of basis that simplifies the complex. This is because
$\begin{pmatrix}
1&1\\
0&1
\end{pmatrix}$
and
$\begin{pmatrix}
0&1\\
1&0
\end{pmatrix}$
are conjugate in $\GL 2 {\F}$. However, the conjugacy problem is the only issue. We prove the following theorem which states that whenever a local system does not correspond to a permutation matrix, it cannot be simplified.
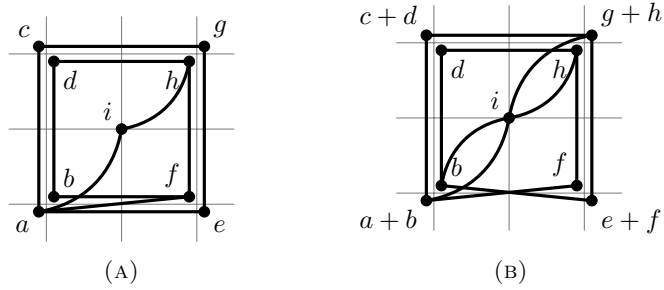
\begin{figure}[t]
    \begin{subfigure}[b]{0.40\textwidth}
        \centering
        \begin{tikzpicture}[scale=1.0]
            \draw[step=1.0,gray,thin] (0.5,0.5) grid (3.5,3.5);
            \def\a{0.1}
            \draw[black, very thick] (1+\a, 1+\a) -- (1+\a, 3-\a) -- (3-\a, 3-\a) -- (3-\a,1+\a) -- (1+\a,1+\a);
            \draw[black, very thick] (1-\a, 1-\a) -- (1-\a, 3+\a) -- (3+\a, 3+\a) -- (3+\a,1-\a) -- (1-\a,1-\a);
            \draw[black, very thick] (3-\a, 1+\a) -- (1-\a, 1-\a);
            
            \draw[black, very thick] (3-\a,3-\a) to[out=-100, in=10] (2,2);
            \draw[black, very thick] (2,2) to[out=-100, in=10] (1-\a, 1-\a);
            
            \filldraw[black] (1+\a,1+\a) circle (2pt) node[anchor=south west]{$b$};
            \filldraw[black] (1+\a,3-\a) circle (2pt) node[anchor=north west]{$d$};
            \filldraw[black] (3-\a,1+\a) circle (2pt) node[anchor=south east]{$f$};
            \filldraw[black] (3-\a,3-\a) circle (2pt) node[anchor=north east]{$h$};
            \filldraw[black] (1-\a,1-\a) circle (2pt) node[anchor=north east]{$a$};
            \filldraw[black] (1-\a,3+\a) circle (2pt) node[anchor=south east]{$c$};
            \filldraw[black] (3+\a,1-\a) circle (2pt) node[anchor=north west]{$e$};
            \filldraw[black] (3+\a,3+\a) circle (2pt) node[anchor=south west]{$g$};
            \filldraw[black] (2,2) circle (2pt) node[anchor=south east]{$i$};
        \end{tikzpicture}
        \caption{}
        \label{fig:Hom's example with the incorrect basis}
    \end{subfigure}
    \begin{subfigure}[b]{0.40\textwidth}
        \centering
        \begin{tikzpicture}[scale=1.0]
            \draw[step=1.0,gray,thin] (0.5,0.5) grid (3.5,3.5);
            \def\a{0.1}
            \draw[black, very thick] (1+\a, 1+\a) -- (1+\a, 3-\a) -- (3-\a, 3-\a) -- (3-\a,1+\a);
            \draw[black, very thick] (1-\a, 1-\a) -- (1-\a, 3+\a) -- (3+\a, 3+\a) -- (3+\a,1-\a) -- (1+\a, 1+\a);
            \draw[black, very thick] (3-\a, 1+\a) -- (1-\a, 1-\a);

            \draw[black, very thick] (2,2) to[out=-100, in=10] (1-\a, 1-\a);
            \draw[black, very thick] (3-\a,3-\a) to[out=-100, in=10] (2,2);

            \draw[black, very thick] (2,2) to[out=190, in=80] (1+\a, 1+\a);
            \draw[black, very thick] (3+\a,3+\a) to[out=190, in=80] (2,2);
            
            \filldraw[black] (1+\a,1+\a) circle (2pt) node[anchor=south west]{$b$};
            \filldraw[black] (1+\a,3-\a) circle (2pt) node[anchor=north west]{$d$};
            \filldraw[black] (3-\a,1+\a) circle (2pt) node[anchor=south east]{$f$};
            \filldraw[black] (3-\a,3-\a) circle (2pt) node[anchor=north east]{$h$};
            \filldraw[black] (1-\a,1-\a) circle (2pt) node[anchor=north east]{$a+b$};
            \filldraw[black] (1-\a,3+\a) circle (2pt) node[anchor=south east]{$c+d$};
            \filldraw[black] (3+\a,1-\a) circle (2pt) node[anchor=north west]{$e+f$};
            \filldraw[black] (3+\a,3+\a) circle (2pt) node[anchor=south west]{$g+h$};
            \filldraw[black] (2,2) circle (2pt) node[anchor=south east]{$i$};
        \end{tikzpicture}
        \caption{}
        \label{fig:Hom's example with the correct basis}
    \end{subfigure}
    \caption{Chain complex from \cite[Figure 3]{hom2015infinite} depicted in ($\textsc{a}$) admits a simplified basis. The same chain complex with respect to this simplified basis is drawn in ($\textsc{b}$).}
    \label{fig:Hom's example}
\end{figure}
\begin{theorem}\label{thm:local systems don't have simplified bases}
    Let $L$ be a local system $(s, w, [A])$ where the conjugacy class $[A]$ does not contain any permutation matrix. Then $L$ is not chain homotopy equivalent to any chain complex admitting a simplified basis.
\end{theorem}
\begin{proof}
    Let $L'$ be a local system that admits a simplified basis and assume for the contradiction that $L \simeq L'$. By Theorem \ref{thm:classification of algebraic complexes up to iso} we have $L' \cong L \oplus Z$ where $Z$ is a direct sum of zero complexes. Since $L'$ has a simplified basis, so does $L \oplus Z$ and such a basis gives a simplified basis for $L$.

    So $L$ itself must also admit a simplified basis. However, $[A]$ does not contain any permutation matrix, so no matrix in $[A]$ is a product of permutation matrices. This means that with respect to any basis, some isomorphism in the simplified decomposition of $L$ will contain more than $w$ arrows. Thus $L$ does not admit a simplified basis and we have reached a contradiction.
\end{proof}
What remains is to construct an explicit example of a knot Floer-like complex that cannot be simplified. To do that, we will use the remaining conjugacy class $\left\{ 
\begin{pmatrix}
1&1\\
1&0
\end{pmatrix},
\begin{pmatrix}
0&1\\
1&1
\end{pmatrix}
\right\}$
in $\GL 2 \F$. Instead of describing the basis and the differential $\partial$ of the chain complex algebraically, we construct Example $P$ in Figure \ref{fig:David's basis example}. All dots represent different generators and the arrows represent the differential. Some of the arrows are drawn blue and green for clarity reasons only. We have already used this color scheme in \cite[Lemma 3.10]{popovic2023algebraic} as a part of a more general framework for extending chain complexes over $\mathcal{R}_1$ to chain complexes over $\F[U,V]$. Indeed, one can easily manually verify that $P$ satisfies $\partial^2=0$. The relative bigradings of all basis elements are uniquely determined by the requirement that $\partial$ have bidegree $(-1,-1)$ and they can be made absolute by fixing $\gr(a)=\gr(b)=(0,0)$. This makes $P$ into a $\Z\oplus\Z$ graded chain complex over $\F[U,V]$ with the homology of a knot in $S^3$. It is symmetric; this can be seen through the change of basis that conjugates one of the matrices in the above conjugacy class into the other.

\begin{figure}
    \centering
    \begin{tikzpicture}[scale=1.0]
            \draw[step=1.0,gray,thin] (0.5,0.5) grid (4.5,4.5);
            \def\a{0.1}
            \def\b{1}
            \def\c{3}
            \draw[very thick] (1+\a, 1+\a) -- (1+\a, 3-\a) -- (3-\a, 3-\a) -- (3-\a,1+\a);
            \draw[very thick] (1-\a, 1-\a) -- (1-\a, 3+\a) -- (3+\a, 3+\a) -- (3+\a,1-\a) -- (1+\a, 1+\a);
            \draw[very thick] (3-\a, 1+\a) -- (1-\a, 1-\a);
            \draw[very thick] (1-\a,1-\a) -- (3+\a,1-\a);

            \draw[cyan, very thick] (2-\a,2-\a) to[out=-100, in=10] (1-\a, 1-\a);
            \draw[cyan, very thick] (2+\a,2+\a) to[out=190, in=80] (1+\a,1+\a);
            \draw[cyan, very thick] (2-\a,4+\a) -- (1-\a, 3+\a);
            \draw[cyan, very thick] (2+\a,4-\a) -- (1+\a, 3-\a);
            \draw[cyan, very thick] (3-\a,1+\a) -- (4-\a, 2+\a);
            \draw[cyan, very thick] (3+\a,1-\a) -- (4+\a, 2-\a);
            \draw[cyan, very thick] (4-\a,4-\a) to[out=-100, in=10] (3-\a, 3-\a);
            \draw[cyan, very thick] (4+\a,4+\a) to[out=190, in=80] (3+\a,3+\a);
            \draw[dark green, very thick] (3-\a,3-\a) to[out=-100, in=10] (2-\a, 2-\a);
            \draw[dark green, very thick] (3+\a,3+\a) to[out=190, in=80] (2+\a,2+\a);

            \draw[dark green, very thick] (2+\a,2+\a)--(3-\a,3-\a);

            \draw[black, very thick] (1+\a+\b, 1+\a+\b) -- (1+\a+\b, 3-\a+\b) -- (3-\a+\b, 3-\a+\b) -- (3-\a+\b,1+\a+\b);
            \draw[black, very thick] (1-\a+\b, 1-\a+\b) -- (1-\a+\b, 3+\a+\b) -- (3+\a+\b, 3+\a+\b) -- (3+\a+\b,1-\a+\b) -- (1+\a+\b, 1+\a+\b);
            \draw[black, very thick] (3-\a+\b, 1+\a+\b) -- (1-\a+\b, 1-\a+\b);
            \draw[very thick] (2-\a,2-\a) -- (4+\a,2-\a);
            
            \filldraw[black] (1+\a,1+\a) circle (2pt) node[anchor=south west]{};
            \filldraw[black] (1+\a,3-\a) circle (2pt) node[anchor=north west]{};
            \filldraw[black] (3-\a,1+\a) circle (2pt) node[anchor=south east]{};
            \filldraw[black] (3-\a,3-\a) circle (2pt) node[anchor=north east]{};
            \filldraw[black] (1-\a,1-\a) circle (2pt) node[anchor=north east]{};
            \filldraw[black] (1-\a,3+\a) circle (2pt) node[anchor=south east]{};
            \filldraw[black] (3+\a,1-\a) circle (2pt) node[anchor=north west]{};
            \filldraw[black] (3+\a,3+\a) circle (2pt) node[anchor=south west]{};
            
            \filldraw[] (1+\a+\b,1+\a+\b) circle (2pt) node[anchor=south west]{};
            \filldraw[] (1+\a+\b,3-\a+\b) circle (2pt) node[anchor=north west]{};
            \filldraw[] (3-\a+\b,1+\a+\b) circle (2pt) node[anchor=south east]{};
            \filldraw[] (3-\a+\b,3-\a+\b) circle (2pt) node[anchor=north east]{};
            \filldraw[] (1-\a+\b,1-\a+\b) circle (2pt) node[anchor=north east]{};
            \filldraw[] (1-\a+\b,3+\a+\b) circle (2pt) node[anchor=south east]{};
            \filldraw[] (3+\a+\b,1-\a+\b) circle (2pt) node[anchor=north west]{};
            \filldraw[] (3+\a+\b,3+\a+\b) circle (2pt) node[anchor=south east]{$b$};
            \filldraw[] (4+3*\a,4+3*\a) circle (2pt) node[anchor=south west]{$a$};
        \end{tikzpicture}
    \caption{The diagram illustrates a knot Floer-like complex $P$; a symmetric chain complex over $\F[U,V]$ with the homology of a knot in $S^3$. The main reason for our interest in $P$ is that it is not chain homotopy equivalent to any complex admitting a simplified basis.}
    \label{fig:David's basis example}
\end{figure}
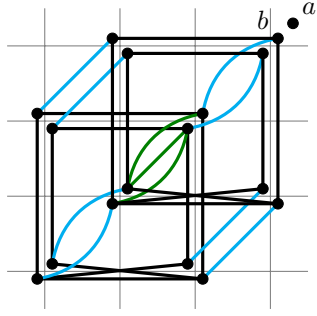

\vspace{1em}

We are now ready to conclude Theorem \ref{thm:example P}.
\begin{proof}[Proof of Theorem \ref{thm:example P}]
A mod $UV$ reduction of $P$ splits into a direct sum of two local systems of the same shape and $\mathcal{R}_1$. By Theorem \ref{thm:local systems don't have simplified bases}, neither of the local systems is chain homotopy equivalent to a chain complex admitting a simplified basis. Hence, neither is $P$.
\end{proof}

\subsection{Example \texorpdfstring{$E$}{E}}
Since the main theorems of this paper work over arbitrary fields, it might be interesting to wonder whether the splitting of a chain complex into local systems is the same over any field. This is not the case; our final example is a chain complex $E$ over $\Z[U,V]$ that reduces to different direct sums of local systems when it is reduced over $\F_2$ and $\F_3$. See Figure \ref{fig:example E}. 
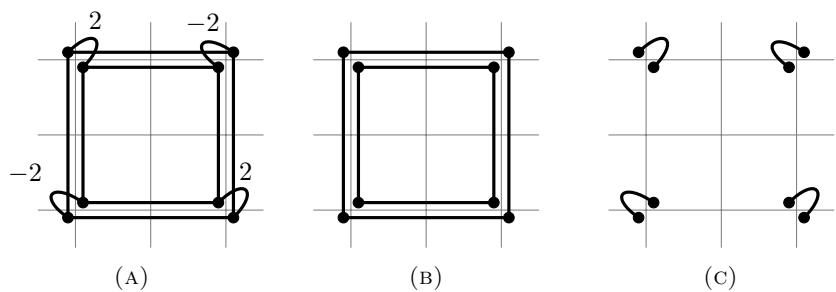
\begin{figure}[t]
    \begin{subfigure}[b]{0.30\textwidth}
        \centering
        \begin{tikzpicture}[scale=1.0]
            \draw[step=1.0,gray,thin] (0.5,0.5) grid (3.5,3.5);
            \def\a{0.1}
            \draw[black, very thick] (1+\a, 1+\a) -- (1+\a, 3-\a) -- (3-\a, 3-\a) -- (3-\a,1+\a) -- (1+\a,1+\a);
            \draw[black, very thick] (1-\a, 1-\a) -- (1-\a, 3+\a) -- (3+\a, 3+\a) -- (3+\a,1-\a) -- (1-\a,1-\a);

            \draw[very thick] (1-\a, 1-\a) to[out=145, in=145, distance=0.5cm] node[anchor = south east]{$-2$} (1+\a, 1+\a) {};
            \draw[very thick] (3+\a, 1-\a) to[out=45, in=45, distance=0.5cm] node[above]{$2$} (3-\a, 1+\a) {};
            \draw[very thick] (1+\a, 3-\a) to[out=45, in=45, distance=0.5cm] node[above]{$2$} (1-\a, 3+\a){};
            \draw[very thick] (3-\a, 3-\a) to[out=145, in=145, distance=0.5cm] node[above]{$-2$} (3+\a, 3+\a) {};
            
            \filldraw[] (1+\a,1+\a) circle (2pt) node[anchor=south west]{};
            \filldraw[] (1+\a,3-\a) circle (2pt) node[anchor=north west]{};
            \filldraw[] (3-\a,1+\a) circle (2pt) node[anchor=south east]{};
            \filldraw[] (3-\a,3-\a) circle (2pt) node[anchor=north east]{};
            \filldraw[] (1-\a,1-\a) circle (2pt) node[anchor=north east]{};
            \filldraw[] (1-\a,3+\a) circle (2pt) node[anchor=south east]{};
            \filldraw[] (3+\a,1-\a) circle (2pt) node[anchor=north west]{};
            \filldraw[] (3+\a,3+\a) circle (2pt) node[anchor=south west]{};
        \end{tikzpicture}
        \caption{}
    \end{subfigure}
    \begin{subfigure}[b]{0.30\textwidth}
        \centering
            \begin{tikzpicture}[scale=1.0]
            \draw[step=1.0,gray,thin] (0.5,0.5) grid (3.5,3.5);
            \def\a{0.1}
            \draw[black, very thick] (1+\a, 1+\a) -- (1+\a, 3-\a) -- (3-\a, 3-\a) -- (3-\a,1+\a) -- (1+\a,1+\a);
            \draw[black, very thick] (1-\a, 1-\a) -- (1-\a, 3+\a) -- (3+\a, 3+\a) -- (3+\a,1-\a) -- (1-\a,1-\a);
            
            \filldraw[] (1+\a,1+\a) circle (2pt) node[anchor=south west]{};
            \filldraw[] (1+\a,3-\a) circle (2pt) node[anchor=north west]{};
            \filldraw[] (3-\a,1+\a) circle (2pt) node[anchor=south east]{};
            \filldraw[] (3-\a,3-\a) circle (2pt) node[anchor=north east]{};
            \filldraw[] (1-\a,1-\a) circle (2pt) node[anchor=north east]{};
            \filldraw[] (1-\a,3+\a) circle (2pt) node[anchor=south east]{};
            \filldraw[] (3+\a,1-\a) circle (2pt) node[anchor=north west]{};
            \filldraw[] (3+\a,3+\a) circle (2pt) node[anchor=south west]{};
        \end{tikzpicture}
        \caption{}
    \end{subfigure}
    \begin{subfigure}[b]{0.30\textwidth}
        \centering
        \begin{tikzpicture}[scale=1.0]
            \draw[step=1.0,gray,thin] (0.5,0.5) grid (3.5,3.5);
            \def\a{0.1}

            \draw[very thick] (1-\a, 1-\a) to[out=145, in=145, distance=0.5cm] (1+\a, 1+\a) {};
            \draw[very thick] (3+\a, 1-\a) to[out=45, in=45, distance=0.5cm] (3-\a, 1+\a) {};
            \draw[very thick] (1+\a, 3-\a) to[out=45, in=45, distance=0.5cm] (1-\a, 3+\a){};
            \draw[very thick] (3-\a, 3-\a) to[out=145, in=145, distance=0.5cm](3+\a, 3+\a) {};
            
            \filldraw[] (1+\a,1+\a) circle (2pt) node[anchor=south west]{};
            \filldraw[] (1+\a,3-\a) circle (2pt) node[anchor=north west]{};
            \filldraw[] (3-\a,1+\a) circle (2pt) node[anchor=south east]{};
            \filldraw[] (3-\a,3-\a) circle (2pt) node[anchor=north east]{};
            \filldraw[] (1-\a,1-\a) circle (2pt) node[anchor=north east]{};
            \filldraw[] (1-\a,3+\a) circle (2pt) node[anchor=south east]{};
            \filldraw[] (3+\a,1-\a) circle (2pt) node[anchor=north west]{};
            \filldraw[] (3+\a,3+\a) circle (2pt) node[anchor=south west]{};
        \end{tikzpicture}
        \caption{}
    \end{subfigure}
    \caption{A chain complex $E$ over $\Z[U,V]$ is shown in $(\textsc{a})$; the adopted convention is that the arrows of length $0$ go from the outer to the inner square. Reducing $E$ modulo $2$ gives a chain complex over $\F_2[U,V]$ depicted in $(\textsc{b})$ and reducing $E$ modulo $3$ gives a chain complex over $\F_3[U,V]$ isomorphic to the one depicted in $(\textsc{c})$. This shows that a chain complex over $\Z[U,V]$ can split into different local systems when reduced over different fields.}
    \label{fig:example E}
\end{figure}

\bibliography{mybib.bib}

\begin{thebibliography}{DHST21}

\bibitem[Art11]{artin2011algebra}
Michael Artin.
\newblock {\em Algebra - Second Edition}.
\newblock Pearson Education, 2011.

\bibitem[Bin23]{binns2023cfk}
Fraser Binns.
\newblock The cfk infty type of almost l-space knots.
\newblock {\em arXiv preprint arXiv:2303.07249}, 2023.

\bibitem[DHST21]{dai2021more}
Irving Dai, Jennifer Hom, Matthew Stoffregen, and Linh Truong.
\newblock More concordance homomorphisms from knot floer homology.
\newblock {\em Geometry \& Topology}, 25(1):275--338, 2021.

\bibitem[Ghi08]{ghiggini2008knot}
Paolo Ghiggini.
\newblock Knot floer homology detects genus-one fibred knots.
\newblock {\em American journal of mathematics}, 130(5):1151--1169, 2008.

\bibitem[Han23]{hanselman2023knot}
Jonathan Hanselman.
\newblock Knot floer homology as immersed curves.
\newblock {\em arXiv preprint arXiv:2305.16271}, 2023.

\bibitem[HKP20]{hom2020ribbon}
Jennifer Hom, Sungkyung Kang, and JungHwan Park.
\newblock Ribbon knots, cabling, and handle decompositions.
\newblock {\em arXiv preprint arXiv:2003.02832}, 2020.

\bibitem[HM17]{hendricks2017involutive}
Kristen Hendricks and Ciprian Manolescu.
\newblock Involutive heegaard floer homology.
\newblock {\em Duke Mathematical Journal}, 166(7):1211--1299, 2017.

\bibitem[Hom14]{hom2014epsilon}
Jennifer Hom.
\newblock Bordered heegaard floer homology and the tau-invariant of cable
  knots.
\newblock {\em Journal of Topology}, 7(2):287--326, 2014.

\bibitem[Hom15]{hom2015infinite}
Jennifer Hom.
\newblock An infinite-rank summand of topologically slice knots.
\newblock {\em Geometry \& Topology}, 19(2):1063--1110, 2015.

\bibitem[HRW16]{hanselman2016bordered}
Jonathan Hanselman, Jacob Rasmussen, and Liam Watson.
\newblock Bordered floer homology for manifolds with torus boundary via
  immersed curves.
\newblock {\em arXiv preprint arXiv:1604.03466}, 2016.

\bibitem[HRW22]{hanselman2022heegaard}
Jonathan Hanselman, Jacob Rasmussen, and Liam Watson.
\newblock Heegaard floer homology for manifolds with torus boundary: properties
  and examples.
\newblock {\em Proceedings of the London Mathematical Society},
  125(4):879--967, 2022.

\bibitem[HW14]{hom2014nu+}
Jennifer Hom and Zhongtao Wu.
\newblock Four-ball genus bounds and a refinement of the ozsv{\'a}th-szab{\'o}
  tau-invariant.
\newblock {\em arXiv preprint arXiv:1401.1565}, 2014.

\bibitem[JMZ20]{juhasz2020knot}
Andr{\'a}s Juh{\'a}sz, Maggie Miller, and Ian Zemke.
\newblock Knot cobordisms, bridge index, and torsion in floer homology.
\newblock {\em Journal of Topology}, 13(4):1701--1724, 2020.

\bibitem[Juh08]{juhasz2008floer}
Andr{\'a}s Juh{\'a}sz.
\newblock Floer homology and surface decompositions.
\newblock {\em Geometry \& Topology}, 12(1):299--350, 2008.

\bibitem[Kra15]{krause2015krull}
Henning Krause.
\newblock Krull--schmidt categories and projective covers.
\newblock {\em Expositiones Mathematicae}, 33(4):535--549, 2015.

\bibitem[LOT18]{lipshitz2018bordered}
Robert Lipshitz, Peter Ozsv{\'a}th, and Dylan Thurston.
\newblock {\em Bordered Heegaard Floer homology}, volume 254.
\newblock American Mathematical Society, 2018.

\bibitem[Ni07]{ni2007knot}
Yi~Ni.
\newblock Knot floer homology detects fibred knots.
\newblock {\em Inventiones mathematicae}, 170(3):577--608, 2007.

\bibitem[OS03]{ozsvath2003tau}
Peter Ozsv{\'a}th and Zolt{\'a}n Szab{\'o}.
\newblock Knot floer homology and the four-ball genus.
\newblock {\em Geometry \& Topology}, 7(2):615--639, 2003.

\bibitem[OS04a]{OS2004holomorphicDisksAndGenusBounds}
Peter Ozsv{\'a}th and Zolt{\'a}n Szab{\'o}.
\newblock Holomorphic disks and genus bounds.
\newblock {\em Geometry \& Topology}, 8(1):311--334, 2004.

\bibitem[OS04b]{ozsvath2004holomorphicKnots}
Peter Ozsv{\'a}th and Zolt{\'a}n Szab{\'o}.
\newblock Holomorphic disks and knot invariants.
\newblock {\em Advances in Mathematics}, 186(1):58--116, 2004.

\bibitem[OS04c]{ozsvath2004holomorphicSequel}
Peter Ozsv{\'a}th and Zolt{\'a}n Szab{\'o}.
\newblock Holomorphic disks and three-manifold invariants: properties and
  applications.
\newblock {\em Annals of Mathematics}, pages 1159--1245, 2004.

\bibitem[OS04d]{ozsvath2004holomorphic}
Peter Ozsv{\'a}th and Zolt{\'a}n Szab{\'o}.
\newblock Holomorphic disks and topological invariants for closed
  three-manifolds.
\newblock {\em Annals of Mathematics}, pages 1027--1158, 2004.

\bibitem[OS05]{OS-HFKandUnknottingNumber}
Peter Ozsv{\'a}th and Zolt{\'a}n Szab{\'o}.
\newblock Knots with unknotting number one and heegaard floer homology.
\newblock {\em Topology}, 44(4):705--745, 2005.

\bibitem[OS08a]{ozsvath2008holomorphicLinks}
Peter Ozsv{\'a}th and Zolt{\'a}n Szab{\'o}.
\newblock Holomorphic disks, link invariants and the multi-variable alexander
  polynomial.
\newblock {\em Algebraic \& Geometric Topology}, 8(2):615--692, 2008.

\bibitem[OS08b]{OS-HFKandIntegerSurgeries}
Peter Ozsv{\'a}th and Zolt{\'a}n Szab{\'o}.
\newblock Knot floer homology and integer surgeries.
\newblock {\em Algebraic \& Geometric Topology}, 8(1):101--153, 2008.

\bibitem[OS08c]{OS2008linkThurstonNorm}
Peter Ozsv{\'a}th and Zolt{\'a}n Szab{\'o}.
\newblock Link floer homology and the thurston norm.
\newblock {\em Journal of the American Mathematical Society}, 21(3):671--709,
  2008.

\bibitem[OS10]{OS-HFKandRationalSurgeries}
Peter~S Ozsv{\'a}th and Zolt{\'a}n Szab{\'o}.
\newblock Knot floer homology and rational surgeries.
\newblock {\em Algebraic \& Geometric Topology}, 11(1):1--68, 2010.

\bibitem[OSS17]{OSS2017upsilon}
Peter~S Ozsv{\'a}th, Andr{\'a}s~I Stipsicz, and Zolt{\'a}n Szab{\'o}.
\newblock Concordance homomorphisms from knot floer homology.
\newblock {\em Advances in Mathematics}, 315:366--426, 2017.

\bibitem[Pet13]{petkova2013cables}
Ina Petkova.
\newblock Cables of thin knots and bordered heegaard floer homology.
\newblock {\em Quantum Topology}, 4(4):377--409, 2013.

\bibitem[Pop23]{popovic2023algebraic}
David Popovi\'c.
\newblock Algebraic realizability of knot floer-like complexes.
\newblock {\em arXiv preprint arXiv:2306.04861}, 2023.

\bibitem[Ras03]{rasmussen2003floer}
Jacob~Andrew Rasmussen.
\newblock {\em Floer homology and knot complements}.
\newblock Harvard University, 2003.

\bibitem[Zem19]{zemke2019connected}
Ian Zemke.
\newblock Connected sums and involutive knot floer homology.
\newblock {\em Proceedings of the London Mathematical Society},
  119(1):214--265, 2019.

\end{thebibliography}
\bibliographystyle{alpha}

\end{document}